\documentclass[11pt]{amsart}
\usepackage{xcolor,amssymb,amsmath,hyperref,multirow,epsfig,float,bbm}
\usepackage[headings]{fullpage}   %  Makes the text very wide

\usepackage[symbol]{footmisc}
\definecolor{navyblue}{RGB}{0, 51, 102}
\definecolor{darkgreen}{RGB}{0, 100, 0}
\definecolor{maroon}{RGB}{128, 0, 0}

\hypersetup{
    colorlinks=true,      % Colors text instead of drawing boxes around links
    linkcolor=navyblue,   % Color of internal links (sections, equations, figures)
    citecolor=darkgreen,  % Color of citation links
    urlcolor=maroon       % Color of external URLs
}

\newtheorem{theorem}{Theorem}[section]
\newtheorem{example}[theorem]{Example}
\newtheorem{corollary}[theorem]{Corollary}

\newtheorem{lemma}[theorem]{Lemma}
\newtheorem{proposition}[theorem]{Proposition}
\theoremstyle{remark}

\numberwithin{equation}{section}

\renewcommand{\subset}{\subseteq}

\newcommand{\A}{\mathbf{A}}
\newcommand{\Af}{\mathbf{A}_{\fin}}

\newcommand{\bs}{\backslash}

\newcommand{\comment}[1]{}
\newcommand{\cInd}{\operatorname{c-Ind}}
\newcommand{\C}{\mathbf{C}}

\newcommand{\ds}{\displaystyle}
\newcommand{\e}{\varepsilon}

\newcommand{\f}{f}%{{f^{\n}_{\k}}}
\newcommand{\F}{\mathbf{F}}
\newcommand{\fin}{\operatorname{fin}}

\newcommand{\fn}{f^{\n}}
\newcommand{\fnp}{f_\ell^{\n}}
\newcommand{\fp}{f_\p}
\newcommand{\fv}{f^{\mfn}}
\newcommand{\g}{\gamma}

\newcommand{\Gal}{\operatorname{Gal}}
\newcommand{\GL}{\operatorname{GL}}

\newcommand{\Ind}{\operatorname{Ind}}

\renewcommand{\k}{\mathbbm{k}}

\newcommand{\mat}[4]{\begin{pmatrix} {#1} & {#2} \\ {#3} & {#4}
  \end{pmatrix}}
\newcommand{\mfn}{\mathfrak{n}_v}
\newcommand{\meas}{\operatorname{meas}}
\renewcommand{\mod}{\text{ mod }}

\newcommand{\n}{\mathtt{n}}
\newcommand{\new}{\mathrm{new}}

\newcommand{\N}{\mathbb{N}}

\newcommand{\ord}{\operatorname{ord}}
\renewcommand{\O}{\mathcal{O}}

\newcommand{\ol}{\overline}
\newcommand{\olG}{\overline{G}}

\newcommand{\p}{\mathfrak{p}}

\newcommand{\row}[2]{\left({#1}\,\,\,{#2}\right)}

\newcommand{\PGL}{\operatorname{PGL}}
\newcommand{\Q}{\mathbf{Q}}

\newcommand{\R}{\mathbf{R}}
\renewcommand{\Re}{\operatorname{Re}}

\newcommand{\sg}[1]{\left<{#1}\right>}
\newcommand{\sgn}{\operatorname{sgn}}
\newcommand{\smat}[4]{\bigl(\begin{smallmatrix}{#1}&{#2}\\{#3}&{#4}\end{smallmatrix}\bigr )}

\newcommand{\SL}{\operatorname{SL}}
\newcommand{\SO}{\operatorname{SO}}

\newcommand{\Span}{\operatorname{Span}}
\newcommand{\Supp}{\operatorname{Supp}}
\newcommand{\T}{\mathbb{T}}

\newcommand{\tr}{\operatorname{tr}}

\newcommand{\w}{\omega}

\newcommand{\Z}{\mathbf{Z}}
\newcommand{\Zhat}{\widehat{\Z}}

\newcommand{\zpvector}{\begin{pmatrix} \O_F \\ \O_F \end{pmatrix}}
\makeatletter
\def\@tocline#1#2#3#4#5#6#7{\relax
  \ifnum #1>\c@tocdepth % then omit
  \else
    \par \addpenalty\@secpenalty\addvspace{#2}%
    \begingroup \hyphenpenalty\@M
    \@ifempty{#4}{%
      \@tempdima\csname r@tocindent\number#1\endcsname\relax
    }{%
      \@tempdima#4\relax
    }%
    \parindent\z@ \leftskip#3\relax \advance\leftskip\@tempdima\relax
    \rightskip\@pnumwidth plus4em \parfillskip-\@pnumwidth
    #5\leavevmode\hskip-\@tempdima
      \ifcase #1
       \or\or \hskip 1em \or \hskip 2em \else \hskip 3em \fi%
      #6\nobreak\relax
    \hfill\hbox to\@pnumwidth{\@tocpagenum{#7}}\par
    \nobreak
    \endgroup
  \fi}
\makeatother

\begin{document}
\title{Counting locally supercuspidal newforms}
\author{Andrew Knightly}
\address{Department of Mathematics \& Statistics\\University of Maine
\\Neville Hall\\ Orono, ME  04469-5752, USA }

\begin{abstract} 
The trace formula is a versatile tool for computing sums of spectral data across families of
automorphic forms.
Using specialized test functions, one can treat small families with refined spectral properties.
This has proven fruitful in analytic applications.
We detail such methodology here, with the aim of counting newforms in certain small families.
The result (Theorem \ref{dimST}) is a general formula for the number of holomorphic newforms
 of weight $k$ and level $N$ whose local representation type at each $p|N$ is a fixed
supercuspidal representation $\sigma_p$ of $\GL_2(\Q_p)$.  
This is given in terms of
  local elliptic orbital integrals attached to matrix coefficients of the $\sigma_p$.
We evaluate the formula explicitly in the case where each $\sigma_p$ has conductor $\le p^3$.
The technical heart of the paper is the explicit calculation in \S\ref{pN} of 
elliptic orbital integrals attached to such $\sigma_p$.
 We also compute the traces of Hecke operators on the span of these newforms.
Some applications are given to biases among root numbers of newforms.
\end{abstract}

\maketitle
Corrected version, \today\footnote{Two summations in Theorem \ref{mainST} were improperly
truncated in the published version of this article.
We correct this here, and also take the opportunity to simplify the
orbital integral formulas appearing in Proposition \ref{gunram} (see its Remark 3)
  and Proposition \ref{elld0}.}
%\noindent 
\thispagestyle{empty}
%\hskip 1cm {\em Working draft, please do not distribute}
\tableofcontents

\section{Introduction} \label{intro}

\subsection{Overview}

Modular forms are holomorphic functions on the 
complex upper half-plane $\mathbf H$ that obey a type of symmetry
under the action of $\SL_2(\Z)$ (or a congruence subgroup) on $\mathbf H$ 
  by linear fractional transformations.
They belong to the realm of analysis,
but this symmetry embodies a deep link with number theory and algebra.
Indeed, Langlands' famous Functoriality Conjecture predicts that there is
a precise connection between the algebraic structure of the field $\Q$
of rational numbers (as captured by representations of its absolute Galois group)
and spectral properties of automorphic forms (the latter being
simultaneous eigenfunctions of the Laplace operator and its $p$-adic analogs, the Hecke operators),
\cite{Ge1}.
This connection is expressed as an equality of $L$-functions.

Automorphic forms can be elusive, and for most purposes it is not feasible to study 
  them and their $L$-functions one at a time.  
The trace formula is a technique that provides access to averages of spectral data across families of forms,
where the family is determined by a choice of test function.
For instance, by choosing a test function with certain invariance 
properties, one obtains a sum of Hecke eigenvalues $\lambda_n(h)$ for all eigenforms $h$
  of a given level and weight, i.e., the trace of the Hecke operator $T_n$ on $S_k(N)$ (see, for example,
\cite{KL}).  

The trace formula and its relative cousins have seen widespread use in
analytic number theory, with applications to such problems as estimating moments of $L$-functions
 with consequent subconvexity bounds for a single $L$-function, determining the asymptotic distribution
of the Hecke eigenvalues of a growing family of cusp forms (vertical Sato-Tate laws), 
and finding densities of low-lying zeros of families $L$-functions (Katz-Sarnak philosophy). 
See \cite{Bl} for a recent survey of these and other applications.  

Our aim in the present article is to train the trace formula microscope more narrowly 
through the use of specialized test functions, thereby providing access to 
thinner families in the automorphic spectrum.  This is achieved using the ``simple trace formula", 
variants of which have been in use since the 1970's, \cite[(7.21)]{GJ}.  
Our motivation (described in the next section) is to count
cusp forms in these thin families.  But the explicit and flexible local-to-global techniques detailed here
for $\GL(2)$ can be used in many other applications.

Counterintuitively, by considering smaller families, in some situations one obtains
simpler trace formulas and stronger analytic results.  We mention here some examples that illustrate 
this.  First, Hu, Petrow and Young have recently developed Fourier relative trace formulas for newforms with 
certain prescribed local representation types, \cite{H}, \cite{HPY}.  This is used to estimate 
thin averages of Rankin-Selberg $L$-functions, leading to improved hybrid subconvexity
bounds. 

In a different direction, in 2007 Booker and Str\"ombergsson used the Selberg trace formula to 
provide evidence for Selberg's conjecture that the first Laplace eigenvalue in the cuspidal 
spectrum of $\Gamma\bs \mathbf H$ for a congruence subgroup $\Gamma\subset \SL_2(\Z)$ is
$\ge 1/4$, \cite{BS}.  In verifying the conjecture for $\Gamma=\Gamma_1(N)$ for square-free $N<857$, 
they observed that the trace formula simplifies upon sieving out the contribution 
of oldforms in this case.  They were also able to restrict to the even (or odd) part of the spectrum.
  With Lee in \cite{BLS}, they subsequently extended 
this work to remove the square-free hypothesis on $N$.  However, in this case removing the 
oldform contribution introduces further complication.  
To proceed, they developed a novel method to sieve the 
spectrum down further to twist-minimal newforms, arriving at a simpler formula. 
In both papers, working with a 
thinner family extended the reach of their numerical computations.

A general discussion about the value of isolating small families of automorphic forms is given
in \cite[\S1.5]{PY2}.  In the breakthrough papers \cite{PY1}-\cite{PY2}, Petrow and Young established Weyl-type 
subconvexity bounds for Dirichlet $L$-functions using a family of Maass forms that is 
locally principal series at all finite places.

\subsection{Description of main results}

Given an integer $N=\prod_{p|N}p^{N_p}>1$, let $H_k(N)$ be the set of cuspidal Hecke newforms of level $N$
and weight $k$.  Each $h\in H_k(N)$ corresponds to a cuspidal automorphic representation
$\pi_h$ of $G(\A)=G(\R)\prod'_pG(\Q_p)$ where $G=\PGL_2$. The representation $\pi_h$ factors as a restricted tensor
product 
\[\pi_h\cong \otimes'_{p\le \infty} \pi_{h,p}\]
of infinite-dimensional irreducible admissible representations of the local groups.
We know that $\pi_{h,\infty}=\pi_k$ is the weight $k$ discrete series representation,
 that for each prime $p\nmid N$, $\pi_{h,p}$ is an unramified principal series representation with
Satake parameters determined by the $p$-th Hecke eigenvalue of $h$,
  and that for each $p|N$, $\pi_{h,p}$ is ramified of conductor $p^{N_p}$ (see, for example, \cite{Ge0}).

There is an algorithm, due to Loeffler and Weinstein, to determine the isomorphism class of
  each ramified $\pi_{h,p}$ given $h$, \cite{LW}.  
Here we consider the opposite problem, namely to understand
the cusp forms $h$ with prescribed local ramification behavior.
To this end, 
we define the following spaces of newforms.
For each $p|N$, fix an irreducible 
admissible representation $\sigma_p$ of $\PGL_2(\Q_p)$ of conductor $p^{N_p}$,
and let $\widehat{\sigma}=(\sigma_p)_{p|N}$.  We then let 
 $H_k(\widehat{\sigma})$ be the set of weight $k$ newforms of level $N$
 having the local representation type $\sigma_p$ at each $p|N$:
\[ H_k(\widehat{\sigma})=\{h\in H_k(N)|\, \pi_{h,p}\cong \sigma_p\text{ for all }p|N\}.\]
Defining
\[S_k(\widehat{\sigma})=\Span H_k(\widehat{\sigma}),\qquad S_k^{\new}(N)=\Span H_k(N),\]
we have
\begin{equation}\label{Skdim}
    S_k^{\new}(N)= \bigoplus_{\widehat{\sigma}}S_k(\widehat{\sigma}),
\end{equation}
where $\widehat{\sigma}$ ranges over all tuples as above.

The dimensions of the spaces $S_k^{\new}(N)$ have been computed by Greg Martin in \cite{GMar},
by sieving the well-known dimension formulas for the full spaces $S_k(N)$.
It is an open problem to refine these dimension formulas by computing
$\dim S_k(\widehat{\sigma})= |H_k(\widehat{\sigma})|$ for each tuple $\widehat{\sigma}$.
More generally one can ask for the traces of Hecke operators on $S_k(\widehat{\sigma})$. 
A complete solution to this problem seems well out of reach, but even special cases are of great
interest. For example, such information would enable investigations into the effect of the 
underlying representation type on various statistical properties 
of cusp forms.

In some special cases, asymptotic results about $|H_k(\widehat{\sigma})|$ are known.
When $p$ is a finite prime, the representation $\sigma_p$ of $G(\Q_p)$ is 
either principal series, special, or supercuspidal (\cite[\S9.11]{BH}).
  Only the latter two types are 
square-integrable (assuming unitary central character),
 and these are amenable to study via the trace formula.
Kim, Shin and Templier gave asympotics for automorphic
representations with prescribed supercuspidal local behavior in a quite general 
  setting, \cite{KST}. In the case of $\PGL_2(\Q)$, their work shows that
 if each $\sigma_p$ is supercuspidal,
\begin{equation}\label{KST}
 |H_k(\widehat{\sigma})|\sim \frac{k-1}{12}\prod_{p|N}d_{\sigma_p}
\end{equation}
as $k,N\to\infty$,
where $d_{\sigma_p}$ is the formal degree of $\sigma_p$, suitably normalized.
They use the trace formula, and the main technical input is a bound for the
elliptic orbital integrals attached to supercuspidal matrix coefficients.
In related earlier work, Weinstein gave asymptotics for cusp forms 
with prescribed local inertial types, concluding that the set of types lacking global realization is finite, \cite{W}.
Fixing inertial type is weaker than fixing the local representation, 
but this result includes types which are not square-integrable.
This is discussed further in a recent paper of Dieulefait, Pacetti and Tsaknias, \cite{DPT}.

We remark that in Corollary \ref{kodd} we will show that the asymptotic \eqref{KST}
is in fact an equality when $k\ge 3$ is odd (so in particular the nebentypus is nontrivial) and 
$N$ has a prime factor $p>3$ with $\ord_p(N)$ odd.

When $N$ is square-free, each $\sigma_p$ is necessarily special.
Going beyond asymptotics, Kimball Martin computed
$|H_k(\widehat{\sigma})|$ explicitly in this case, by applying Yamauchi's trace formula
for Atkin-Lehner operators, \cite{KMar}.  As an interesting consequence, he discovered that
there is a bias among newforms of square-free level, favoring root number $+1$:
letting $S_k^\pm(N)$ denote the span of the newforms of root number $\pm 1$, we have
\[\dim S_k^+(N)-\dim S_k^-(N)\ge 0\]
when $N$ is square-free,
with the inequality being strict with finitely many explicit exceptions.
For example, if $N>3$ and $k>2$, 
\begin{equation}\label{biasA}
\dim S_k^+(N)-\dim S_k^-(N)= c_N h(-N),
\end{equation}
where $c_N\in \{\frac12,1,2\}$ is a constant depending on the equivalence class of $N$ modulo 8,
and $h(-N)$ is the class number of $\Q(\sqrt{-N})$. 

In the present paper, we further investigate
the case where each $\sigma_p$ is supercuspidal.  Our first main result is
Theorem \ref{stf2} giving, for such tuples $\widehat{\sigma}$,
 a general formula for the trace of a Hecke operator $T_\n$ on $S_k(\widehat{\sigma})$ as a 
main term plus a finite sum of elliptic orbital integrals $\Phi(\g,f)$.
This theorem is obtained from the adelic $\GL_2$ trace formula using a test function $f$ built using 
supercuspidal matrix coefficients at the ramified places.  
In \S \ref{fact} we show how each global elliptic orbital integral can be factorized 
into a product of local ones, multiplied by a global measure term that is computed in 
Theorem \ref{meas}.  This global measure is the source of the class numbers of quadratic 
  number fields that appear in classical trace formulas.
  The local orbital integrals at primes not dividing the level are evaluated explicitly 
over an arbitrary local field of characteristic $0$ in \S \ref{ghyp}-\ref{gell}.
We have kept these calculations as general as possible in order that they
  may find use in other applications of the trace formula.

Theorem \ref{stf2} thereby reduces explicit evaluation of $\tr(T_\n|S_k(\widehat{\sigma}))$ to 
  the calculation of certain local elliptic orbital integrals 
  at the places dividing the level.
We demonstrate proof of concept in \S\ref{ex}-\ref{pN} by carrying out the latter
in the special case where each $\sigma_p$ has conductor $\le p^3$.
As recalled in \S \ref{bg}, the supercuspidals come in two series: the unramified supercuspidals,
 of conductor $p^{2r}$, and the ramified supercuspidals, of conductor $p^{2r+1}$.
  We thus treat the first ($r=1$) family in each series. 
The result is the following explicit formula for $\tr(T_\n|S_k(\widehat{\sigma}))$
under this restriction.  We allow nontrivial nebentypus, which requires the tuple 
$\widehat{\sigma}$ to satisfy a global central character constraint described in \S\ref{test}.
Of course, when $\dim S_k(\widehat{\sigma})=1$ as sometimes happens when $k$ and $N$ are small,
it provides a direct way to compute the Fourier coefficients of the associated newform.
% EXAMPLES of this: look in S_6(4), S_k(8), S_4(25), S_4(49), S_4(23^2), S_4(11^2), S_4(9), S_4(14^2).

\begin{theorem}\label{mainST}
Let $N=S^2T^3>1$ for $S,T$ relatively prime and square-free, and let $\w'$ be a Dirichlet
character of level $N$ and conductor dividing $ST$.
Let $\widehat{\sigma}=(\sigma_p)_{p|N}$ be a tuple of supercuspidal representations,
with $\sigma_p$ of conductor $p^2$ (resp. $p^3$) if $p|S$ (resp. $p|T$), chosen 
compatibly with $\w'$ as in \S\ref{test}.
For $k>2$ satisfying $\w'(-1)=(-1)^k$, 
  let $S_k(\widehat{\sigma})\subset S_k^{\new}(N,\w')$ be the associated space of newforms.
 Then for $(\n,N)=1$ and $T_\n$ the usual Hecke operator defined in \S\ref{spectral},
\[\tr (T_\n|S_k(\widehat{\sigma}))=\n^{k/2-1}\left[\ol{\w'(\sqrt{\n})}\frac{k-1}{12}
  \prod_{p|S}(p-1)\prod_{p|T}\frac{p^2-1}2 +\frac12\sum_{M|T}\Phi(\mat{}{-\n M}1{})\right.\]
\[\left.
+\sum_{M|T}\sum_{1\le r<\sqrt{4\n/M}}\Phi(\mat0{-\n M}1{rM})
\right]
\]
where $\w'(\sqrt \n)$ is taken to be $0$ if $\n$ is not a perfect square.
Each orbital integral $\Phi(\g)$ as above may be evaluated explicitly using
\begin{equation}\label{phifact}
\Phi(\g)=\frac{2h(E)}{w_E2^{\w(d_E)}}\Phi_\infty(\g)\prod_{p| N}
  \Phi_p(\g)\prod_{\ell|\Delta_\g,\atop\ell\nmid N}\Phi_\ell(\g).
\end{equation}
Here, $\ell$ and $p$ denote prime numbers,
$\Delta_\g$ is the discriminant of the characteristic polynomial of $\g$, 
$E=\Q[\g]$ is an imaginary quadratic 
field with class number $h(E)$, discriminant $d_E$ (with
$\w(d_E)$ distinct prime factors) and $w_E$ roots of unity. 
Given 
\[\g=\mat0{-\n M}1{rM}\]
for $0\le r< \sqrt{4\n/M}$, 
 the factors in \eqref{phifact} are given explicitly as follows.

Taking $\theta_\g=\arctan(\sqrt{|\Delta_\g|}/rM)$ (interpreted as $\pi/2$ if $r=0$),
\[ \Phi_\infty(\g)= -\frac{\sin({(k-1)\theta_\g})}{\sin({\theta_\g})}
\]
(as in Proposition \ref{ellg}).

Suppose $\ell|\Delta_\g$  and $\ell\nmid N$. Then if $\g$ is hyperbolic in $G(\Q_\ell)$, 
  \[\Phi_\ell(\g)=|\Delta_\g|_\ell^{-1/2}\]
(as in Proposition \ref{ellh}).
 If $\g$ is elliptic in $G(\Q_\ell)$, then (as in Proposition \ref{elle} and \eqref{index})
    \[\Phi_\ell(\g)=e_\g(\ell)\sum_{j=0}^{\ord_\ell(b)}\ell^j(1+\frac{2-e_\g(\ell)}\ell\delta_{j>0}),\]
 where $\delta_{j>0}$ is an indicator function, 
  $e_\g(\ell)\in\{1,2\}$ is $2$ if and only if $\ell$ ramifies in $E$,
and $b$ is defined by $\Delta_\g=b^2d_E$ for $d_E$ the discriminant of $E$.

Suppose $p|N$.
  If $\g$ is hyperbolic in $G(\Q_p)$, then $\Phi_p(\g)=0$.
So we will assume that $\g$ is elliptic in $G(\Q_p)$.  
 We consider the three cases $p|M$, $p|\frac TM$, and $p|S$ separately.
If $p|M$, then $\Phi_p(\g)=0$ unless
there exists $y$ such that $y^2\equiv -pt_p/\n M\mod p$, where $t_p$ is the parameter of
the fixed supercuspidal representation $\sigma_p=\sigma_{t_p}^{\zeta_p}$ of conductor $p^3$ (see \S\ref{ssr}).
  In this case,
\[\Phi_p(\mat0{-\n M}1{rM}) =\ol{\zeta_p}\Bigl[e(-\frac{yrM}{p^2})\w_p(y) +\delta(p\neq 2)e(\frac{yrM}{p^2})\w_p(-y)\Bigr]\]
(as in Proposition \ref{Phiell}), where $\zeta_p$ and $\w_p$ are the root number and
central character of $\sigma_p$ respectively, 
  $e(x)=e^{2\pi i x}$ and $\delta$ is an indicator function.

If $p|\frac{T}{M}$, then 
$\Phi_p(\g)=0$ unless the characteristic polynomial $P_\g$ of $\g$ has a nonzero double
root modulo $p$, say
\begin{equation}\label{Pgz2}
P_\g(X)\equiv (X-z)^2\mod p
\end{equation}
for some $z\in(\Z/p\Z)^*$.
  Under this condition, (see also the refinement \eqref{**} below Proposition \ref{gunram})
\[
\Phi_p(\g) =\frac{\ol{\w_p(z)}}p\sum_{n=1}^{\ord_p(\Delta_\g)} \sum_{c\mod p}
\mathcal{N}_\g(c,n)
\sum_{y=1}^{p-1} e(\frac{yc}{zp})
 e(-\frac{t_p}{yzp})^{\delta(n=1)}
\]
where $t_p\in (\Z/p\Z)^*$ is the parameter of $\sigma_p=\sigma_{t_p}^{\zeta_p}$, $\w_p$
is its central character, $e(x)=e^{2\pi i x}$, and 
\[\mathcal{N}_\g(c,n)=\#\{b\mod p^{n+1}|\, P_\g(b)\equiv cp^n\mod p^{n+1}\}.\]

Finally, suppose $p|S$. If \eqref{Pgz2} is satisfied,
then (as in Proposition \ref{elld0}),
\[
\Phi_p(\g)= -2\ol{\w_p(z)}.
\]
%\[
%\Phi_p(\g)= -\ol{\w_p(z)}+\frac{\ol{\w_p(z)}}p\sum_{n=1}^{\ord_p(\Delta_\g)}\left[
%(p-1)\mathcal{N}_\g(0,n)-\sum_{c=1}^{p-1}\mathcal{N}_\g(c,n)\right],
%\]
%for $\mathcal{N}_\g(c,n)$ as above.
On the other hand, if $P_\g$ is irreducible modulo $p$, then
\[
\Phi_p(\g)=-\ol{\nu(\g)}-\ol{\nu^p(\g)}
\]
where $\nu$ is the primitive character of $\F_{p^2}^*$ attached to the fixed 
supercuspidal $\sigma_p$ of conductor $p^2$ (see \S\ref{d0}), $\w_p=\nu|_{\F_p^*}$, 
and we interpret the above to mean $-\ol{\nu(x)}-\ol{\nu^p(x)}$ if $x\in \F_{p^2}^*$ has the
 same minimum polynomial over $\F_p$ as the reduction of $\g$ mod $p$.
\end{theorem}
\noindent{\em Remarks:} 
1. What we call $S_k(N,\w')$ is usually called $S_k(N,{\w'}^{-1})$.  
See the beginning of \S \ref{count} for explanation.
The reason we assume that the conductor of $\w'$ divides $ST$ is that this is necessary for the 
existence of tuples $\widehat{\sigma}$ given the conductor hypotheses, by \cite[Proposition 3.4]{Tu}.
\vskip .1cm

2. The theorem contains various simple conditions under which an orbital integral as 
in \eqref{phifact} vanishes.  These are summarized and established in Proposition \ref{relevant}.
\vskip .1cm

3. Analytic applications often require uniform bounds for the orbital integrals appearing on the
geometric side.  
Such bounds were established in a much more general context by 
Kim, Shin and Templier, \cite[(1.5),(1.6),(1.8)]{KST}.  
Using these, they proved a vertical (fixed $p$) equidistribution 
result for $p$-th Hecke eigenvalues in $S_k(\widehat{\sigma})$ as $N\to\infty$, 
refining the result \cite{Serre} of Serre.
Their paper includes several helpful examples to explain their results in the setting of $\PGL(2)$.
The explicit formulas for local orbital integrals developed in the present paper 
illustrate their bounds; see the remarks after Proposition \ref{gunram}, for example.
\vskip .1cm

5. Although we describe some interesting consequences of Theorem \ref{mainST} in \S\ref{bias} below,
 perhaps the main utility 
 of this article is the methodology leading to the theorem, rather than
this particular trace formula.  Indeed, there are any number of variants that one could pursue
 simply by doing some additional local computations and updating the set of relevant
global $\g$'s on the geometric side:
\begin{itemize}
\item One could capture newforms with prescribed representation type at some places, and,
less restrictively,
  prescribed local conductor at some other places.
  For the latter places, the local elliptic orbital integral calculation 
  is carried out in \cite{thm}.
\vskip .1cm
\item We have excluded the case where $\ord_p(N)=1$ at a prime $p$ for the same reason that we impose
 $k>2$: the matrix coefficients of the local representations in such cases are square-integrable but
 not integrable (\cite[Prop. 14.3]{KL}, \cite{Si2}).  So these functions cannot be used directly in the 
 trace formula.  One could incorporate these representation types either by using 
  pseudo-coefficients (\cite{Ko}, \cite[Example 6.6]{KST}, \cite{P}), 
  or, via the Jacquet-Langlands correspondence, by computing the 
  corresponding local orbital integrals on a quaternion algebra (\cite{KRo}).
\vskip .1cm
\item One could capture Maass newforms with prescribed local behavior by taking the archimedean
  component $f_\infty$ of the test function to be bi-$\SO(2)$-invariant,
  as described, for example, in \cite[\S3-4]{ftf}. 
 In this case, the inclusion of a supercuspidal matrix coefficient at some place $p$ will annihilate
the continuous and residual spectra, but at least two such places would be needed in order to 
annihilate the hyperbolic and unipotent terms on the geometric side of the trace formula, 
as explained in Theorem \ref{stf} below.  Further, in this case $\g$ need no longer 
be elliptic in $G(\R)$ in order to contribute nontrivially, so there are more relevant $\g$'s 
  that would have to be considered.
\vskip .1cm
\item The non-archimedean local calculations in the present paper are all carried out over arbitrary 
  $p$-adic fields, so with some additional global considerations 
 one could work over a number field.
\end{itemize}
\vskip .1cm

The technical heart of the paper is \S \ref{pN}, in which we 
calculate local elliptic orbital 
  integrals attached to the supercuspidal representations of conductor $\le p^3$.
Character values of supercuspidal representations on various groups appear in many places, but
the orbital integral calculations in \S\ref{pN} are new.
Some related calculations were made by Palm in his doctoral thesis \cite[\S9.11]{P}. 
Although there are some errors in that work, the methods have been adapted for
our computations.

In \S \ref{dimex} we illustrate Theorem \ref{mainST} by computing dimension
formulas and some examples of $\tr(T_\n|S_k(\widehat{\sigma}))$ for $\n>1$.

\subsection{Dimension formulas and root number bias}\label{bias}

Upon taking $\n=1$ in Theorem \ref{stf2}, we obtain a general formula 
for $\dim S_k(\widehat{\sigma})$, given in Theorem \ref{dimST}.
As shown there, the list of relevant $\g$ can be narrowed
considerably when $\n=1$; only $M=T,\tfrac T2$ contribute to the formula when $T>3$.
 We will state some special cases below, but
first we provide some additional motivation.

Simple supercuspidals are the representations of $\GL_2(\Q_p)$ with conductor $p^3$.
Assuming trivial central character,
they can be parametrized by the pairs $(t,\zeta)$ where 
$t\in (\Z/p\Z)^*$ and $\zeta\in\{\pm 1\}$.
There are thus $2(p-1)$ such representations, denoted $\sigma_t^\zeta$, and each is
constructed in the same way via compact induction from a character $\chi_{t,\zeta}$ of
a certain open compact-mod-center subgroup $H_t'$ of $\GL_2(\Q_p)$.

An interesting question is whether each member of 
such a local family has the same global multiplicity, in the following sense.
For $T>1$ square-free, consider $N=T^3$ in \eqref{Skdim}, with
$\widehat{\sigma}$ running over all tuples $(\sigma_{t_p}^{\zeta_p})_{p|T}$.
(We assume trivial central character for the moment, though the general case is considered
in the main body of this paper.)
In this case we have the dimension formula
\begin{equation}\label{dimcube}
\dim S_k^{\new}(T^3)=\frac{k-1}{12}\prod_{p|T}(p^2-1)(p-1)
\end{equation}
as in \cite{GMar}.
Since there is no immediately apparent 
reason for nature favoring one simple supercuspidal over another, 
one might surmise that the subspaces $S_k(\widehat{\sigma})$ all
 have the same dimension, i.e., that the asymptotic \eqref{KST}, which in the present
situation becomes
\begin{equation}\label{KST2}
\dim S_k(\widehat{\sigma})\sim \frac{k-1}{12}\prod_{p|T}\frac{p^2-1}2,
\end{equation}
 is an equality.
(Note that the right-hand side of \eqref{KST2} results from dividing \eqref{dimcube} by
 the number $2(p-1)$ of simple supercuspidals at each place $p|T$.)
This would be consistent with a 2011 calculation of Gross, who fixed the
 tuple of parameters $(t_p)_{p|N}$ and allowed
 the $\zeta_p$ parameters to vary, \cite[p. 1255]{G}.  Using the trace formula he showed 
\begin{equation}\label{gross}\sum_{(\zeta_p)_{p|T}} 
\dim S_k((\sigma_{t_p}^{\zeta_p})_{p|T})=
\frac{k-1}{12}\prod_{p|T}(p^2-1),\end{equation}
which is what one would expect, upon dividing \eqref{dimcube} by the number of tuples $(t_p)_{p|T}$.

However, \eqref{KST2} is {\em not} in fact an equality in general, for the simple reason that,
as we spell out at \eqref{id}, the right-hand side of \eqref{KST2}
 fails to be an integer for infinitely many values of $T$.
This is manifested in recent work \cite{PQ} of Pi and Qi, who considered a sum different from
that treated by Gross, namely, varying the $t_p$ and $\zeta_p$ parameters subject to the
constraint $(-1)^{k/2}\prod_{p|T}\zeta_p=\epsilon$ for fixed $\epsilon\in\{\pm 1\}$.
This amounts to counting
  the newforms with root number $\epsilon$.  They found, for $k\ge 4$ even
and square-free $T>3$, that
\begin{equation}\label{pq}
\dim S_k^{\new}(T^3)^{\pm}=\frac{k-1}{24}\prod_{p|T}(p^2-1)(p-1)\pm \frac{c_T}2\varphi(T)h(-T),
\end{equation}
where $c_T$ and $h$ are as in \eqref{biasA} and $\varphi$ is Euler's $\varphi$-function.
This shows that, just as in the case of square-free level, there is a bias in favor of 
positive root number.
  Instead of the Arthur-Selberg trace formula, they used a
Petersson formula obtained using the simple supercuspidal new vector matrix coefficient
from \cite{super}.

By evaluating the $S=\n=1$ case of Theorem \ref{mainST}, in \S \ref{dim} we obtain
an explicit formula for $\dim S_k(\widehat{\sigma})$ that
 refines each of the above results.   
 For example, we have the following.

\begin{theorem}\label{simplethm} Let $N=T^3$ for $T>3$ odd and square-free, let $k>2$ be even, 
and let $\widehat{\sigma}=(\sigma_{t_p}^{\zeta_p})_{p|N}$ 
be a tuple of simple supercuspidal representations with trivial central characters.
Then
\begin{equation}\label{simple}
\dim S_k(\widehat{\sigma})=\frac{k-1}{12}\prod_{p|N}\frac{p^2-1}2
  +\Delta(\widehat{t})\epsilon(k,\widehat{\zeta}) b_Th(-T),
\end{equation}
where $\Delta(\widehat{t})\in\{0,1\}$ is nonzero if and only if $-pt_p/T$ is a square
 modulo $p$ for each $p|T$, 
  $\epsilon(k,\widehat{\zeta})=(-1)^{k/2}\prod_{p|N}\zeta_p$ is the common global
 root number of the newforms comprising $H_k(\widehat{\sigma})$, 
$b_T$ is a certain power of $2$ depending on $T\mod 8$, and $h(-T)$ is the class number of $\Q[\sqrt{-T}]$. 
\end{theorem}
\noindent This is a special case of Theorem \ref{main}, which also allows for $T$ even.
The presence of $\Delta(\widehat{t})$ demonstrates that the dimension is 
not simply a function of the weight, level and root number
(even when the right-hand side of \eqref{KST2} is an integer).  Indeed, as described in \cite{He} for example,
each $\sigma_p$ has attached a ramified quadratic extension of $\Q_p$, namely $E_{\sigma_p}=\Q_p(\sqrt{t_pp})$,
 which depends only on the Legendre symbol $\left(\frac {t_p}p\right)$.  
So $\dim S_k(\widehat{\sigma})$ depends only on $T,k$, the fields $E_{\sigma_p}$,
and the global root number. (If $T$ is even, the dimension also depends on
the local root number $\zeta_2$.)

The second term in \eqref{simple} comes from an elliptic orbital integral. These do
  not appear in \eqref{gross}, but combine to form the error 
  term in \eqref{pq}.  Indeed, the local root number already appears as
  a coefficient in our local test function at the places dividing $T$, 
  so the global root number naturally appears in 
 the elliptic orbital integral that yields the error term in \eqref{simple}.  This helps explain the 
  positive bias of the root number in this situation.

At the end of \S \ref{dim}, we indicate how our results recover 
the dimension formula \eqref{dimcube} and the root number bias \eqref{pq}.
In Theorem \ref{N3dim} we find that the root numbers
of newforms of level 27 have a strict bias toward $-1$ (among the possibilities $\pm 1,\pm i$)
 when $k\equiv 5\mod 6$ and the nebentypus has conductor $3$.
%  Mentioning conductor of nebentypus is important because there are other newforms of level 27 
% that are principal series (though with higher conductor nebentypus)

In a more recent paper, K.\,Martin addressed the question of root number bias in $S_k^{\new}(N)$
 for arbitrary levels, \cite{KMar2}.  He showed that there is a bias towards
root number $+1$ with one exception: when $N=S^2$ for a square-free number $S$ and $(-1)^{k/2}
=-\prod_{p|S}(-1)$, then the root number has a strict negative bias when $k$ is sufficiently large.
In discussing why the exceptions arise, he noted that the picture is obscured by the 
existence of newforms of level $S^2$ which are twists of 
forms of lower level.  (No such forms exist in the $N=T^3$ case discussed above.)

Theorem \ref{mainST} allows us to investigate this further, since 
the subspace $S_k^{\min}(S^2)\subset S_k^{\new}(S^2)$ spanned by the newforms which are not 
twists of newforms of lower level is the direct sum
\[S_k^{\min}(S^2) = \bigoplus_{\widehat{\sigma}}S_k(\widehat{\sigma})\]
ranging over all $\widehat{\sigma}=(\sigma_p)_{p|S}$ with each $\sigma_p$ a supercuspidal
representation of conductor $p^2$ (a ``depth zero" supercuspidal) and trivial central character.

In fact, even without using a specialized trace formula, we can infer the existence of negative bias
for the root numbers in $S_k^{\min}(S^2)$ for many pairs $(S,k)$ by the following heuristic 
  coming from finite fields (see \S\ref{d0} for more detail and a summary of the construction
of depth zero supercuspidals).
Given an odd prime $p$, there are $p-1$ primitive characters of $\F_{p^2}^*$ with trivial restriction
to $\F_p^*$.  It follows that the number of $\sigma_p$ as above 
  is $\frac{p-1}2$.  If $p\equiv 3\mod 4$, this number is odd, so the set of such $\sigma_p$ contains
a preponderance either of local root number $\epsilon_p=+1$ or $\epsilon_p=-1$. 
So if $S$ is a product of such primes, then for some integer $c\ge 1$ there are $c$
 more tuples $\widehat{\sigma}$ with  
one non-archimedean sign $\epsilon_{\fin}=\prod_{p|S}\epsilon_p$ than the other.  
By \eqref{KST}, the spaces $S_k(\widehat{\sigma})$ all have roughly the same dimension 
  $\frac{k-1}{12}\prod_{p|S}(p-1)$, up to variations of lower magnitude when $k+S$ is sufficiently
large.  Then with $k/2$ of the appropriate parity,
there is a bias towards root number $\epsilon=(-1)^{k/2}\epsilon_{\fin}=-1$, 
with roughly $c\frac{k-1}{12}\prod_{p|S}(p-1)$
more forms of global root number $-1$ than $+1$.  (We will show that in fact $c=1$; see
Proposition \ref{Sbias}.)

To make a precise statement, we first apply Theorem \ref{mainST} with $\n=T=1$
 to obtain the following.

\begin{theorem}\label{d0thm} Let $N=S^2$ for $S>1$ square-free, let $k>2$ be even, 
and let $\widehat{\sigma}=(\sigma_{\nu_p})_{p|N}$ 
be a tuple of depth zero supercuspidal representations with trivial central characters,
 with $\nu_p$ the primitive character of $\F_{p^2}^*$ associated to $\sigma_p$.  
Then
\begin{equation}\label{d0dima}
\dim S_k(\widehat{\sigma})=\frac{k-1}{12}\prod_{p|S}(p-1)+
D_4(S)\frac{\epsilon(k,\widehat{\sigma})}4\prod_{\text{odd }p|S}2
+D_3(S)b(k)\frac{(-1)^{\delta_{3|S}}}3
 \prod_{p|S, p\neq 3}B(\nu_p),
\end{equation}
 where $\epsilon(k,\widehat{\sigma})$ is the common global root number of the 
newforms in $S_k(\widehat{\sigma})$, 
$D_4(S)\in\{0,1\}$ is $0$ if and only if $S$ is divisible by a prime $p\equiv 1\mod 4$,
$D_3(S)\in\{0,1\}$ is $0$ if and only if $S$ is divisible by a prime $p\equiv 1\mod 3$,
$\delta$ is the indicator function defined in \eqref{delta},
\begin{equation}\label{bk}
b(k)=\begin{cases}1&\text{if }6|k\\-1&\text{if }k\equiv 2\mod 6\\0&\text{otherwise,}\end{cases}
\end{equation}
and, for $p\equiv 2\mod 3$,
\[B(\nu_p)=\begin{cases}-2&\text{if the order of $\nu_p$ (in the character group of $\F_{p^2}^*$) divides
$\frac{p+1}3$}\\
1&\text{otherwise}.\end{cases}\]
\end{theorem}

The above is a special case of Theorem \ref{d0dim}, which allows for nontrivial nebentypus and $k$ odd.
We will use Theorem \ref{d0thm} to derive an explicit formula for the bias
\[\Delta(S^2,k)^{\min}=\dim S_k^{\min}(S^2)^+-\dim S_k^{\min}(S^2)^-\]
for $k>2$ even and $S>1$ square-free. This is given in Proposition \ref{Sbias}.
For the time being, we just state the following consequence, which is somewhat different from the 
behavior observed for the larger spaces of newforms of level $S^2$ appearing 
in \cite[Theorem 1.1(3) and Proposition 1.3]{KMar2}.

\begin{proposition}\label{Sbiasa}
Assume $k\ge 4$ is even.  With notation as above, $\Delta(S^2,k)^{\min}=0$ in each of the following situations:
(i) $D_4(S)=D_3(S)=0$, (ii) $S$ is divisible by some prime $p\equiv 5\mod 12$,
(iii) $D_4(S)=0$ and $k\equiv 4\mod 6$.

If $D_4(S)=0$, $k\equiv 0,2\mod 6$, $D_3(S)\neq 0$, and case (ii) does not apply, then 
$\Delta(S^2,k)^{\min}\neq 0$ and 
\[\sgn\Delta(S^2,k)^{\min}=(-1)^{\delta(k\equiv 6,8\mod 12)}\mu(S)\]   %\prod_{p|S}(-1)\]
 for the indicator $\delta$ as in \eqref{delta} and the M\"obius function $\mu$.

If $D_4(S)=1$ and $k\ge 6$, then apart from the two exceptions $S_8^{\min}(2^2)=S_6^{\min}(3^2)=0$, 
$\Delta(S^2,k)^{\min}\neq 0$, and
\[\sgn \Delta(S^2,k)^{\min}=(-1)^{\delta(2|S)+k/2}.\]

If $D_4(S)=1$ and $k=4$, then $\Delta(S^2,4)^{\min}\ge0$ for all square-free $S>1$:
\[\Delta(S^2,4)^{\min}=\begin{cases}\frac12\prod_{p|S}(p-1)&\text{if } 2\nmid S\\
0&\text{if }2|S.\end{cases}\]
\end{proposition}
\noindent{\em Remark:} A noteworthy difference between the above and the bias for the full space of newforms is that
here for any fixed even $k\ge 6$ there are infinitely many levels $S^2$ for which $\Delta(S^2,k)^{\min}<0$, whereas
by \cite[Theorem 1.1(3)]{KMar2}, for any fixed even $k$ there are only finitely many levels $N$
for which $\Delta(N,k)^{\new}<0$.

\vskip .2cm
\noindent{\bf Acknowledgements:} I am grateful to 
Charles Li for his very helpful input on the quotient measure
that is computed in \S \ref{msec}.  I also thank the anonymous referees for their many
insightful comments which led to improved exposition, and Andrew Booker, David Bradley, Min Lee, and 
  Kimball Martin for helpful conversations.
Partial support for this research was provided by an AMS-Simons Research Enhancement Grant
for Primarily Undergraduate Institution Faculty.

\section{Notation and Haar measure}\label{notation}

If $P$ is a statement, then we will frequently use the indicator function
\begin{equation}\label{delta}
\delta(P)=\delta_P =\begin{cases}1&\text{if $P$ is true}\\0&\text{if $P$ is false}.
\end{cases}\end{equation}
We also use the shorthand 
\[e(x)=e^{2\pi i x}.\] 
For rings $R$, we let $R^*$ denote the group of units in $R$.

Let $G$ be the group $\GL(2)$, and set $\olG=G/Z$,
where $Z$ is the center of $G$.  
If $H$ is a subgroup of $G$, then $\ol{H}$ will denote the group $HZ/Z \cong H/(H\cap Z)$. 
  For $\ell$ prime, we set $Z_\ell=Z(\Q_\ell)$ the center and
$K_\ell=G(\Z_\ell)$ the maximal compact subgroup of $G(\Q_\ell)$.
Groups $K_0(\p), K_1(\p^j), K'$ will be defined in \S \ref{bg} and \S\ref{ssr}.

     Let $\A$ be the adele ring of the rational numbers
$\Q$.
     We give $\olG(\A)$ the standard Haar measure for which
\[\meas(\olG(\Q)\bs\olG(\A))=\pi/3,\]
with the discrete group $\olG(\Q)$ receiving the counting measure.
We normalize Haar measure on $\olG(\Q_\ell)$ so that $\ol{K_\ell}$ has measure 1.
With this choice, there is a unique Haar measure on $\olG(\R)$
for which the above measure on $\olG(\A)$ is the restricted product of the measures
on $\olG(\Q_\ell)$ for $\ell\le \infty$.
It has the form $dm\,dn\,dk$, where $dm$ is the 
measure $(dx/|x|)^2$ on the diagonal subgroup $M\cong \R^*\times \R^*$, $dn$ is the 
measure $dx$ on the unipotent subgroup $N\cong \R$, and $dk$ is the measure on $K_\infty=\SO(2)$
of total measure $1$ (\cite[Corollary 7.45]{KL}).

For a unitary Hecke character $\w$, let
$L^2(\w)=L^2(G(\Q)\bs G(\A),\w)$ be the space of 
(classes of) measurable $\C$-valued functions $\phi$ on $G(\A)$ transforming under the center by $\w$ 
  and square integrable modulo $Z(\A)G(\Q)$. 
Let $L^1(\ol{\w})=L^1(G(\A),\ol{\w})$ be defined in the analogous way; its elements are
integrable modulo $Z(\A)$.

\section{The simple trace formula}

\subsection{Background on supercuspidal representations of $\GL(2)$}\label{bg}

Let $F$ be a non-archi\-me\-de\-an local field of characteristic 0 with integer ring
$\O$ and prime ideal $\p$.  In this section only, let $G=\GL_2(F)$,
 $B=B(F)$ the upper-triangular Borel subgroup, $N=N(F)$ the unipotent subgroup of $B$,
 $M$ the diagonal subgroup, $Z$ the center, and $K=G(\O)$ the standard maximal compact subgroup. 

Given a smooth irreducible representation $(\pi,V)$ of $G$, it is supercuspidal if it
satisfies any of the following equivalent
conditions (see, e.g., \cite[\S9-10]{BH}):
\begin{itemize}
 \item $V$ is the span of the vectors of the form $\pi(n)v-v$ for $v\in V$ and $n\in N$;
 \item The matrix coefficients of $\pi$ are compactly supported modulo the center;
 \item $\pi$ is not principal series or special, i.e., not
  a subquotient of a representation induced from a character of $B$.
\end{itemize}

The following property found by Harish-Chandra is crucial in what follows.  We sketch a proof
here for the reader's convenience, following \cite[\S2.2]{Si}.

\begin{proposition}
 Let $\pi$ be a supercuspidal representation of $G$, and let
   $f(g)=\sg{\pi(g)v,v'}$ be a matrix coefficient.  Then for all $g,h\in G$,
   \begin{equation}\label{supercusp}\int_N f(gnh)dn=0.\end{equation}
\end{proposition}
\begin{proof}
 We assume for simplicity that $\pi$
 is unitary, which is always the case if the central character is unitary.
 Then 
 \[f(gnh)=\sg{\pi(g)\pi(n)\pi(h)v,v'}=\sg{\pi(n)\pi(h)v,\pi(g^{-1})v'},\] 
 so we can assume without loss of generality that $g=h=1$.
 By linearity and the first bullet point above, we may also assume that $v=\pi(n_0)w-w$ for some 
 $w\in V$ and $n_0\in N$.  
 
 Let $N(v)$ be an open compact subgroup of $N$ containing $n_0$.
 Then
 \[\int_{N(v)}\pi(n)vdn = \int_{N(v)}\pi(n)(\pi(n_0)w-w)dn =\int_{N(v)}\pi(n)wdn-\int_{N(v)}\pi(n)wdn=0.\]
 (By smoothness, there exists an open compact subgroup $N'$ of $N(v)$ that fixes $w$, so
 the above integrals are really just finite sums.)
 Since $f$ has compact support, the support of $f|_N$ is contained in some open compact subgroup
   $N(v)$ as above.  Therefore 
 \[\int_N f(n)dn =\int_{N(v)}\sg{\pi(n)v,v'}dn =\sg{\int_{N(v)}\pi(n)vdn, v'} = 0.\qedhere\]
\end{proof}

\begin{corollary}\label{mc0}  If $f$ is a matrix coefficient of a supercuspidal representation, then
 for any $g,h\in G$ and $m\in M$ ($M$ being the diagonal subgroup),
 \[\int_N f(gn^{-1}mnh)dn=0.\]
\end{corollary}
\begin{proof} 
 This follows from $n^{-1}m n = m n'$, making a change of variables to integrate over $n'$,
 and applying the above proposition.
\end{proof}

%One good overview (with references) of some of the more explicit features of the supercuspidal 
%representations of $G$ is \cite[\S3]{LW}.
For any supercuspidal representation $\sigma$ of $G$, there exists an open and closed subgroup
$H\subset G$ containing $Z$, with $H/Z$ compact, and an irreducible representation $\rho$ of $H$,
such that $\sigma$ is compactly induced from $\rho$: $\sigma=\cInd_H^G(\rho)$.
Let $K_0(\p)=\mat{\O^*}{\O}{\p}{\O^*}$ be the Iwahori subgroup of $G$, and fix a prime element
 $\varpi$ of $\O$. 
Up to conjugacy, there are two maximal compact-mod-center subgroups of $G$, namely
\begin{equation}\label{max}
J=\begin{cases} ZK&\text{(the unramified case)}\\
ZK_0(\p)\cup Z\smat{}1\varpi{}K_0(\p)&\text{(the ramified case)}.\end{cases}
\end{equation}
The latter is the normalizer of $K_0(\p)$.  Without loss of generality, one of these 
contains $H$, and we call $\sigma$ {\bf unramified} or {\bf ramified} accordingly.\footnote[2]{
It should be borne in mind that in standard terminology, all supercuspidals are {\em ramified} 
in the sense that they have no $K$-fixed vector.  We are using the word in a different sense
here, reflecting the nature of the quadratic extension $E/F$ determined by $\sigma$, \cite{He}.
}
There is a unique ideal $\p^j$, called the {\bf conductor} of $\sigma$, such that the space of
vectors in $\sigma$ fixed by the group
\[K_1(\p^j)=\mat{\O^*}{\O}{\p^j}{1+\p^j}\]
is one-dimensional.  By \cite{Tu}, $j\ge 2$, and as explained in \cite{He},
$j$ is even
in the unramified case, and odd in the ramified case.

\subsection{Simple trace formula}\label{tf}

Given a unitary Hecke character $\w$ and a function $f\in L^1(\ol{\w})$, 
  we define the operator $R(f)$ on $L^2(\w)$ via
      \begin{equation}\label{Rf}
       R(f)\phi(x)=\int_{\olG(\A)}f(g)\phi(xg)dg.
      \end{equation}

    For $k>2$, let $\mathcal{C}_k$ denote the space of all continuous factorizable functions 
    $f=f_\infty\prod_{\ell< \infty}f_\ell$ on $G(\A)$ which transform
    under the center by $\ol{\w}$, such that
    $f_\ell$ is smooth and compactly supported modulo the center $Z_\ell$ for all $\ell$, 
    there is a finite set $S$ of places of $\Q$ such that for all $\ell\notin S$, 
    $f_\ell$ is supported on $Z_\ell K_\ell$ and has the value $1$ on $K_\ell$,
    and lastly,
    \[f_\infty(\smat abcd)\ll_{k} \frac{(ad-bc)^{k/2}}{(a^2+b^2+c^2+d^2+2|ad-bc|)^{k/2}}.\]
Then $\mathcal{C}_k\subset L^1(\ol{\w})$, and we can consider the operators $R(f)$
for such $f$.

     Recall that $\g\in G(\Q)$ is {\em elliptic} if its characteristic polynomial is irreducible.
     This concept is well defined on conjugacy classes and cosets of the center.
We will use the following simple trace formula.

     \begin{theorem}\label{stf}
      For $f\in \mathcal{C}_k$, suppose that for some finite place $v$ of $\Q$,
      $f_{v}$ is a matrix coefficient of a supercuspidal
      representation of $G_{v}=G(\Q_v)$, and therefore by Corollary \ref{mc0} its local hyperbolic
      orbital integrals vanish identically:
      \begin{equation}\label{hyp}
       \int_{M_v\bs G_v} f_v(g^{-1}\mat a{}{}1g)dg=0
      \end{equation}
      for all $a\in \Q_v^*$, where $M_v$ is the diagonal subgroup of $G_v$.
      Suppose further that \eqref{hyp} is also satisfied at a {\em second} place $w\neq v$
      (which may be archimedean).
      Then
    \[\tr R(f) = \meas(\olG(\Q)\bs\olG(\A)) f(1) + \sum_{\mathfrak{o}\text{ elliptic in }
    \olG(\Q)} \Phi({\mathfrak{o}},f),\]
      where, for an elliptic conjugacy class $\mathfrak{o}\subset \olG(\Q)$, the orbital
      integral is defined by
    \begin{equation}\label{Phio}
\Phi({\mathfrak{o}},f)=\int_{\olG(\Q)\bs\olG(\A)} \sum_{\g\in\mathfrak{o}}
    f(g^{-1}\g g)dg.
\end{equation}
\end{theorem}
  \begin{proof}  See \cite[Proposition V.2.1 and Theorem V.3.1]{Ge}.  The idea is that the validity
   of \eqref{hyp} at two distinct places
   kills off the hyperbolic and (nonidentity) unipotent 
   terms on the geometric side of the Arthur-Selberg trace formula, while the stronger 
   condition \eqref{supercusp} on $f_{v}$
   also forces the operator $R(f)$
   to have purely cuspidal image, so the continuous and residual spectral terms
   vanish as well. % \cite[Proposition 1.1]{Rog}.
   In Gelbart's exposition it is assumed that
   $f$ is compactly supported, but for $f\in \mathcal{C}_k$ everything still converges absolutely
   as shown in \cite{KL}, so the same proof is valid.
  \end{proof}

\subsection{Factorization of orbital integrals}\label{fact}

Here we explain how to compute elliptic orbital integrals locally.
The statements and proofs in this section are applicable over an arbitrary number field, though we express everything
in terms of $\Q$.

For $\g\in G(\Q)$, let $G_\g$ be its centralizer.
There are two related groups that will be needed.  First, since $Z(\Q)\subset G_\g(\Q)$, 
we may form the quotient, denoted $\ol{G_\g(\Q)}$.  
Second, the centralizer of $\g$ (or, more accurately, of the coset $\g Z(\Q)$)
 in $\olG(\Q)$ is denoted $\olG_\g(\Q)$.  
In general these are distinct subgroups of $\olG(\Q)$.  This will be clarified
in the proof of Lemma \ref{1/2} below.

 Giving the discrete group $\ol{G_\g(\Q)}$ the counting measure, define
\[\Phi(\g,f) =\int_{\ol{G_\g(\Q)}\bs\olG(\A)}f(g^{-1}\g g)dg.\]
For fixed measures on $\ol{G_\g(\R)}$ and $\ol{G_\g(\Q_\ell)}$, we also define the local orbital integrals
            \[\Phi(\g,f_\infty)=\int_{\ol{G_\g(\R)}\bs \olG(\R)}f_\infty(g^{-1}\g g)dg,\]
and
        \[\Phi(\g,f_\ell)=\int_{\ol{G_\g(\Q_\ell)}\bs \olG(\Q_\ell)}f_\ell(g^{-1}\g g)dg.\]
For compatibility, some care must be taken regarding the normalization of measures.  See the statement of
Proposition \ref{Pfactor} below.

\begin{lemma}\label{1/2}
 For an elliptic element $\g\in G(\Q)$, let $\mathfrak{o}$ be its conjugacy class in $\olG(\Q)$.
Then with notation as above and in \eqref{Phio},
     \[\Phi(\mathfrak{o},f)=\begin{cases}\Phi(\g,f)&\text{if }\tr\g\neq 0\\
     \ds\frac12\Phi(\g,f)&\text{if }\tr\g=0.\end{cases}\]
\end{lemma}
 \begin{proof}
  By definition,
\[\Phi(\mathfrak{o},f)=\int_{\olG(\Q)\bs\olG(\A)}\sum_{\delta\in \olG_\g(\Q)\bs \olG(\Q)}
f(g^{-1}\delta^{-1}\g\delta g)dg=\int_{\olG_\g(\Q)\bs\olG(\A)}f(g^{-1}\g g)dg.\]
  Notice that in the definition of $\Phi(\g,f)$, the quotient object is $\ol{G_\g(\Q)}$
  rather than $\olG_\g(\Q)$.  The former is a subgroup of the latter, and we claim that
  \[[\olG_\g(\Q):\ol{G_\g(\Q)}]=\begin{cases}1&\text{if }\tr\g \neq 0\\
  2&\text{if }\tr\g=0.\end{cases}\]
  The lemma follows immediately from this claim.  To prove the claim, note that
  \[\ol{G_\g(\Q)}=\{\delta\in\ G(\Q)|\, \delta^{-1}\g\delta=\g\}/Z(\Q)\]
  and
  \[   \olG_\g(\Q)
   =\{\delta\in G(\Q)|\, \delta^{-1}\g\delta =z\g\text{ for some }z\in\Q^*\}/Z(\Q).\]
For any such $z$, taking determinants we see that $z^2=1$, so $z=\pm 1$.
We also see that $\tr\g=z\tr \g$, so $z=1$ if $\tr\g\neq 0$, and in this case
  the two groups are equal, as claimed.  
  On the other hand, if $\tr\g=0$, then $\g$ is conjugate in $G(\Q)$ to
  its rational canonical form $\smat 0{-\det\g}10$, and
  \[\mat1{}{}{-1}\mat 0{-\det\g}10\mat1{}{}{-1}= \mat0{\det\g}{-1}0,\]
  from which it follows that $\delta^{-1}\g\delta=-\g$ has a solution $\delta$.
Given any such $\delta$, we find easily that 
  \[\olG_\g(\Q)=\ol{G_\g(\Q)}\cup \delta\,\ol{G_\g(\Q)}.\qedhere\]
 \end{proof}

\begin{proposition}\label{Pfactor}
 Let $f\in \mathcal{C}_k$ as defined in \S\ref{tf}, and let $\g\in G(\Q)$ be an elliptic element.
         Then for any fixed choice of Haar measures on $\olG(\A)$ and $\ol{G_\g(\A)}$,
    \begin{equation}\label{factor}
        \Phi(\g,f) =\meas(\ol{G_\g(\Q)}\bs\ol{G_\g(\A)})\prod_{\ell\le \infty}\Phi(\g,f_\ell),
    \end{equation}
where the measures on the groups $\olG(\Q_\ell)$ are chosen (noncanonically)
so that the measure on $\olG(\A)$ is the restricted product of these local measures relative
to the maximal compact subgroups almost everywhere,
  and likewise the measures on the groups $\ol{G_\g(\Q_\ell)}$ are chosen compatibly
with the fixed measure on $\ol{G_\g(\A)}=\prod'_{\ell\le \infty}\ol{G_\g(\Q_\ell)}$.
 \end{proposition}

\noindent{\em Remarks:} (1) This is well known, but as we have not found a proof 
in the literature, we include one below.  Tate's thesis shows that if the product is absolutely
convergent, then the left-hand integral converges absolutely and the equality holds.  
But here we need a kind of converse: we know
a priori that $\Phi(\g,f)$ is absolutely convergent.

(2) The specific choice of measures to be used in this paper is summarized in \S\ref{etasec}, 
  where it is shown that
the quotient space $\ol{G_\g(\Q)}\bs\ol{G_\g(\A)}$ is compact, and its measure is
computed explicitly in the more general setting with $\Q$ replaced by an
arbitrary number field.

\begin{proof}  
    Observe that
     \begin{align*}
         \Phi(\g,f)&=\int_{\ol{G_\g(\Q)}\bs \olG(\A)}f(g^{-1}\g g)dg\\
         &=\meas(\ol{G_\g(\Q)}\bs\ol{G_\g(\A)})\int_{\ol{G_\g(\A)}\bs \olG(\A)}f(g^{-1}\g g)dg
     \end{align*}
    for any choice of Haar measure on $\ol{G_\g(\A)}$.  Absolute convergence is
    proven for $f\in \mathcal{C}_k$  in \cite[Corollary 19.3]{KL}.

    For notational convenience, write $\phi(g)=f(g^{-1}\g g)$ (a function on $\olG(\A)$), and 
    $\phi_\ell(g_\ell)=f_\ell(g_\ell^{-1}\g g_\ell)$ for $\ell\le \infty$,
    so $\phi(g)=\prod_{\ell\le \infty} \phi_\ell(g_\ell)$.
Also, define 
\[X=\ol{G_\g(\A)}\bs\olG(\A).\]
  Then $X$ is the restricted product of the spaces
    \[X_\ell=\ol{G_\g(\Q_\ell)}\bs \olG(\Q_\ell)\]
relative to the open compact subsets
    $H_\ell = \ol{G_\g(\Q_\ell)}\bs \ol{G_\g(\Q_\ell) K_\ell}\subset X_\ell$.
Indeed, the natural map from $\olG(\A)$ to $\prod'X_\ell$ is clearly surjective, with
kernel $\ol{G_\g(\A)}$. 

Fix Haar measures on
each of the local groups $\olG(\Q_\ell)$ and $\ol{G_\g(\Q_\ell)}$ 
compatibly with the fixed Haar measures on $\olG(\A)$ and $\ol{G_\g(\A)}$.
This determines a right-$\olG(\Q_\ell)$-invariant measure on $X_\ell$ with the property 
that $H_\ell$ has measure $1$ for almost all $\ell$.
 Let $S$ be the finite set of places of $\Q$ outside of which $f_\ell$ is supported on
 $Z(\Q_\ell)K_\ell$ with $f_\ell(zk)=\ol{\w}(z)$.  Let $S'$ be a finite set of places
outside of which (1) $\g\in K_\ell$, and (2) $H_\ell$ has measure $1$. 
Then setting $S_0=S\cup S'$,  for $\ell\notin S_0$ we have
\[\int_{H_\ell}\phi_\ell(h)dh=\int_{H_\ell}f_\ell(k^{-1}\g k)dk =\meas(H_\ell)=1.\]
Let
    \[S_0\subset S_1\subset S_2\subset\cdots\]
    be a sequence of finite sets of primes (including $\infty$) whose union is the full 
set of primes.
    Let $\chi_{n}$ be the characteristic function of $X_{S_n}=\prod_{\ell \in S_n}X_\ell\times
    \prod_{\ell\notin S_n}H_\ell,$ and let $\phi_n=\phi\cdot \chi_n$.  Note that
    $\phi_n\to\phi$ pointwise.  Since 
    $\phi\in L^1(X)$ as mentioned above, so is $\phi_n$. By the Dominated Convergence Theorem,
    \[\int_X \phi(x)dx = \lim_{n\to\infty}\int_X\phi_n(x)dx
    =\lim_{n\to\infty}\prod_{\ell\in S_n}\int_{X_\ell}\phi_\ell(x_\ell)dx_\ell,\]
    as needed.
\end{proof}

\section{Counting locally supercuspidal newforms}\label{count}

Here we explain how to use the simple trace
formula to count cusp forms with prescribed supercuspidal ramification.
To set notation,
let $N=\prod_{p|N}p^{N_p}>1$ be a positive integer with the property that 
$N_p\ge 2$ for each prime $p|N$.   Fix a Dirichlet character $\w'$ modulo $N$ of conductor
dividing $\prod_{p|N}p^{\lfloor N_p/2\rfloor}$. This requirement comes from the fact
that the central character of a supercuspidal representation of conductor $p^{N_p}$ divides
$p^{\lfloor N_p/2\rfloor}$, \cite[Proposition 3.4]{Tu}.
Let $\w: \A^*\longrightarrow \C^*$ be the finite order Hecke character associated to $\w'$ via
\begin{equation}\label{wfact}
\A^*=\Q^*(\R^+\times \Zhat^*)\longrightarrow \Zhat^*/(1+N\Zhat)\cong (\Z/N\Z)^*\longrightarrow\C^*,
\end{equation}
where the last arrow is $\w'$.
Letting $\w_p$ be the restriction of $\w$ to $\Q_p^*$, for any prime $p|N$ we have
\begin{equation}\label{wpp}
\w_p(p) = \w(1,\ldots,1,p,1,\ldots) = \w(p^{-1},\ldots,p^{-1},1,p^{-1},\ldots)
=\prod_{\substack{\ell|N,\\{\ell\neq p}}}\w_\ell(p^{-1}).
\end{equation}

Fix an integer $k\ge 2$ satisfying 
\[\w'(-1)=(-1)^k,\]
and let $S_k(N,\w')$ be the space of cusp forms $h$ satisfying
\[h\Bigl(\frac{az+b}{cz+d}\Bigr)=\w'(d)^{-1}(cz+d)^kh(z)\]
for $\smat abcd\in \Gamma_0(N)$.   The inverse on $\w'(d)$ is somewhat nonstandard.
It ensures that the adelic cusp form attached to $h$ has central character $\w$ rather than
$\w^{-1}$. See, e.g., \cite[\S12.2-12.4]{KL}.  Because we mostly work in the adelic 
setting, it eases the notation to include the inverse in the classical setting.

For each $p|N$, fix a supercuspidal representation $\sigma_p$ 
of $\GL_2(\Q_p)$ of conductor $p^{N_p}$ and central character $\w_p$, 
and let $\widehat{\sigma}=\{\sigma_p\}_{p|N}$.  
We define $H_k(\widehat{\sigma})$ to be the set of newforms $h\in S_k(N,\w')$
  whose associated cuspidal representation $\pi_h$ has the local representation type $\sigma_p$
   at each $p|N$.
We set $S_k(\widehat{\sigma})=\Span H_k(\widehat{\sigma})$. 
The Dirichlet character $\w'$ is uniquely determined by the tuple $\widehat{\sigma}$ via
\begin{equation}\label{w'}
\w'(d)=\prod_{p|N}\w_p(d)\qquad ((d,N)=1),
\end{equation}
and this justifies our suppression of the central character $\w'$ 
  from the notation $S_k(\widehat{\sigma})$.   At a certain point we will use the fact that
\begin{equation}\label{wp}
\prod_{p|N}\w_p(N)=\w(N^{-1},\ldots,N^{-1},1,\ldots,1,N^{-1},N^{-1},\ldots)=\w'(1)=1.
\end{equation}

\subsection{Isolating $S_k(\widehat{\sigma})$ spectrally}\label{spectral}

For each prime $p|N$, we can write $\sigma_p=\cInd_{H_p}^{G_p}\rho$, where $H_p$ is
contained either in $Z_pK_p$ or the normalizer of an Iwahori subgroup, as in \eqref{max}.
By \cite[Proposition 2.1]{KR},
  there exists a unit vector $w_p$ in the space of $\sigma_p$ such that
the matrix coefficient $\sg{\sigma_p(g)w_p,w_p}$ is supported in $H_p$.
Fix once and for all such a vector $w_p$ for each $p|N$.
Based on this choice, we define a subspace $A_k(\widehat{\sigma})\subset L^2(\w)$ by
     \[A_k(\widehat{\sigma})=\bigoplus_\pi \C w_\pi,\]
 where $\pi$ ranges over the cuspidal automorphic representations with central character $\w$
  for which $\pi_\infty=\pi_k$, $\pi_p=\sigma_p$ for each $p|N$, and
$\pi_\ell$ is unramified for all finite
 primes $\ell\nmid N$, and $w_\pi=\otimes w_{\pi_\ell}$ is defined by
 \begin{equation}\label{wdef}
w_{\pi_\ell}=\begin{cases} \text{unit lowest weight vector}&\text{if }\ell=\infty\\
     \text{unit spherical vector}&\text{if }\ell\nmid N\infty\\
     w_p \text{ (fixed above)}&\text{if }\ell=p|N.
 \end{cases}
\end{equation}
     Here, for almost all $\ell$, the spherical vector is the one predetermined
     by the restricted tensor product $\pi\cong \otimes'_{\ell\le \infty}\pi_\ell$.
 The space $A_k(\widehat{\sigma})$
 does not consist of adelic newforms in general because at places
$p|N$, $w_p$ is not necessarily a new vector in the space of the local representation
  $\sigma_p$.
Nevertheless, $A_k(\widehat{\sigma})$ has the same dimension
as the space of newforms $S_k(\widehat{\sigma})=\Span H_k(\widehat{\sigma})$.

Using matrix coefficients, we can define a test function $f\in L^1(\ol{\w})$ 
for which $R(f)$ is the orthogonal projection of $L^2(\w)$ onto $A_k(\widehat{\sigma})$. 
Without much extra work, we can incorporate a Hecke operator into the test function.

Fix an integer $\n>1$ with $\gcd(\n,N)=1$, and let $T_\n$ be the classical Hecke operator
defined by
\[T_\n h(z) = \n^{k-1}\sum_{ad=\n\atop{a>0}}\sum_{r\mod d}{\w'(a)}^{-1}d^{-k}h\Bigl(\frac{az+r}d\Bigr)
\qquad(h\in S_k(N,\w'),\, z\in \mathbf{H}).\]
When $\n=1$, $T_\n$ is simply the identity operator.

The operator $T_\n$ can be realized adelically.  Let
\[M(\n)_\ell=\{g\in M_2(\Z_\ell)|\, \det g\in \n\Z_\ell^*\}\]
for each prime $\ell\nmid N$.
(If working over a larger number field $F$, one would take $\n$ to be an ideal of the integer 
ring and for a place $v<\infty$, set $M(\n)_v=\{g\in M_2(\O_v)|\, (\det g)\O_v = \n\}$.)
Define a function $\fnp:G(\Q_\ell)\longrightarrow \C$ by
\begin{equation}\label{fnp}
\fnp(g)=\begin{cases}\ol{\w_\ell(z)}&\text{if }g=zm\text{ for }z\in Z_\ell, m\in M(\n)_\ell\\
0&\text{if }g\notin Z_\ell M(\n)_\ell,\end{cases}
\end{equation}
where $\w_\ell$ is the local component of the Hecke character $\w$.
% though in this purely local section it can be any unramified unitary character of $\Q_\ell^*$.
Note that $\fnp$ is bi-$K_\ell$-invariant, and indeed
when $\n\in \Z_\ell^*$, this function is given by
\begin{equation}\label{fell}
f_\ell(g)=\begin{cases}\ol{\w_\ell(z)}&\text{if }g=zk\in Z_\ell K_\ell\\
0&\text{if }g\notin Z_\ell K_\ell.\end{cases}
\end{equation}

Next, let $\pi_k$ be the discrete series representation of $\olG(\R)$
of weight $k$, and let $v$ be a lowest weight unit vector in the space of $\pi_k$.
We define $f_\infty=d_k\ol{\sg{\pi_k(g)v,v}}$, where $d_k=\frac{k-1}{4\pi}$
is the formal degree of $\pi_k$.
Explicitly, with Haar measure on $\olG(\R)$ normalized as in \S\ref{notation},
\begin{equation}\label{finf}
f_\infty(\mat abcd)=\begin{cases}\ds\frac{k-1}{4\pi}\frac{(ad-bc)^{k/2}(2i)^k}
{(-b+c+(a+d)i)^k}&\text{if }ad-bc>0\\0&\text{otherwise}\end{cases}
\end{equation}
(\cite[Theorem 14.5]{KL}).  This function is integrable over $\olG(\R)$ exactly when $k>2$,
so the latter will be assumed throughout.  It would be possible to treat the $k=2$
case by using a pseudo-coefficient of $\pi_k$, but we have not attempted to carry this out (see \cite{P}).

At places $p|N$, define
\begin{equation}\label{sigmamc}
f_p(g)=d_{\sigma_p}\ol{\sg{\sigma_p(g)w_p,w_p}},
\end{equation}
where $d_{\sigma_p}$ is the formal degree and $w_p$ is the unit vector fixed above.
    The formal degree depends on 
  a choice of Haar measure on $\olG(\Q_p)$, which
  we normalize as in \S\ref{notation}.  By our choice of $w_p$, the support of $f_p$ is 
contained in one of the two groups \eqref{max}, according to whether or not $\sigma_p$ is
ramified.

  Finally, we define the global test function
\begin{equation}\label{fn}
\fn = f_\infty\prod_{p|N}f_p\prod_{\ell\nmid N}\fnp,
\end{equation}
for $f_\infty$ of weight $k$ as in \eqref{finf}, $f_p$ as in \eqref{sigmamc}, 
and $\fnp$ as in \eqref{fnp}.

 \begin{proposition}\label{proj}
 With the above definition of $\fn$, 
 the operator $R(\fn)$ (defined in \eqref{Rf} 
  taking Haar measure on $\olG(\A)$ as normalized in \S\ref{notation})
 factors through the orthogonal projection onto the finite dimensional subspace
 $A_k(\widehat{\sigma})$.   On this space, $R(\fn)$ acts diagonally, with the vectors $w_\pi$
 being eigenvectors.  
In more detail, given a newform $h\in H_k(\widehat{\sigma})$ with $T_\n h = a_\n(h)h$,
 let $w\in A_k(\widehat{\sigma})$ be the vector associated to $\pi_h$ as in \eqref{wdef}.
 Then 
 \[R(f^\n)w=\n^{1-k/2}a_\n(h)w.\]
 Consequently,
 \[\tr(T_\n|S_k(\widehat{\sigma}))=\n^{k/2-1}\tr R(f^\n).\]
 \end{proposition}
\noindent{\em Remarks:} (1) The vector $w$ is defined only up to unitary scaling,
  but of course the eigenvalue is independent of the choice.

(2) One can also take $f_p$ to be the complex conjugate of the 
trace of the representation $\rho$ inducing $\sigma_p$, if normalized correctly.  
See Proposition \ref{projtr} and its remark.

\begin{proof} %[Proof of Proposition \ref{proj}]
 The first statement is proven in \cite[Proposition 2.3]{skuzrr}, but we need to reproduce
some of the argument here for the second part.
 Let $h\in H_k(\widehat{\sigma})$, let $\pi$ be the associated 
  cuspidal representation, and let $w=w_\pi\in A_k(\widehat{\sigma})$.
For each place $v|\infty N$, the test function $f_v$ was chosen so that
\[\pi_v(f_v)w_v = w_v,\]
\cite[Cor. 10.26]{KL}. 
   Write
\[w=w_\infty\otimes \bigotimes_{p|N}w_p\otimes w'\otimes \bigotimes_{\ell|\n}w_\ell,\]
where $w'=\otimes_{\ell \nmid N\n}w_\ell$. 
We may likewise decompose $\pi$ as
\[\pi=\pi_\infty\otimes \bigotimes_{p|N} \pi_p\otimes \pi'\otimes\bigotimes_{\ell |\n}\pi_\ell,\]
where $\pi'$ is a representation of $G'=\prod'_{p\nmid N\n}G(\Q_p)$.
 Then letting $f'=\prod_{p\nmid N\n }f_p$, it is elementary to show that $\pi'(f')w'=w'$.
  Therefore (by \cite[Prop. 13.17]{KL})
\[R(f^\n)w=\pi_\infty(f_\infty)w_\infty\otimes\bigotimes_{p|N} \pi_p(f_p)w_p\otimes
 \pi'(f')w'\otimes \bigotimes_{\ell|\n}\pi_\ell(f_\ell^\n)w_\ell\]
\[
=w_\infty\otimes \bigotimes_{p|N}w_p\otimes w'\otimes \bigotimes_{\ell|\n}\pi_\ell(f_\ell^\n)w_\ell.\]
Since $w_\ell$ is an unramified unit vector in the principal series 
  representation $\pi_\ell=\pi(\chi_1,\chi_2)$ (say), we have
$\pi_\ell(f_\ell^\n)w_\ell=\lambda_\ell w_\ell$ for 
\[\lambda_\ell=\ell^{a/2}\sum_{j=0}^a\chi_1(\ell)^j\chi_2(\ell)^{a-j}\]
where $a=\ord_\ell(\n)$ (see e.g. \cite[Prop. 4.4]{pethil}).
Thus $R(f^\n)w=\lambda w$, where $\lambda=\prod_{\ell|\n}\lambda_\ell$. 
The result now follows by the well-known fact that $\prod_{\ell|\n}\lambda_\ell=\n^{1-k/2}a_\n(h)$.
The latter may be proven as follows.
If we let $v$ (denoted $\varphi_h$ in \cite{KL}) be the adelic new vector attached to $h$,
then $v$ is a pure tensor, differing from $w$ only at the places $p|N$.  A 
test function $\tilde{f}^\n$, say, is used in \cite{KL} that differs from $\f^\n$ 
  only at the places $p|N$.  By the same argument as above,
\[R(\tilde{f}^\n)v=v_\infty\otimes\bigotimes_{p|N}v_p\otimes v'\otimes \bigotimes_{\ell|\n}R(f_\ell^\n)v_\ell.\]
Since $v_\ell=w_\ell$ at places $\ell|\n$, the eigenvalues are the same, i.e., 
  $R(\tilde{f}^\n)v=\lambda v$.  
By \cite[Theorem 13.14]{KL} (which uses a global argument), $\lambda =\n^{1-k/2}a_\n(h)$. 
\end{proof}

\subsection{First main result: the trace of a Hecke operator}

We now state our first main theorem, which is a general formula for the trace of $T_\n$ 
  on $S_k(\widehat{\sigma})$.  Its proof will occupy the remainder of \S\ref{count}.

\begin{theorem}\label{stf2}
Let $k>2$, let the level $N$, nebentypus $\w'$, and tuple $\widehat{\sigma}=(\sigma_p)_{p|N}$
of supercuspidals  be fixed
as at the beginning of \S\ref{count} (ensuring compatibility of central characters with $\w'$),
and let $f=f^\n$ as in \eqref{fn}. 
Let $T$ be the product of all primes $p|N$ with $\ord_p(N)$ odd.  Then
\[\tr(T_\n|S_k(\widehat{\sigma}))=\n^{k/2-1}\left[\ol{\w'(\n^{1/2})}\frac{k-1}{12}
\prod_{p|N}d_{\sigma_p}+\frac12\sum_{M|T}\Phi(\mat{}{-\n M}1{},f)\right.\]
\[\left.+\sum_{M|T}\sum_{1\le r<\sqrt{4\n/M}}\Phi(\mat0{-\n M}1{rM},f)\right],\]
where $\w'(\n^{1/2})$ is taken to be $0$ if $\n$ is not a perfect square,  
$d_{\sigma_p}$ is the formal degree of $\sigma_p$ relative to Haar measure fixed in \S\ref{notation},
and the orbital integrals $\Phi(\g,f)$ are defined in \S\ref{fact}.

An orbital integral $\Phi(\g,f)$ as above vanishes unless $\g$ is elliptic in $G(\Q_p)$ for
each $p|N$.  Assuming this condition is satisfied,
let $E=\Q[\g]$ be the imaginary quadratic extension of $\Q$ generated by $\g$, and let
$h(E), w(E)$, and $d_E$ be the class number, number of units, and discriminant of $E$
   respectively.  Then
\begin{equation}\label{Phi}
\Phi(\g,f)=-\frac{2h(E)}{w(E)2^{\w(d_E)}}\frac{\sin((k-1)\theta_\g)}{\sin(\theta_\g)}
\prod_{p|\Delta_\g N}\Phi(\g,f_p),
\end{equation}
where $\Delta_\g$ is the discriminant
of $\g$, $\theta_\g =\arctan(\sqrt{|\Delta_\g|}/\tr\g)$ (interpreted as $\pi/2$ if
$\tr\g =0$) is the argument of one of the complex eigenvalues of $\g$, 
$\w(d_E)$ is the number of prime factors of $d_E$, 
 and our choice of measure for $\Phi(\g,f_p)$ is 
summarized in \S\ref{etasec} below.
\end{theorem}

\noindent{\em Remarks:} (1) For primes $p\nmid N$, the local orbital integrals 
  $\Phi(\g,f_p)$ are computed explicitly in \S\ref{ghyp}-\ref{gell} below.
Thus, for the explicit calculation of $\tr(T_\n|S_k(\widehat{\sigma}))$ 
it only remains to calculate the local orbital integrals $\Phi(\g,f_p)$ for $p|N$.\\
(2) When $\n=1$, the set of relevant $\g$ is considerably smaller than what appears above if $T>1$, due
to local considerations at $p|T$.  See Theorem \ref{dimST}.\\

The proof of Theorem \ref{stf2} involves results from the rest of \S \ref{count},
 outlined as follows.  First, the test function $f$
satisfies the hypotheses of Theorem \ref{stf}.  Indeed,
 the hyperbolic orbital integrals of $f_\infty$ vanish as shown in 
 \cite[Proposition 24.2]{KL}, and the fact that $f\in \mathcal{C}_k$ is
 a consequence of the formula for $f_\infty$ (see \cite[Lemma 14.2]{KL}).

Since we are normalizing measure so that
   $\meas(\olG(\Q)\bs\olG(\A))=\frac{\pi}{3}$,
the identity term in Theorem \ref{stf} is 
\[\frac\pi3f(1)=\frac{k-1}{12}\prod_{p|N}d_{\sigma_p}\prod_{\ell|\n}\fnp(1).\]
From the definition \eqref{fnp} of $\fnp$, we see that $\fnp(1)\neq 0$ only if $1\in Z_\ell M(\n)_\ell$, 
which holds if and only if $\n$ is a perfect square. Assuming this is the case, 
\[\fnp(1)=\fnp(\mat{\sqrt\n}{}{}{\sqrt\n}^{-1} \mat{\sqrt\n}{}{}{\sqrt\n}) 
=\w_\ell(\sqrt{\n}).
\]
Note that by \eqref{w'}
\[\prod_{\ell|\n}\w_\ell(\sqrt{\n})=\prod_{\ell\nmid N}\w_\ell(\sqrt{\n})
=\prod_{\ell|N}\ol{\w_\ell(\sqrt{\n})} = \ol{\w'(\sqrt{\n})}.\]
Therefore the identity term is
\[\frac\pi3f(1)=
\ol{\w'(\sqrt\n)}\frac{k-1}{12}\prod_{p|N}d_{\sigma_p},\]
where it is to be understood that $\w'(\sqrt{\n})=0$ if $\n$ is not a perfect square.

The structure of the first part of Theorem \ref{stf2} is then immediate from Theorem \ref{stf}, 
  Lemma \ref{1/2}, and Proposition \ref{proj}.
The set of relevant $\g$ is determined in \S\ref{relsec} below, simply by considering the 
  supports of the local test functions.
The vanishing of $\Phi(\g,f)$ if $\g$ is hyperbolic in $G(\R)$ or $G(\Q_p)$ for
some  $p|N$ is explained in Proposition \ref{ellg} below.

As for \eqref{Phi}, the first factor is equal to $\meas(\ol{G_\g(\Q)}\bs \ol{G_\g(\A)})$ under
our normalization of Haar measures on $G(\A)$ and $G_\g(\A)$. 
  This is shown in Theorem \ref{meas} below.  
  The second factor of \eqref{Phi} (along with the negative sign) is
  $\Phi(\g,f_\infty)$ as in \eqref{Rell} below.  
In \S\ref{ghyp}-\ref{gell} we explicitly compute the local orbital integrals away from the 
level, and see in particular that the value is $1$ at places not dividing $\Delta_\g N$. 

The local orbital integrals at the places dividing $N$ of course depend on
  the choice of supercuspidal representations.  The method we use to treat 
  the special cases of simple supercuspidals and depth zero supercuspidals in the 
second part of this paper
  is presumably applicable to other cases as well.

\subsection{Known results about the elliptic terms}

We record here some basic properties of the elliptic orbital integrals that
arise in Theorem \ref{stf2}.

\begin{proposition}\label{ellg}
    Let $\g$ be elliptic in $G(\Q)$.  Then for the test function $f=\fn$ of \eqref{fn}:
    \begin{enumerate}
        \item $\Phi(\g,f)$ is absolutely convergent;
        \item $\Phi(\g,f)$ depends only on the conjugacy class of $\g$ in $G(\A)$ (rather than
            in $G(\Q)$), and likewise for any prime $\ell$,
            $\Phi(\g,f_\ell)$ depends only on the $G(\Q_\ell)$-conjugacy
            class of $\g$;
        \item $\Phi(\g,f)=0$ unless: $\det\g>0$ and $\g$ is elliptic both in $G(\R)$ and
in $G(\Q_p)$ for each $p|M$;
        \item If $\g$ is elliptic in $G(\R)$ with a complex eigenvalue $\rho=re^{i\theta}$,
 then
   \begin{align}\label{Rell}
   \Phi(\g,f_\infty)=-r^{2-k}\frac{\rho^{k-1}-\ol{\rho}^{k-1}}{\rho-\ol{\rho}}
=-\frac{e^{i(k-1)\theta}-e^{-i(k-1)\theta}}{e^{i\theta}-e^{-i\theta}}
=-\frac{\sin({(k-1)\theta})}{\sin({\theta})}.
    \end{align}
    \end{enumerate}
\end{proposition}
\noindent{\em Remarks:} If $\g$ has discriminant $\Delta_\g<0$ and nonzero trace, then we may take
$\theta=\arctan(\sqrt{|\Delta_\g|}/\tr\g)$ in \eqref{Rell}. 
If $\g$ has the form $\smat{}u1{}$, then we may take $\theta=\frac\pi2$, giving
    \begin{equation}\label{Phiinf}
        \Phi(\mat{}u1{},f_\infty)=-\left[\frac{i^{k-1}-(-i)^{k-1}}{2i}\right]
      =\begin{cases}(-1)^{k/2}& \text{if $k$ is even}\\0&\text{if $k$ is odd.}\end{cases}
    \end{equation}

     \begin{proof} Nearly everything is proven in \cite[pp. 295-302]{KL}.  The 
         only remaining point is that $\Phi(\g,f_p)=0$ if $\g$ is hyperbolic in $G(\Q_p)$ for
         some $p|M$.
For such $\g$, after conjugating we can take
         $\g$ diagonal, so $G_\g(\Q_p)=M(\Q_p)$. The orbital integral is then 
         taken over $M_p\bs G_p$
         and involves integrating over $N(\Q_p)$ (see \eqref{hypint} below).
         We can use \eqref{supercusp} to show that it vanishes (as in \eqref{hyp}).  
     \end{proof}

\subsection{Local orbital integrals at primes $\ell\nmid N$: hyperbolic case}\label{ghyp}

If $\g\in G(\Q)$ is elliptic, then for each prime $\ell$,
$\g$ is either hyperbolic or elliptic in $G(\Q_\ell)$.  
In this section and the next we evaluate the local elliptic orbital integrals 
  at primes $\ell\nmid N$.   The methods are standard and the results are presumably not new.
For the dimension formulas we require the test function $f_\ell$ given by \eqref{fell}.
However, without any extra work we can consider
a general local Hecke operator, and consider an arbitrary $p$-adic field.

Thus, we let $F$ be a $p$-adic field with valuation $v$, uniformizer $\varpi$,
 ring of integers $\O_F$, maximal ideal $\p=\varpi\O_F$, and $q_v=|\O_F/\p|$.  Fix an unramified unitary character
$\w_v:F^*\longrightarrow\C^*$.
For an integral ideal $\mfn\subset \O_F$, define
\[M(\mfn) =\{g\in M_2(\O_F)|\, (\det g)\O_F =\mfn\}\]
and
\begin{equation}\label{fv}
\fv(g)=\begin{cases}\ol{\w_v(z)}&\text{if }g=zm\in Z(F)M(\mfn)\\
0&\text{if }g\notin Z(F)M(\mfn).\end{cases}
\end{equation}

If $\g$ is hyperbolic in $G(F)$, then replacing it by a conjugate if necessary, we 
can assume that it is diagonal.  In this case, $G_\g(F)=M(F)$ is the set of invertible
diagonal matrices.  We may
integrate over $\olG(F)$ using the Iwasawa coordinates
\[\int_{\olG(F)}\phi(g)dg=\int_{\ol{M}(F)}\int_{N(F)}\int_{K_v}\phi(mnk) dm\,dn\,dk,\]
where $K_v=G(\O_F)$.
Therefore if $\phi$ is $M(F)$-invariant,
\begin{equation}\label{hypint}
    \int_{\ol{G_\g(F)}\bs \olG(F)}\phi(g)dg=\int_{N(F)}\int_{K_v}\phi(nk) dn\,dk.
\end{equation}
We normalize the measures $dn$ and $dk$ by taking $\meas(N(\O_F))=\meas(K_v)=1$.

\begin{proposition}\label{ellh} For $F$ as above, suppose $\g$ is hyperbolic in $G(F)$.
Assuming $\g\in M(\mfn)$, and letting $\Delta_\g\in\O_F$ be its discriminant, we have
\[\Phi(\g,\fv)=|\Delta_\g|_v^{-1/2}.\]
In particular, if $\Delta_\g$ is a unit, then $\Phi(\g,\fv)=1$.
\end{proposition}

\begin{proof}
    We may assume that $\g=\mat \alpha{}{}\beta$   for some distinct $\alpha,\beta\in\O_F$.
    By \eqref{hypint} and the fact that $\fv$ is right $K_v$-invariant,
    \[\Phi(\g,\fv)=\int_{F} \fv(\mat1{-t}{}1\mat \alpha{}{}\beta\mat1t{}1)dt
    =\int_{F}\fv(\mat \alpha{t(\alpha-\beta)}{}\beta)dt.\]
 Choose $j\ge 0$ so that
    $\alpha-\beta\in \varpi^j\O_F^*$.
    By hypothesis, $\alpha,\beta\in\O_F$ and $\alpha\beta\O_F=\mfn$, so the 
  integrand is nonzero if and only if 
   $t(\alpha-\beta)\in\O_F$, which is equivalent to
    $t\in \varpi^{-j}\O_F$. Therefore
\[\Phi(\g,\fv)=\meas(\varpi^{-j}\O_F)=q_v^j=|\alpha-\beta|_v^{-1}.\]
 Now let $D=\det\g$ and $r=\tr\g$.  Note that
\begin{equation}\label{Deltag}
4D=4\alpha\beta=(\alpha+\beta)^2-(\alpha-\beta)^2=r^2-(\alpha-\beta)^2.
\end{equation}
Therefore
\[\Phi(\g,\fv)=|\alpha-\beta|_v^{-1}=|r^2-4D|_v^{-1/2},\]
as claimed.
\end{proof}

\subsection{Local orbital integrals at primes $\ell\nmid N$: elliptic case}\label{gell}

If $\g$ is elliptic over a field $F$ of characteristic $0$,
then $E=F[\g]$ is a quadratic field extension of $F$, and
\[G_\g(F)= E^*\]
(\cite[Prop. 26.1]{KL}).
The center $Z(F)$ is isomorphic to $F^*$.

\begin{proposition}\label{compact}
Let $F$ be a local field of characteristic $0$, and 
   suppose $\g$ is elliptic in $G(F)$. Then $G_\g(F)/Z(F)$ is compact.
\end{proposition}
\begin{proof}
    If $F=\R$, then $G_\g(\R)=\R[\g]^*\cong \C^*$, and the map $z\mapsto z/|z|$ gives
    rise to $\C^*/\R^*\cong \SO(2)/\{\pm 1\}$, which is compact.
    
    Now suppose that $F$ is nonarchimedean,
    with valuation $v$ and integer ring $\O_F$.
    Let $E=F[\g]$, and choose a prime element $\pi\in\O_E$.
    Then letting $e\in\{1,2\}$ be the ramification index of $E/F$,
    \begin{equation}\label{olGg}
        G_\g(F)/Z(F)\cong E^*/F^*=\bigcup_{j=0}^{e-1} \pi^j\O_E^*/\O_F^*,
    \end{equation}
    which is compact.
\end{proof}

Consider a $p$-adic field $F$, with all notation as in the previous subsection.
For $\g$ elliptic in $G(F)$, the above leads to the following natural choice 
of $G(F)$-invariant measure on the quotient space $\ol{G_\g(F)}\bs \olG(F)$.
We assign the compact group
$\ol{G_\g(F)}$ a total volume of $1$.  We assign $\olG(F)$ the Haar measure for which
$\olG(\O_F)$ has measure $1$.  Together these choices determine the quotient measure via
\[\int_{\ol{G_\g(F)}\bs\olG(F)}\int_{\ol{G_\g(F)}}\phi(xy)dx\,dy=\int_{\olG(F)}\phi(g)dg.\]
In fact, by our normalization, if $\phi$ is left $G_\g(F)$-invariant, then
\begin{equation}\label{ellint}
    \int_{\ol{G_\g(F)}\bs\olG(F)}\phi(y)dy=\int_{\olG(F)}\phi(g)dg
\end{equation}
when $\g$ is elliptic in $G(F)$.

For such $\g$, $E=F[\g]$ is a quadratic extension of $F$.  Fix an $F$-integral basis $\{1,\e\}$
for the ring of integers $\O_E$, so
\begin{equation}\label{O_E}\O_E=\O_F+\O_F\e.
\end{equation}
We will need some facts about orders and lattices in $E$.
Recall that an {\bf $F$-order} in $E$ is a subring containing $\O_F$ which has rank $2$ as an $\O_F$-module.

\begin{proposition}\label{orders}
 Let $\mathfrak{O}_{E/F}$ denote the set of all $F$-orders in $E$.
For $r\ge 0$ and $\e$ as above, define
\[\O_r=\O_F+\p^r\e,\]
where $\p$ is the maximal ideal of $\O_F$, so in particular $\O_0=\O_E$.  Then 
\[\mathfrak{O}_{E/F}=\{\O_r|\, r\ge 0\}.\] 
Furthermore, letting $e=e(E/F)$ be the ramification index, for $r>0$  we have
\begin{equation}\label{index}
[\O_E^*:\O_r^*] 
= \begin{cases}q_v^r&\text{if }e=2\\
q_v^r+q_v^{r-1}&\text{if }e = 1.\end{cases} 
\end{equation}
\end{proposition}

\begin{proof}
 See also \cite[\S6.6-6.7]{Mi} for the case $F=\Q_p$.  
Here we loosely follow Okada \cite[\S2.3]{O}.  Clearly $\O_r\in\mathfrak{O}_{E/F}$.  Conversely,
let $\O\in \mathfrak{O}_{E/F}$.  The elements of $\O$ are integral over $E$ (\cite[Prop. I.2.2]{N})
so $\O\subset \O_E$.  %Further, as $\O$ is an $\O_F$-module and $\O_F$ is a PID, 
  Hence there exists $\alpha\in\O\subset \O_E$ such that
\[\O=\O_F+\O_F\alpha.\]
Since $\alpha\notin\O_F$, by topological considerations we see that there exists $r\ge 0$
such that 
$\alpha\in \O_F+ \varpi^r\O_E=\O_r$ but $\alpha\notin \O_F+\varpi^{r+1}\O_E=\O_{r+1}$.
Hence
\[\O_{r+1}\subsetneq \O\subset \O_r.\]
We see easily that $\O_r/\O_{r+1}\cong \p^r/\p^{r+1}\cong \O_F/\p$ as $\O_F$-modules.  
Since the latter is 
$1$-dimensional as a vector space over $\O_F/\p$, it has no nonzero proper submodules.
  It follows that $\O=\O_{r}$. 

For the second part, consider the sequence
\[1\longrightarrow \O_F^*/(1+\p^r)\longrightarrow \O_E^*/(1+\p^r\O_E)\longrightarrow \O_E^*/\O_r^*
  \longrightarrow 1,\]
where the maps are the obvious ones.  It is straightforward to check that the sequence
is exact.  Therefore
\[[\O_E^*:\O_r^*] = \frac{|\O_E^*/(1+\p^r\O_E)|}{|\O_F^*/(1+\p^r)|}.\]
Let $e=e(E/F)$, so that $\p\O_E=\mathfrak{P}^e$, where $\mathfrak{P}$ is the maximal ideal
of $\O_E$.  Then
\[|\O_E^*/(1+\p^r\O_E)| =|\O_E^*/(1+\mathfrak{P}^{er})| 
  =[\O_E^*:1+\mathfrak{P}]\prod_{j=2}^{er}[1+\mathfrak{P}^{j-1}:1+\mathfrak{P}^j]
 =(q_E-1)q_E^{er-1}\]
(\cite[p. 139]{N}).  Here,
\[q_E=|\O_E/\mathfrak{P}| =\begin{cases}q_v&\text{if }e=2\\q_v^2&\text{if }e=1.\end{cases}\]
Likewise $|\O_F^*/(1+\p^r)|=(q_v-1)q_v^{r-1}$, and \eqref{index} follows immediately.
\end{proof}
 
For the purposes of this subsection, a {\bf lattice} in $F^2=F\times F$ is an $\O_F$-submodule
of rank 2. The group $F^*$ acts by multiplication on the set of lattices, and
    the orbits are called {\bf lattice classes}.
    The map $g\mapsto L=g\zpvector$ from $G(F)$ to the set of lattices in $F^2$
    induces a bijection between $\olG(F)/\ol{K_v}$ and the
    set of lattice classes, since $K_v$ is the stabilizer of $\zpvector$.

    With notation as in \eqref{O_E}, 
 we may identify a lattice $L\subset F^2$ with
    the lattice $\row 1{\e}L\subset E$, so that in particular $\zpvector$ is identified
    with $\O_E$.
Given $\eta\in E^*$, it acts by scalar multiplication on the set of lattices in $E$, and
  by matrix multiplication (via $E=F[\g]$) on the lattices in $F^2$.  In general, these actions
  are not compatible with the above identification.  % of lattices in $F^2$ with lattices in $E$.
   However, as shown in \cite[Lemma 26.20]{KL}, after possibly replacing $\g$ (or equivalently, $\e$)
   by a $G(F)$-conjugate, these two actions do coincide for all $\eta\in E^*$.
    Explicitly, for any $g\in G(F)$,
    \[\eta \row1{\e}g\zpvector = \row 1\e\eta g\zpvector,\]
    where on the left $\eta$ acts as a scalar via $\eta\row1{\e}=\row{\eta}{\eta\e}$,
  and on the right it is acting by matrix multiplication.
    We will assume that $\g$ is chosen in this way, as we may since the value of the orbital
    integral depends only on $\g$'s conjugacy class in $G(F)$.
    
  We associate to any lattice $L\subset E$ the order
    \[\O_L=\{\mu\in E|\, \mu L\subset L\}.\]
    This depends only on the lattice class to which $L$ belongs.
    Since $E$ is local, every lattice in $E$ is principal in the sense that there 
    exists $y\in E^*$ such that $yL=\O_L$. (One may adapt the proof of 
  \cite[Prop. 26.13]{KL}, which follows \cite{La}).

Given an order $\O$, $L$ is a {\bf proper $\O$-lattice} if $\O_L=\O$.
    Two proper $\O$-lattices $y\O$ and $z\O$ (for $y,z\in E^*$) are equal if and only if
    $y/z\in \O^*$.  Therefore the set of all proper $\O$-lattices corresponds 
bijectively with $E^*/\O^*$.  

\begin{lemma}\label{fnplem}
Suppose $(\det \g)\O_F=\mfn$ and $g\in G(F)$.  Then for $\fv$ given by \eqref{fv},
$\fv(g^{-1}\g g)\neq 0$ if 
and only if $\g\in \O_L$ for $L=g\zpvector$.
\end{lemma}
\begin{proof}
We observe that 
\begin{align*}\g\in \O_L&%\iff \g\O_L\subset \O_L
\iff \g L \subset L
\iff g^{-1}\g g\zpvector\subset \zpvector\\
&\iff g^{-1}\g g\in M_2(\O_F).
\end{align*}
Given that $\ord_v(\det\g)=\ord_v(\mfn)$, the above is equivalent to $g^{-1}\g g$ belonging
to the support $Z(F) M(\mfn)$ of $\fv$.
\end{proof}

\begin{proposition}\label{elle}
    Let $\fv$ be given by \eqref{fv}.
    Then for $\g\in G(F)$ elliptic, the orbital integral
    \[\Phi(\g,\fv)=\int_{\ol{G_\g(F)}\bs \olG(F)}\fv(g^{-1}\g g)dg\]
    vanishes unless some conjugate of $\g$ lies in $Z(F) M(\mfn)$.
  Taking $\g\in M(\mfn)$, with measure normalized as in \eqref{ellint} we have
    \[\Phi(\g,\fv)=e_\g\sum_{r=0}^{n_\g}[\O_E^*:\O_r^*],\]
    where $E=F[\g]$ is the associated quadratic extension of $F$ with 
    ramification index $e_\g\in\{1,2\}$ and ring of integers $\O_E=\O_F+\O_F\e$,
    \[\O_r=\O_F+\p^r\e\]
    is the order of index $q_v^r$ inside $\O_E$, 
    and $n_\g\ge 0$ is defined by $\O_\g = \O_F+\O_F\g =\O_r$ for $r=n_\g$.
    In particular, if $\O_\g=\O_E$ and $\p$ is inert in $E$, 
    then $\Phi(\g,\fv)=1$.
\end{proposition}

\noindent{\em Remarks:} 1. Let $P_\g(X)\in \O_F[X]$ be the characteristic polynomial of $\g$.  
If $P_\g$ is irreducible modulo $\p$, then $e_\g=1$ and $\O_\g=\O_E$,
\cite[p. 18]{S}. Hence $\Phi(\g,\fv)=1$ in this case.
\vskip .1cm

2. The index $[\O_E^*:\O_r^*]$ is given explicitly in \eqref{index} when $r>0$ (and is $1$ when $r=0$).

3. Let $\mathfrak{d}_{E/F}=\det\mat1\e1{\ol\e}^2\O_F$ be the relative discriminant (with the bar
denoting Galois conjugation),
write $\g=s+b\e$ for $s,b\in\O_F$, and let $\Delta_\g=r^2-4D$ be the discriminant of $\g$.  Then
\begin{equation}\label{ngb}
n_\g = \ord_v(b) = \frac{\ord_v(\Delta_\g)-\ord_v(\mathfrak{d}_{E/F})}2.
\end{equation}
This follows from the fact that the relative discriminant of 
\[\O_\g=\O_F+\O_F\g=\O_F+\O_Fb\e\] is
given on the one hand by
\[\det\mat 1{b\e}1{b\ol\e}^2\O_F = b^2\mathfrak{d}_{E/F},\] 
and also (using \eqref{Deltag}) by
\[\det\mat 1\g1{\ol\g}^2\O_F =(\g - \ol{\g})^2\O_F =\Delta_\g\O_F.\]
Further, if $F$ is the completion of a number field $L$ at a place $v$, $\{1,\e_L\}$ 
is an integral basis of $L[\g]$ over $L$, and we write $\g=s_L+b_L\e_L$, then equation \eqref{ngb}
also holds with $b_L$ in place of $b$.  Indeed the same argument applies in the global case
to give $b_L^2\mathfrak{d}_{L[\g]/L} =\Delta_\g\O_L$. By the fact
that the global discriminant is the product of the local ones and (due to $\g$ being elliptic in $G(F)$) there
is only one prime of $L[\g]$ lying over $v$, we see that $\ord_v(b_L)=\ord_v(b)$.

4. If $E=\Q_\ell[\sqrt{d}]$ for $d\in\Z$ square-free, then 
(see \cite[\S6.10]{M}, for example)
\begin{equation}\label{OE}
\O_E=\begin{cases}
\Z_2[\frac{1+\sqrt{-3}}2]&\text{if }\ell=2, E=\Q_2[\sqrt{-3}]\\
\Z_\ell[\sqrt{d}]&\text{otherwise.}\end{cases}\end{equation}
In particular, if $\ell>2$ and the valuation $\alpha=v_\ell(\Delta_\g)$
of the discriminant of $\g$ is {\em odd}, then $e_\g=2$, $n_\g=(\alpha-1)/2$, and assuming $\g\in M(\n)_\ell$,
\begin{equation}
\Phi(\g,\fnp)=2\sum_{r=0}^{(\alpha-1)/2}\ell^r.
\end{equation}
\vskip .2cm

\begin{proof}[Proof of Proposition \ref{elle}]
The first statement is clear.  Now suppose $\g\in M(\mfn)$.  By \eqref{ellint},
    \[\Phi(\g,\fv)=\int_{\olG(F)}\fv(g^{-1}\g g)dg.\]
    The integrand is right $\ol{K_v}$-invariant as a function of $g$.
    Since $\ol{K_v}$ is open with
    measure $1$, $\olG(F)/\ol{K_v}$ is discrete with the counting measure.
    Therefore
    \[\Phi(\g,\fv)=\sum_{g\in \olG(F)/\ol{K_v}} \fv(g^{-1}\g g).\]
By our earlier remarks, we can view the sum
as a sum over the lattice classes, and by Lemma \ref{fnplem}, 
$\Phi(\g,\fv)$ is equal to the number of lattice classes preserved by $\g$.

    Since $\g\in E$ is integral over $\O_F$, $\O_\g=\O_F+\O_F\g$ is an order in $E$ 
(cf. \cite[Lemma 26.10]{KL}).  
    We claim that $\g\O_r\subset \O_r$ if and only if
    $0\le r\le n_\g$, where $q_v^{n_\g}$ is the index of $\O_\g$.
    Indeed, 
\[\g\O_r\subset\O_r\iff \g\in \O_r\iff \O_\g\subset\O_r\iff r\le n_\g.\]
    It follows that 
\[\Phi(\g,\fv)=\sum_{r=0}^{n_\g} (\#\text{ of classes of proper $\O_r$-lattices}).\]
  Recall from earlier that the set of proper $\O_r$-lattices corresponds bijectively with 
  $E^*/\O_r^*$. 
Since we are counting $F^*$-classes of lattices
    rather than lattices themselves, we find
    \[\Phi(\g,\fv)=\sum_{0\le r\le n_\g}|E^*/F^*\O_r^*|\qquad (\O_\g=\O_{n_\g}).\]
Because $\O_F^*\subset \O_r^*$, it follows from \eqref{olGg} that
    $|E^*/F^*\O_r^*|=e_\g[\O_E^*:\O_r^*]$, where $e_\g\in\{1,2\}$ is the ramification index
of $E/F$. The result now follows.
\end{proof}

\begin{corollary}\label{g2}
 For $\fv$ as in \eqref{fv}, 
let $\g\in M(\mfn)$ have characteristic polynomial $P_\g(X)=X^2-rX+D\in\O_F[X]$
with discriminant $\Delta_\g=r^2-4D$.
Then if $\g$ is hyperbolic in $G(F)$, $\Phi(\g,\fv)=|\Delta_\g|_v^{-1/2}$.
If $\g$ is elliptic in $G(F)$ and $P_\g(X)$ does not have a double root
in $\O_F/\p$, then $\Phi(\g,\fv)=1$.

Consequently, for $\g\in M(\mfn)$ elliptic or hyperbolic in $G(F)$,
$\Phi(\g,\fv)=1$ if 
\[\Delta_\g\notin\p.\]
\end{corollary}
\begin{proof} 
The hyperbolic case is just a restatement of Proposition \ref{ellh}.
Suppose $\g$ is elliptic. If $P_\g$ does not have a double root in $\O_F/\p$,
then it cannot have a simple root either, because otherwise that root would lift to a root
in $F$ by Hensel's Lemma.  By the first remark after Proposition \ref{elle},
$\Phi(\g,\fv)=1$.

Furthermore, suppose $\p\nmid 2$, and note that $P'_\g(X)=2X-r$ vanishes only at
   $r/2 \in \O_F/\p$.
On the other hand,
\[P_\g(r/2)= D-\frac{r^2}4,\]
which shows that $P_\g$ has a repeated root modulo $\p$ 
if and only if $\p|(r^2-4D)$. Hence when $\p\nmid 2$ and $\Delta_\g\notin\p$, $\Phi(\g,\fv)=1$.

If $\p|2$ and $(r^2-4D)\notin\p$, then $r\in\O_F^*$, and therefore $P_\g'(X)=2X-r$ is nonzero
mod $\p$.  Hence $P_\g$ does not have a repeated root, and $\Phi(\g,\fv)=1$ in this case as well.
\end{proof}

Although the result of Proposition \ref{elle} appears complicated, it is not so hard
to evaluate it by hand, using the remarks that follow the proposition and standard results
about quadratic extensions of $p$-adic fields.

\begin{example}\label{Nell}
Let $\ell$ be a prime not dividing $D$, and let 
$\g=\smat0{-D}10$.
Then for $f_\ell$ as in \eqref{fell},
    \[\Phi(\g,f_\ell) = \begin{cases} 2&\text{if }\ell=2\text{ and }D\equiv 1,5,7\mod 8\\
        4&\text{if }\ell=2\text{ and }D\equiv 3\mod 8\\
    1&\text{if $\ell\neq 2$.}\end{cases}\]
\end{example}
\noindent{\em Remark:} Some additional examples are given in \S\ref{ng1}.
\begin{proof}
    First suppose $\ell\neq 2$.  Since the discriminant $-4D$ of $P_\g(X)=X^2+D$ is not
divisible by $\ell$,  % has discriminant $-4N$ and $\ell\nmid 2N$,
$\Phi(\g,f_\ell)=1$ by Corollary \ref{g2}.

 Now suppose $\ell=2$, so $D$ is odd since $\ell\nmid D$. 
   Recall that the squares of 
    $\Q_2^*$ are exactly the elements of the set $2^{2\Z}(1+8\Z_2)$
    (\cite[Theorem II.4]{Sc}).  Thus
$-D$ is a square in $\Q_2^*$ if and only if
\[D\equiv 7\mod 8.\]
When this congruence is satisfied, $\g$ is hyperbolic, and by Corollary \ref{g2},
\[\Phi(\g,f_2)=|-4D|_2^{-1/2}=2.\]

Now suppose that $-D$ is not a square in $\Q_2$, i.e., it is not $1\mod 8$.
We recall some facts about the quadratic extensions of $\Q_2$ (see e.g. \cite[Ch. 6]{M}).
There are exactly seven such extensions, namely $\Q_2[\sqrt{d}]$ for
\[d=-1, \pm 3, \pm 2,\pm 6,\]
with $\Q_2[\sqrt{-3}]$ being the unique unramified quadratic extension.
With the exception of $d=-3$, the ring of integers is $\Z_2[\sqrt{d}]$.  For $d=-3$,
the ring of integers is $\Z_2[\frac{1+\sqrt{-3}}2]$.
Under the given hypothesis, $-D\equiv d\mod 8$, where $d\in\{-1,\pm3\}$. 
So $-D=dx$ for some $x\in 1+8\Z_2$, and hence
$-D=dy^2$ for some $y\in\Z_2^*$.
Therefore, writing $E=\Q_2[\sqrt{-D}]$, we have $\O_E=\O_\g$ unless $d=-3$.
In the former case, $E/\Q_2$ is ramified, so by Proposition \ref{elle},
\[\Phi(\g,f_2)=2\qquad (D\equiv 1,5 \mod 8).\]
If $D\equiv 3\mod 8$, then 
$\O_E=\Z_2+\Z_2\e$ for $\e=\frac{1+\sqrt{-3}}2$.  Hence
\[\O_\g=\Z_2+\Z_2\sqrt{-D}=\Z_2+\Z_2\sqrt{-3}=\Z_2+\Z_2 2\e.\]
So in the notation of Proposition \ref{elle}, $n_\g=1$. 
Since $E/\Q_2$ is unramified, using \eqref{index}, we have
    \[\Phi(\g,f_2)=[\O_E^*:\O_E^*]+[\O_E^*:\O_\g^*] = 1+3 =4.\qedhere\]
\end{proof}

\subsection{The set of relevant $\g$}\label{relsec}

Here we determine explicitly the finite set of conjugacy classes in $\olG(\Q)$ that can have a nonzero
contribution to the trace of $R(f)$ for $f$ as in \eqref{fn}.
Writing $N=\prod_{p|N}p^{N_p}$, define the square-free integers
\[S=\prod_{p|N,\atop{N_p\text{ even}}}p,\quad T=\prod_{p|N,\atop{N_p\text{ odd}}}p.\]
We say that an elliptic element $\g\in G(\Q_p)$ is {\bf unramified} (at $p$)
if $v_p(\det \g)$ is even, and {\bf ramified} otherwise.

\begin{lemma}\label{detnp}
Let $\g\in G(\Q)$ be elliptic, and suppose $\Phi(\g,f)\neq 0$ for $f=f^\n$ as in \eqref{fn}.
Then there exists a unique positive divisor $M|T$ and a scalar $z\in\Q^*$ such that
$\tr(\g z)\ge 0$ is an integer and
\[\det(z\g)=\n M.\]
In particular, the rational canonical form of $z\g$ lies in $M_2(\Z)$. 
\end{lemma}

\begin{proof}
If $p|S$, then $\g$ is unramified at $p$ since $f_p$ is supported in $Z_pK_p$.  For $p|T$, 
the support of $f_p$ has both ramified and unramified elements (cf. \eqref{max}).
Let $M$ be the product of those primes $p|T$ at which $\g$ is ramified.
For each prime $\ell\nmid N$, some conjugate of $\g$ must lie in $\Supp(\fnp)
    =Z_\ell M(\n)_\ell$ since otherwise the integrand of $\Phi(\g,f)$ vanishes.
It follows that $v_p(\frac{\det\g}{\n M})$ is even for {\em all} primes $p$, where $v_p$ is
 the $p$-adic valuation.  Hence
$\det\g\in \pm \n M\Q^{*2}$, where $\Q^{*2}$ is the set of squares in $\Q^*$. 
 Because $f_\infty$ is supported on $G(\R)^+$, there is a
 scalar $z\in\Q^*$ such that $\det(z\g)=\n M$, as claimed.
 Because $\Phi(z\g,f)=\Phi(\g,f)\neq 0$, some $G(\Af)$-conjugate of $z\g$ lies in 
\begin{equation}\label{dom}
\prod_{p|M}\smat{}1p{} K_p\times \prod_{p|\frac {ST}M}K_p\times \prod_{\ell\nmid N}M(\n)_\ell\subset M_2(\Zhat)
\end{equation}
(recall that $f_p$ is supported in the group $J$ of \eqref{max}).
 In particular, $\tr(z\g) \in \Zhat\cap \Q=\Z$.
Scaling $z$ by $-1$ if necessary, we may arrange further that $\tr(z\g)\ge 0$.
\end{proof}

\begin{lemma}\label{pd} Let $F$ be a $p$-adic field, and $\g$ an elliptic element of $G(F)$ 
    with $\tr\g\in\O_F$ and $\det\g\in \p$.
    Then $\tr \g\in \p$. 
\end{lemma}
\begin{proof}
    Denote the characteristic polynomial of $\g$ by
    \[P_\g(X)=X^2-dX+\det\g,\]
    where $d=\tr\g$.  Notice that $P_\g(0)\equiv 0\mod \p$.
    Furthermore, $P_\g'(0)\equiv - d\mod \p$.  If $d$ is nonzero modulo $\p$, then by Hensel's
    Lemma, $P_\g$ has a root in $\p$, contradicting the fact that $\g$ is elliptic
    in $G(F)$.  Hence $d\in \p$.
\end{proof}

\begin{proposition}\label{rel1}
 For $\g\in \olG(\Q)$ elliptic, and $f=f^\n$ the test function defined in \eqref{fnt},
 $\Phi(\g,f)=0$ unless the conjugacy class of $\g$ has a representative in $G(\Q)$
 of the form $\mat0{-\n M}1{rM}$ for some $M|T$ and $0\le r<\sqrt{\frac{4\n}M}$.
\end{proposition}
\noindent{\em Remark:}
If the characteristic polynomial of $\g=\smat0{-\n M}1{rM}$ has a root in $\Q_p$, then
  $\Phi(\g,f)=0$ by Proposition \ref{ellg}.

\begin{proof}
Let $\mathfrak o$ be an elliptic conjugacy class in $\olG(\Q)$ with
 $\Phi(\mathfrak o,f)\neq 0$. 
By Lemma \ref{detnp},
$\mathfrak o$ has a unique representative $\g\in G(\Q)$ with characteristic polynomial of the form
\[P_\g(X)=X^2-dX+\n M\in \Z[X],\]
where $d=\tr\g\ge 0$ and $M|T$.
By Proposition \ref{ellg}, we know that $\g$ is elliptic in $G(\Q_p)$ for each $p|N$
and also in $G(\R)$. It follows by Lemma \ref{pd} that $M|d$.
Write $d=rM$.
Given that $\g$ is elliptic in $G(\R)$, we have $d^2<4\n M$, i.e., 
\[r^2M < 4\n.\]
So, taking $\g$ in rational canonical form as we may, it has the form
\[\g =\mat 0{-\n M}1{rM},\quad 0\le r<\sqrt{\frac{4\n}{M}}.\qedhere\]
\end{proof}

\subsection{The measure of $\ol{G_\g(F)}\bs\ol{G_\g(\A_F)}$.}\label{msec}

Let $F$ be a number field with adele ring $\A_F$, and let $\g$ be an elliptic element of $G(F)$.
With $G_\g$ the centralizer of $\g$ in $G$, here we will compute the measure of
$\ol{G_\g(F)}\bs \ol{G_\g(\A_F)}$.  The result is given in Theorem \ref{meas} below.
A related discussion can be found in \cite[\S5]{Go}.

The basic idea is straightforward: we know that
 $G_\g(\A_F)=\A_F[\g]^* = \A_E^*$, where $E=F[\g]$ is
 a quadratic extension of $F$. (The proof of this fact given in \cite[Prop. 26.1]{KL} 
for $F=\Q$ applies to any number field.)
The center of $G(\A_F)$ is isomorphic to $\A_F^*$, so 
\begin{equation}\label{Ggisom}
\ol{G_\g(\A_F)} \cong \A_F^*\bs \A_E^*
\end{equation}
topologically and algebraically.  Finally, $G_\g(F)\cong F[\g]^*=E^*$ by the same reference, so
\begin{align*}
    \ol{G_\g(F)}\bs\ol{G_\g(\A_F)}&=\A_F^*E^*\bs\A_E^*\cong (F^*\bs\A_F^*)\bs (E^*\bs \A_E^*)\\
       &\cong (F^*\bs \A_F^1)\bs (E^*\bs \A_E^1),
\end{align*}
where the superscript $1$ indicates ideles of norm 1, and the latter 
isomorphism comes from modding out by an embedded copy of
$\R^+$ in $\A_F^*\subset \A_E^*$.
For any number field $L$ the measure of $L^*\bs \A_L^1$ 
is computed in Tate's thesis under suitable normalization, which we may use with $L=E,F$
 to obtain the measure of the above space.
However, as will be seen, we need to be very careful about the normalization of measures,
particularly in the last step.

\subsubsection{Quotient measure}

Recall that if $H<G$ are unimodular locally compact groups with Haar measures 
  $\mu_H$ and $\mu_G$ and $H$ closed in $G$, there is a unique left $G$-invariant 
 quotient measure $\mu_{G/H}$ on $G/H$ satisfying
\[\int_{G/H} \left[\int_H f(gh)d\mu_H(h)\right] d\mu_{G/H}(g) = \int_G f(g)d\mu_G(g)\]
for all $f\in C_c(G)$.

\begin{lemma}\label{modT}
Let $H,K$ and $T$ be unimodular locally compact groups, with Haar measures $\mu_T,\mu_H,\mu_K$
  respectively.  Assume that $H<K$, and let $G=T\times K$ and $J=T\times H$.
Then relative to the product measures $\mu_G=\mu_T\times \mu_K$ and $\mu_J=\mu_T\times \mu_H$,
we have $\mu_{G/J}=\mu_{K/H}$ on the group $G/J\cong K/H$. 
\end{lemma}
\begin{proof} For $f\in C_c(G)$, 
\[\int_{K/H}\left[\int_Jf(xy)d\mu_J(y)\right]d\mu_{K/H}(x) = 
\int_{K/H}\left[\int_H\int_Tf(xht)d\mu_T(t)d\mu_H(h)\right]d\mu_{K/H}(x)\]
\[= \int_K\left[\int_Tf(kt)d\mu_T(t)\right]d\mu_{K}(k) = \int_Gf(g)d\mu_G(g).\qedhere\]
\end{proof}

\subsubsection{A volume from Tate's thesis}\label{Tatevol}

Let $L$ be a number field with adele ring $\A_L=\prod'_v L_v$, where $v$ ranges over the
places of $L$.  In Tate's thesis, measures $\mu_v$ on the local multiplicative
groups $L_v^*$ are normalized as follows.
If $v$ is real,
\begin{equation}\label{Rmu}
d\mu_v(x)=\frac{dx}{|x|}
\end{equation}
for $x\in \R^*$.  If $v$ is complex,
\begin{equation}\label{Cmu}
d\mu_v(z)=2\frac{dx\,dy}{x^2+y^2}=\frac 2r dr d\theta
\end{equation}
for $z=x+iy=re^{i\theta}\in\C^*$. Finally, at a nonarchimedean place $v$, $\mu_v$ is the
Haar measure on $L_v^*$ satisfying
\begin{equation}\label{ND}
\mu_v(\mathcal{O}_v^*) = (\N\mathfrak{D}_v)^{-1/2},
\end{equation}
where $\O_v$ is the ring of integers of $L_v$, $\mathfrak{D}_v$ is the different 
of $L_v$ and $\N\mathfrak{D}_v=|\O_{v}/\mathfrak{D}_v|$.
Taking the restricted product of the above local measures, we obtain a Haar measure
\[\mu_L={\prod_v}' \mu_v\]
 on $\A_L^*$.  

Let $L_\infty^*=\prod_{v|\infty}L_v^*$; we embed it into $\A_L^*$ by taking $1$'s at the
nonarchimedean components.  We embed $\R^+$ into $L_\infty^*$ and hence into $\A_L^*$ via
\[\lambda(t)= (t^{1/n},t^{1/n},\ldots,t^{1/n}),\]
where $n=n_L=[L:\Q]$.  Then if $L$ has $r_1$ real embeddings and $2r_2$ complex embeddings, 
for $t\in \R^+$ we have
\[|\lambda(t)|_{\A_L}=\prod_{v|\infty}|t|_v^{1/n} = t^{\frac{r_1+2r_2}n} = t\]
(recall that in the ideles we take the square of the usual absolute value at the complex places).

Let $T\cong\R^+$ denote the image of the map $\lambda$.  We give it the Haar measure $dt/t$. 
We have
\begin{equation}\label{AL1}
\A_L^*\cong T\times \A_L^1,
\end{equation}
where $\A_L^1$ is the subgroup consisting of ideles of norm $1$.
There is a unique measure $\mu_L^1$ on $\A_L^1\cong \A_L^*/T$ such that
\[\mu_L = \tfrac{dt}t \times \mu_L^1.\]

The multiplicative group $L^*$ embeds diagonally
in $\A_L^*$ as a discrete subgroup,  and by the product formula, $L^*\subset \A_L^1$.

\begin{theorem}\label{Tm}
\cite[Theorem 4.3.2]{T}. The group $L^*$ is discrete and cocompact in $\A_L^1$.  
Giving $L^*$ the counting measure, for $\mu_L^1$ as above we have
\[\mu_L^1(L^*\bs \A_L^1) = \frac{2^{r_1}(2\pi)^{r_2} h(L) R_L}{|d_L|^{1/2}w_L},\] 
where $h(L), R_L, d_L$ and $w_L$ are the class number, regulator, discriminant, and
number of roots of unity of $L$, respectively.
\end{theorem}
\noindent{\em Remark:} This is the residue of the Dedekind zeta function of $L$
 at $s=1$.

\subsubsection{Haar measure for orbital integrals}\label{etasec}

Let $\g\in G(F)$ be an elliptic element.  Here we define a Haar measure $\eta$ on 
$\ol{G_\g(\A_F)}$ which is convenient to use for
computing the elliptic orbital integrals.
Given a nonarchimedean place $v$ of $F$, $\g$ is necessarily either elliptic or hyperbolic 
in $G(F_v)$.  
We select a compact open subgroup $H_v$ of $\ol{G_\g(F_v)}=Z(F_v)\bs G_\g(F_v)$
 as follows. 
If $\g$ is elliptic in $G(F_v)$, then the full group is compact by Proposition \ref{compact},  
and we take $H_v=\ol{G_\g(F_v)}$.
If $\g$ is hyperbolic in $G(F_v)$, then $G_\g(F_v)$ is
conjugate to the diagonal subgroup $M(F_v)$.  In this case we define $H_v$ to be
the subgroup of $\ol{G_\g(F_v)}$ taken by this conjugation to $\ol{M}(\O_v)\cong \O_v^*$, where
 $\O_v$ is the ring of integers of $F_v$.

Next, we choose a local Haar measure $\eta_v$ on $\ol{G_\g(F_v)}$ for each place $v$ of $F$
 as follows.
If $v\nmid \infty$, we normalize $\eta_v$ so that $\eta_v(H_v)=1$.  
If $v|\infty$ is a real place of $F$ and $\g$ is elliptic over $F_v$, we take
$\eta_v(\ol{G_\g(F_v)})=1$.  If $v|\infty$ and  $\g$ is hyperbolic over $F_v$, 
then $\ol{G_\g(F_v)}\cong M(F_v)/F_v^*\cong F_v^*$, and we give it the measure 
$d\eta_v(x)=d\mu_v(x)$ for $\mu_v$ as in \eqref{Rmu} or \eqref{Cmu}.\footnote[2]{ 
With $F=\Q$, these are the measures that are used in the local
orbital integral calculations in the present paper.
See \S \ref{ghyp}-\ref{gell} for finite $\ell\nmid N$ and \cite[\S26.2]{KL} for the
  $\ell=\infty$ calculation yielding \eqref{Rell}.  For $\ell|N$, in \S \ref{pN} we 
will use the same measure used in \S\ref{gell}.}

Note that
\[\ol{G_\g(\A_F)}={\prod_{v}}'\,\ol{G_\g(F_v)},\]
where the product is restricted relative to the subgroups $H_v$.
We let $\eta$ denote the Haar measure on $\ol{G_\g(\A_F)}$ which is the restricted product
of the above local measures $\eta_v$.

As explained in \eqref{Ggisom}, for $E=F[\g]$ we have
\[\ol{G_\g(\A_F)} \cong \A_E^*/ \A_F^*.\]
So another natural measure on $\ol{G_\g(\A_F)}$ is the quotient measure $\mu_{E/F}$
  coming from the Haar measures
$\mu_E$ and $\mu_F$ on $\A_E^*$ and $\A_F^*$ obtained by taking 
  $L=E$ and $L=F$ respectively in \S\ref{Tatevol}.

Let us next determine the constant relating the two measures $\eta$ and $\mu_{E/F}$.
For a place $v$ of $F$ and a place $w$ of $E$ lying over $v$, we have defined the 
measures $\mu_v$ and $\mu_w$ on $F_v^*$ and $E_w^*$ in \S\ref{Tatevol}.
We let $\mu_v'=\prod_{w|v}\mu_w$ be the product measure on $E_v^* = \prod_{w|v}E_w^*$,
and define $\ol{\mu}_v'$ to be the corresponding quotient measure on 
  $E_v^*/ F_v^*\cong \ol{G_\g(F_v)}$.
Then $\mu_{E/F}=\prod'_v \ol{\mu}_v'$ where $v$ runs over the places of $F$.
For each $v$ we need to find the constant relating $\eta_v$ to $\ol{\mu}_v'$.

Let $v$ be a nonarchimdean place of $F$.
Suppose $\g$ is hyperbolic in $G(F_v)$, so that $\eta_v(\ol{M}(\O_v))=1$.
Let $w,\ol{w}$ be the primes of $E$ lying over $v$.  Then 
\[E_v := E\otimes F_v\cong E_w\oplus E_{\ol w},\]
and $\mu_v'=\mu_w\times \mu_{\ol{w}}$ on 
\[G_\g(F_v)\cong E_v^*\cong E_w^*\times E_{\ol w}^*\cong F_v^*\times F_v^*.\]
Hence
\[\mu_v'(\O_w^*\times\O_{\ol w}^*)=\mu_w(\O_w^*)\mu_{\ol w}(\O_{\ol w}^*)=(\N\mathfrak{D}_w)^{-1/2}
 (\N\mathfrak{D}_{\ol w})^{-1/2}\]
 by \eqref{ND}.  (This is in fact equal to $\N\mathfrak{D}_v$, but we prefer to leave it
unsimplified for global reasons.)
 Likewise, the diagonally embedded subgroup $F_v^*\subset E_v^*$ has measure 
  $\mu_v(\O_v^*)=(\N\mathfrak{D}_v)^{-1/2}$.
 Therefore the quotient measure $\ol{\mu}_v'$ on $E_v^*/ F_v^*\cong\ol{G_\g(F_v)}$ 
 gives the open subgroup $(\O_w^*\times \O_{\ol w}^*)/\O_v^*\cong \ol{M}(\O_v)$
  the measure $\frac{(\N\mathfrak{D}_w)^{-1/2}(\N\mathfrak{D}_{\ol w})^{-1/2}}
{(\N\mathfrak{D}_v)^{-1/2}}.$
Consequently,
\[\eta_v=\frac{(\N\mathfrak{D}_w)^{1/2}(\N\mathfrak{D}_{\ol w})^{1/2}}
{(\N\mathfrak{D}_v)^{1/2}}\ol{\mu}_v'\]
for such $v$.

Now suppose $\g$ is elliptic in $G(F_v)$ (again with $v$ nonarchimedean).
  Then there is a unique valuation $w$ of $E$ extending $v$, and $E_w=F_v[\g]$ is a 
  quadratic extension of $F_v$.
Let $\O_{w}$ be its ring of integers, with 
a uniformizer $\varpi$.  Then for the ramification index $e_v=e(w/v)\in\{1,2\}$,
\[        \ol{G_\g(F_v)}\cong E_w^*/F_v^*
=\bigcup_{j=0}^{e_v-1} \varpi^j\O_{w}^*/\O_{v}^*\]
as in \eqref{olGg}.
By definition of the local component $\mu_w$ of $\mu_E$,
   $\mu_w(\O_{w}^*)=(\N\mathfrak{D}_w)^{-1/2}$.  
The local component of $\mu_F$ at $v$ gives $\meas(\O_v^*)=(\N\mathfrak{D}_v)^{-1/2}$.  
Therefore
the quotient measure $\ol{\mu}_v'$ satisfies
\[\ol{\mu}_v'(\O_{w}^*/\O_v^*)=\frac{(\N\mathfrak{D}_w)^{-1/2}}{(\N\mathfrak{D}_v)^{-1/2}}.\]
Since $\eta_v(\ol{G_\g(F_v)})=1$, it follows that
\[\eta_v = \frac1{e_v} \frac{(\N\mathfrak{D}_w)^{1/2}}{(\N\mathfrak{D}_v)^{1/2}}
\ol{\mu}_v'\]
 for such $v$.

Suppose now that $F_v=\R$ and $\g$ is elliptic in $G(F_v)$.  Then $E_w=\C^*$ and
$\ol{G_\g(F_v)}=\C^*/\R^*$. A set of representatives in $\C^*$
is $\{e^{i\theta}|\, \theta\in [0,\pi)\}.$  Since the measure $\mu_v(x)=\frac{dx}{|x|}$
 on $\R^*$ matches the factor $\frac{dr}r$ in $\mu_w(z)=\frac{2 dr\,d\theta}r$
 given in \eqref{Cmu}, it follows that
\[\ol{\mu}_v'(\C^*/\R^*) = 2\pi.\]
Since $\eta_v(\ol{G_\g(\R)})=1$,
\[\eta_v=\frac1{2\pi}\ol{\mu}_v'\]
for such $v$.

If $F_v=\R$ or $\C$ and $\g$ is hyperbolic in $G(F_v)$, then as in the analogous nonarchimedean
case, $E_v^*=E_w\times E_{\ol w}\cong F_v^*\times F_v^*$, and the quotient measure
on $\ol{G_\g(F_v)}\cong E_v^*/F_v^*\cong F_v^*$ is $\ol{\mu}_v'(x)=\mu_v(x)$.  In such
cases we have likewise defined $\eta_v=\mu_v$.  So 
$\eta_v=\ol{\mu}_v'$ for such $v$.

Putting everything together, we have shown that
\[\eta =\left[\prod_{v\nmid \infty}\frac1{e_v}\frac{\prod_{w|v}(\N\mathfrak{D}_w)^{1/2}}
{(\N\mathfrak{D}_v)^{1/2}}\right]\left[\prod_{v|\infty,\atop{\g \text{ elliptic in }G(F_v)}}\frac1{2\pi}\right]
\mu_{E/F}.\]
We can simplify using three well known facts from algebraic number theory (see, e.g., \cite[\S III.2]{N}): 
\begin{enumerate}
\item $e_v=2$ if and only if $\p_v|\mathfrak{d}_{E/F}$ where $\mathfrak{d}_{E/F}$ is the relative discriminant;
\item the absolute discriminant of a local field is the absolute norm of the different; 
\item the product of the local discriminants is the global discriminant.
\end{enumerate}
It follows that taking $d_F,d_E\in \Z$ to be the discriminants of $F$ and $E$ respectively,
\begin{equation}\label{etamu}
\eta = \frac{|d_E|^{1/2}}{|d_F|^{1/2}}\frac1{2^{\w_F(\mathfrak{d}_{E/F})}}
  \frac1{(2\pi)^{\alpha_\g}}\mu_{E/F},
\end{equation}
where $\w_F(\mathfrak{d}_{E/F})$ is the number of distinct prime factors of 
$\mathfrak{d}_{E/F}$ in  $\O_F$, and 
$\alpha_\g$ is the number of (real) archimedean places $v$ of $F$ for which $\g$ is elliptic
 in $G(F_v)$.

\subsubsection{The quotient measure}

We turn now to the quotient space whose measure we need to compute, namely
    $\ol{G_\g(F)}\bs\ol{G_\g(\A_F)}\cong E^*\A_F^*\bs \A_E^*\cong \A_E^*/\A_F^*E^*$.
We have defined the quotient measure $\mu_{E/F}$ on $\A_E^*/\A_F^*$.  By \eqref{AL1},
we have
\[\A_F^*= T\times \A_F^1,\qquad \A_E^* = T\times \A_E^1.\]
We regard $\A_F^*$ as a subset of $\A_E^*$, so $T$ is the set
\[T=\{(a,a,\ldots,a)\in E_\infty^* |a>0\}\subset \A_E^*.\]
We will use Lemma \ref{modT} to relate $\mu_{E/F}$ to the quotient measure on $\A_E^1/\A_F^1$
coming from the measures $\mu_E^1$ and $\mu_F^1$ defined below \eqref{AL1}.
Recall that $T$ is given the measure $d\mu_T(t)=\frac{dt}t$, where $t^{1/n_E}=a$ for
$n_E=[E:\Q]$.  In terms of the parameter $a$, 
\[d\mu_T(t)= {n_E} \frac{da}a.\]
Notice that this is {\em not} the measure given to $\R^+$ upon taking $L=F$ in \eqref{AL1},
which is ${n_F}\frac{da}a = \frac{n_F}{n_E}d\mu_T(t)$. In other words, for $\mu_T$ normalized
as above, $\mu_F^1$ is defined by
\[\mu_F =\frac1{[E:F]}\mu_T\times \mu_F^1.\]
Therefore
\[\mu_F =\mu_T\times \frac1{[E:F]}\mu_F^1
=\mu_T\times \frac12\mu_F^1.\]
Hence by Lemma \ref{modT},
the quotient measure $\mu_{E/F}$ on $\A_E^*/\A_F^*\cong \A_E^1/\A_F^1$
is the same as the quotient measure 
coming from $\mu_E^1$ and $\frac12\mu_F^1$.  We denote this quotient measure by
$\mu_{E/F}^1$.

Finally, taking the quotient by the discrete subgroup $E^*$ we have
\begin{equation}\label{muEF}\mu_{E/F}(\A_E^*/E^*\A_F^*) =
\mu_{E/F}^1((\A_E^1/E^*)/(\A_F^1E^*/ E^*))=
\frac{\mu_E^1(\A_E^1/E^*)}{\frac12\mu_F^1(\A_F^1/F^*)}.
\end{equation}
As a technical point, the measure on the disjoint union
\[\A_F^1E^*=\bigcup_{\alpha\in E^*/F^*}\A_F^1\alpha\]
is simply $\tfrac12\mu_F^1$ on each component since $E^*$ is given the counting measure.  This explains
why the quotient measure on $\A_F^1E^*/E^*$ is the same as $\tfrac12\mu_F^1$ on $\A_F^1/F^*$.
Applying Theorem \ref{Tm} and \eqref{etamu} to \eqref{muEF}, we immediately obtain the following.

\begin{theorem}\label{meas}
    Let $\g\in G(F)$ be an elliptic element, and let
    $\eta$ be the measure introduced in \S\ref{etasec}. Then for $E=F[\g]$,
    \[\eta(\ol{G_\g(F)}\bs\ol{G_\g(\A_F)})=\frac{2^{r_1(E)}(2\pi)^{r_2(E)}h(E)R_E}
    {2^{r_1(F)}(2\pi)^{r_2(F)}h(F)R_F}\cdot\frac{w_F}{w_E}\cdot
\frac 2{2^{\w_F(\mathfrak{d}_{E/F})}(2\pi)^{\alpha_\g}}\]
with notation as in Theorem \ref{Tm}, where $\w_F(\mathfrak{d}_{E/F})$ is the number
of distinct prime ideals of $\O_F$ dividing the relative discriminant $\mathfrak{d}_{E/F}$,
and $\alpha_\g$ is the number of (real) archimdean places $v$ of $F$ for which $\g$
is elliptic in $G(F_v)$. 

In the special case where $F=\Q$ and $E=\Q[\g]$ is quadratic imaginary, 
we have $\alpha_\g = 1$, $w_F=2$, $h(F)=R_E=R_F=1$, so
 \begin{equation}\label{etaQ}
\eta(\ol{G_\g(\Q)}\bs\ol{G_\g(\A)})=\frac{2h(E)}{w_E2^{\w(d_E)}}
\end{equation}
where $\w(d_E)$ is the number of distinct prime factors of the discriminant $d_E$.
\end{theorem}

\noindent With the above in place, the proof of Theorem \ref{stf2} is complete.

\section{The case $N=S^2T^3$: Proof of Theorem \ref{mainST}}\label{ex}

Henceforth, we will focus on the case where $N=S^2T^3$ for $S$ and $T$
relatively prime square-free integers. In order to prove Theorem \ref{mainST},
by Theorem \ref{stf2} we just need to compute the orbital integrals at the primes dividing $N$.
We begin in \S\ref{d0}-\ref{ssr}
by reviewing the construction of supercuspidals of conductor $p^2$ (depth zero case)
and of conductor $p^3$ (simple case), giving explicit formulas for 
the local test functions to be used.   In \S\ref{test} we outline the global setup, and then
compute the required orbital integrals in \S\ref{pN} to complete the proof.

\subsection{Depth zero supercuspidal representations}\label{d0}

Let $F$ be a $p$-adic field, with ring of integers $\O$, maximal ideal $\p=\varpi\O$, and
  residue field $\k=\O/\p$ of size $q$.  The supercuspidal representations
of $G(F)$ of minimal conductor are the so-called depth zero supercuspidals, with conductor $\p^2$.
They have the form $\sigma=\cInd_{ZK}^{G(F)}(\rho)$, where $\rho$ is a $(q-1)$-dimensional representation
of $K=G(\O)$ inflated from a cuspidal representation of $G(\k)$, and $\cInd$ denotes
compact induction.  Some of their 
properties are summarized below (see, e.g., \cite{KR} for more detail).  

Temporarily, write $G=G(\k)$.
Let $L$ be the unique quadratic extension of $\k$.
The multiplicative group $L^*$ embeds as a nonsplit torus $\T\subset G$, with $\k^*$
mapping onto the center $Z\subset G$.   A character $\nu:L^*\rightarrow\C^*$ is
{\bf primitive} (or regular) if $\nu\neq \nu^q$, or equivalently, if $\nu$ is not of the form
$\chi\circ N^L_{\k}$ for a character $\chi$ of $\k^*$, where $N_{\k}^L$ is the norm map.
 There are $q(q-1)$ primitive characters of $L^*$.
Given a character $\w$ of $\k^*$, let $[\w]$ denote the set of primitive characters $\nu$ 
satisfying $\nu|_{\k^*}=\w$.  
By \cite[Proposition 2.3]{KR}, the cardinality of $[\w]$ is
\begin{equation}\label{Pw}
P_{\w}=\begin{cases}q-1&\text{if $q$ is odd and $\w^{(q-1)/2}$ is trivial}\\
q+1&\text{if $q$ is odd and $\w^{(q-1)/2}$ is nontrivial}\\
q&\text{if $q$ is even}.\end{cases}
\end{equation}

Let $U=\mat 1\k01$ be the upper triangular unipotent subgroup of $G$.  A representation
of $G$ is {\bf cuspidal} if it does not contain a $U$-fixed vector.
Fix a nontrivial additive character
\[\psi:\k\longrightarrow\C^*.\]
We will always take $\psi(x)=e(\tfrac{x}p)=e^{2\pi i x/p}$ if $\k=\Z/p\Z$.
We may view $\psi$ as a character of $U$ in the obvious way.

Given a primitive character $\nu$ of $\T$, there is a unique irreducible cuspidal representation $\rho_\nu$
of dimension $q-1$ satisfying
\[\Ind_{ZU}^G(\k)(\nu\otimes\psi) =\rho_\nu\oplus \Ind_\T^G\nu.\]
Every cuspidal representation arises in this way, and 
$\rho_\nu\cong \rho_{\nu'}$ if and only if $\nu'\in\{\nu,\nu^q\}$.  

We have the following well-known formula for the character of $\rho_\nu$. For $x\in G(\k)$, 
\begin{equation}\label{trho}\tr\rho_\nu(x)=\begin{cases} (q-1)\nu(x)&\text{if }x\in Z\\
-\nu(z)&\text{if }x=zu, z\in Z, u\in U, u\neq 1\\
-\nu(x)-\nu^q(x)&\text{if }x\in \T, x\not\in Z\\
0&\text{if no conjugate of $x$ belongs to $\T\cup ZU$}.\end{cases}
\end{equation}
Because $\nu(c^{-1}x c)=\nu(x^q)$ for all $c\in N_G(\T)-\T$, there is no ambiguity 
evaluating $\tr\rho_\nu(y)$ using the third row above if $y$ is conjugate in $G(\k)$ to $x\in \T$.

Working now in the group $G(F)$, given the surjection $K\rightarrow G(\k)$
 obtained by reduction modulo $\p$, we may view
$\rho_\nu$ as a representation of $K$.  Its central character is given by $z\mapsto \nu(z(1+\p))$
for $z\in \O^*$.
 By choosing a complex number $\nu(\varpi)$ of norm $1$, we may extend
$\rho_\nu$ to a representation of $ZK$, and then
\[\sigma_\nu=\cInd_{ZK}^{G(F)}(\rho_\nu)\]
is an irreducible unitary supercuspidal representation of conductor $\p^2$. Its formal degree under the
normalization $\meas(K)=1$ is 
\begin{equation}\label{fda}
d_{\sigma_\nu}=\dim\rho_\nu=q-1.
\end{equation}
The only equivalences among the representations $\sigma_\nu$ are $\sigma_\nu\cong \sigma_{\nu^q}$ 
(provided $\nu^q(\varpi)$ is defined to be the same complex number as $\nu(\varpi)$).

We define the test function $\fp:G(F)\rightarrow \C$ by
\begin{equation}\label{fpd0}
\fp(g)=\begin{cases}\overline{\tr\rho_\nu(g)}&\text{if }g\in ZK\\0&\text{otherwise},\end{cases}
\end{equation}
where $\tr\rho_\nu$ is given in \eqref{trho}.

\begin{proposition}\label{rootd0}
Suppose $\sigma_\nu$ has trivial central character.  Then its root number is given by 
\begin{equation}\label{ed0a}
\epsilon_\nu=\epsilon(\tfrac12,\sigma_\nu,\psi)=\begin{cases} -(-1)^{(q+1)/r}&\text{if $q$ is odd}\\
-1&\text{if $q$ is even},\end{cases} 
\end{equation}
where $r$ is the order of $\nu$ in the character group of $L^*$.
Suppose further that $q$ is odd and $4\nmid (q-1)$ so that $\alpha^2=-1$ for some 
  $\alpha\in L^*-\k^*$.  Then
\begin{equation}\label{ed0b}
\epsilon_\nu=-\nu(\alpha).
\end{equation}
\end{proposition}
\noindent{\em Remark:} Under the hypothesis, $\nu|_{\k^*}$ is trivial, which is
equivalent to $r|(q+1)$ when $q$ is odd. 

\begin{proof}
The root number coincides with the Atkin-Lehner sign of the representation (\cite[3.2.2 Theorem]{Sch}).
We will show that it is a Gauss sum for $\nu$, which can be
evaluated explicitly.  The Atkin-Lehner sign $\epsilon_\nu$ is defined by
\[\sigma_\nu(\smat{}1{\varpi^2}{})\varphi =\epsilon_\nu\varphi,\]
where $\varphi$ is a new vector in the space of $\sigma_\nu$.
Note that $\epsilon_\nu^2=1$ since $\sigma(\smat{}1{\varpi^2}{}^2)
=\sigma(\smat{\varpi^2}{}{}{\varpi^2})$ acts trivially under the hypothesis of trivial
  central character.

A model for $\rho_\nu$ on the space $\C[\k^*]$ of complex-valued functions on $\k^*$ 
is described in \cite{KR}, following \cite{PS}.  In terms of this model,
the new space $(\cInd_{ZK}^{G(F)}(\rho_\nu))^{K_1(\p^2)}$ is spanned by the function 
$\varphi:G(F)\rightarrow \C[\k^*]$
supported on the coset $ZK\mat{\varpi}{}{}1K_1(\p^2)$ and defined by
\[\varphi(zk\smat\varpi{}{}1)=\rho_\nu(zk)w\qquad (z\in Z, k\in K),\]
where $w\in \C[\k^*]$ is the constant function $1$  (\cite[Proposition 3.1]{KR}).
In particular 
\[\varphi(\smat{\varpi}{}{}1)(1)=w(1)=1.\]
Therefore the Atkin-Lehner eigenvalue is given by
\[\epsilon_\nu = [\sigma_\nu(\smat{}1{\varpi^2}{})\varphi](\smat\varpi{}{}1)(1)
=\varphi\Bigl(\smat\varpi{}{}1\smat{}1{\varpi^2}{}\Bigr)(1)
=\varphi\Bigl(\smat\varpi{}{}\varpi\smat{}11{}\smat\varpi{}{}1\Bigr)(1)\]
\[=\nu(\varpi)\Bigl(\rho_\nu(\smat{}11{})w\Bigr)(1)=\Bigl(\rho_\nu\bigl(\smat{}1{-1}{}\smat{-1}{}{}1\bigr)w\Bigr)(1),\]
since we are assuming $\nu|_{F^*}=1$.
Let $f_a\in \C[\k^*]$ be the characteristic function of $a\in \k^*$, so that
$w=\sum_{a\in\k^*}f_a$.  Using \cite[(2-11)]{KR} we see that $\rho_\nu(\smat{-1}{}{}1)w=w$, and
just below (2-16) of the same reference, we have
\[\bigl(\rho_\nu(\smat{}1{-1}{})f_a\bigr)(1)=-\frac1q\nu(a^{-1})
\sum_{u\in L^*\atop{N(u)=a}}\psi(\tr^L_{\k}(u))\nu(u)\]
for all $a\in \k^*$.
We are assuming that $\nu|_{\k^*}=1$, so $\nu(a^{-1})=1$, and summing over $a\in \k^*$ we have
\[\epsilon_\nu=-\frac1q\sum_{u\in L^*}\psi(\tr^L_{\k}(u))\nu(u).\]
This Gauss sum can be evaluated explicitly by an elementary calculation, giving
\eqref{ed0a}; see \cite[Theorem 11.6.1]{BEW} for details.

Now suppose $q$ is odd and $4\nmid (q-1)$, and let $t$ be a generator of the cyclic group
  $L^*$, so in particular $t^{\frac{q^2-1}2}=-1$. If $\nu$ has order $r$, there exists $j$ with $\gcd(j,r)=1$ such that
$\nu(t)=e(\frac jr)$.  Taking $\alpha=t^{\frac{q^2-1}4}$, 
\[\nu(\alpha)=e\bigl(\frac{j(q+1)(q-1)}{4r}\bigr)=(-1)^{\frac{j(q+1)}r\frac{q-1}{2}}
=(-1)^{\frac{j(q+1)}r}\]
since $\frac{q-1}2$ is odd by hypothesis.  The above is equal to $(-1)^{\frac{q+1}r}=-\epsilon_\nu$, since
$r$ is odd when $j$ is even, and $2|(q+1)$.  This proves \eqref{ed0b}.
\end{proof}

\begin{corollary}\label{ecount}
Fix $\epsilon\in\{\pm1\}$.
Then the number of depth zero supercuspidal representations of $G(F)$ with trivial central character
and root number $\epsilon$ is
\[\begin{cases}
\frac{q-1}4&\text{if }q\equiv 1\mod 4\\
\frac{q+1}4&\text{if $q\equiv 3\mod 4$ and $\epsilon = 1$}\\
\frac{q-3}4&\text{if $q\equiv 3\mod 4$ and $\epsilon = -1$}\\
0&\text{if $q$ is even and $\epsilon =1$}\\
\frac q2&\text{if $q$ is even and $\epsilon = -1$}.
\end{cases}
\]
\end{corollary}
\begin{proof}
With notation as in \eqref{Pw}, the number of supercuspidals with a given central character $\w$
is $P_\w/2$.  (We divide by $2$ to account for the fact that $\nu$ and $\nu^q$ induce the same
supercuspidal.)  So the assertion for $q$ even is immediate from \eqref{Pw} and \eqref{ed0a}.

Let $q$ be odd, and let $t$ be a generator of the cyclic group $L^*$.  Then $t^{q+1}$ is a generator
of $\k^*$.  The characters of $L^*$ are the maps $\nu_m$ defined by
\[\nu_m(t)=e(\frac{m}{q^2-1}),\]
for $0\le m<q^2-1$.  
We consider only those characters satisfying $\nu_m|_{\k^*}=1$, i.e., $(q-1)|m$.  
Notice that $\nu_m$ is imprimitive if and only if $\nu_m^{q-1}=1$, which
 holds if and only if $(q+1)| m$.  
So we consider the values $m=k(q-1)$ (for $1\le k< (q+1)$) which are not multiples of $q+1$,
i.e., $k\neq \frac{q+1}2$. 

The order of $\nu_m$ is 
\begin{equation}\label{numorder}
\frac{q^2-1}{\gcd(m,q^2-1)} = \frac{q+1}{\gcd(k,q+1)}.
\end{equation}
By \eqref{ed0a}, $\sigma_{\nu_m}$ has root number
\begin{equation}\label{ek}
\epsilon_{\nu_m} = -(-1)^{\gcd(k,q+1)}=-(-1)^k,
\end{equation}
since $q+1$ is even.
 Notice that the removed value $\frac{q+1}2$ of $k$ is 
  odd if and only if $q\equiv 1\mod 4$.
So in this case, among the remaining $q-1$ values of $k$, half are odd and half are even.
If $q\equiv 3\mod 4$, then $\frac{q-1}2+1=\frac{q+1}2$ of the remaining values of $k$ are odd, 
  and $\frac{q-1}2-1=\frac{q-3}2$ are even.

To count supercuspidal representations, we divide the number of relevant $k$'s 
  by 2 since the distinct characters $\nu_m$ and $\nu_m^q$
induce the same representation.
\end{proof}

\subsection{Simple supercuspidal representations}\label{ssr}

With notation as in the previous section,
we recall here the construction of the supercuspidal representations of $G(F)$ of 
conductor $\p^3$.  The central character of any such representation is at most tamely ramified.
So we begin by fixing a character
$\w_\p$ of the center $Z=Z(F)\cong F^*$ of $G(F)$, trivial on $1+\p$.

Define the following compact open subgroup of $G(F)$:
\[K'=\mat{1+\p}{\O}{\p}{1+\p}.\]
Fix a nontrivial character
\[\psi:\k\longrightarrow\C^*,\]
which we also regard as a character of $\O$ trivial on $\p$. 
Given $t\in \k^*$, define a character
$\chi=\chi_t:K'\longrightarrow \C^*$ by 
\begin{equation}\label{chi}
    \chi(\mat ab{c\varpi}d) = \psi(b+tc).
\end{equation}
The matrix
\[g_t=g_\chi=\mat{}t{\varpi}{}\]
normalizes $K'$, and furthermore
\begin{equation}\label{gchi}
    \chi(g_\chi^{-1}kg_\chi)=\chi(k)
\end{equation}
for all $k\in K'$.

Given $\chi$ as above, let
\begin{equation}\label{H'}
H'=ZK'\cup g_\chi ZK'.
\end{equation}
Although it is not reflected in the notation $H'$, this set depends on both $t$ and the fixed
choice of $\varpi$.
Given that $g_\chi^2=t\varpi$,
we may extend $\chi$ to a character $\chi_\zeta$ of $H'$ via
\begin{equation}\label{xzdef}
    \chi_\zeta(g_\chi^d zk)=\zeta^d\w_\p(z)\chi(k)
\end{equation}
for $z\in Z$ and $k\in K'$,
where $\zeta$ is a fixed complex number satisfying 
\begin{equation}\label{zeta2}
\zeta^2=\w_\p(t\varpi).
\end{equation}

\begin{proposition}\label{ssc}
The compactly induced representation $\sigma_\chi^\zeta=\cInd_{H'}^{G(F)}(\chi_\zeta)$ is an
irreducible supercuspidal representation of conductor $\p^3$, with root number 
\[\epsilon(\tfrac12,\sigma_\chi^\zeta,\psi)=\zeta.\]
Conversely, every irreducible admissible representation
of $G(F)$ of conductor $\p^3$ with central character trivial on $1+\p$ 
arises in this way.
\end{proposition}
\begin{proof} See \cite{Ku}.  
 For a more recent treatment using the above notation (but on $\GL_n$), see
\cite[\S4-5 and Prop.\,7.2]{super}.
The root number is computed in \cite[Corollary 3.12]{AL}.
\end{proof}

We will also use the notation
\[\sigma_t^\zeta = \sigma_\chi^\zeta\]
for $t,\chi$ as in \eqref{chi}, though it should be borne in mind that the representation
depends also on the choice of additive character $\psi$ and uniformizer $\varpi$.
When $F=\Q_p$, we will always take $\varpi=p$ and 
\[\psi(x)=e(\tfrac xp)=e^{2\pi i x/p}\]
 for $x\in \Z/p\Z$.

Henceforth we assume that $\w_\p$, and hence also $\sigma_t^\zeta$, is unitary.
Under the normalization $\meas(\ol{G(\O)})=1$, the formal degree of $\sigma_\chi^\zeta$ is
\begin{equation}\label{fd}
    d_\chi=\frac{q^2-1}2.
\end{equation}
This is seen, for example, from (6.4) of \cite{super} and the last line of the proof 
of Corollary 6.5 of the same paper.

We define the matrix coefficient $f_\p: G(F)\longrightarrow\C$ by
\[f_\p(g)=d_\chi\ol{\sg{\sigma_t^\zeta(g)\frac{\phi}{\|\phi\|},\frac{\phi}{\|\phi\|}}},\]
where $\phi\in \cInd_{H'}^{G(F)}(\chi_\zeta)$ is the function
\begin{equation}\label{phi}
    \phi(g)=\begin{cases}\chi_\zeta(g)&\text{if }g\in H'\\0&\text{otherwise.}
    \end{cases}
    \end{equation}
    Note that 
    \begin{equation}\label{phinorm}
        \|\phi\|^2=\int_{\olG(F)}|\phi(g)|^2dg=\meas(\ol{H'}).
    \end{equation}
    Likewise,
    \[\sg{\sigma_t^\zeta(g)\phi,\phi}=\int_{\olG(F)}\phi(xg)\ol{\phi(x)}dx=
    \int_{\ol{H'}}\phi(xg)\ol{\chi_\zeta(x)}dx\]
 \begin{equation}\label{mc}
     =\begin{cases}
         \meas(\ol{H'})\chi_\zeta(g)&\text{if }g\in H'\\
         0&\text{otherwise.}
    \end{cases}\end{equation}
 By \eqref{fd}, \eqref{phinorm}, and \eqref{mc}, we have
     \begin{equation}\label{fp}
         f_\p(g)=\begin{cases}\frac{q^2-1}2\ol{\chi_\zeta(g)}&\text{if }g\in H'\\0&\text{otherwise.}
   \end{cases}
     \end{equation}

\subsection{Global setup}\label{test}

Fix square-free integers $S,T>0$ with $ST>1$ and $\gcd(S,T)=1$, and let $k>2$. 
Set $N=S^2T^3$, and let 
  $\w'$ be a Dirichlet character of modulus $N$ satisfying
\begin{equation}\label{winf}
\w'(-1)=(-1)^k.
\end{equation} 
Let $\w$ be the Hecke character attached to $\w'$ in \eqref{wfact}.
We assume in addition that for each $p|N$, $\w_p$ is trivial on $1+p\Z_p$, 
since this is true of the central character of every supercuspidal representation 
of conductor $\le p^3$.  Equivalently, the conductor of $\w'$ divides $ST$.

\begin{proposition}\label{k2N2}
If $N=2^2$ or $2^3$ and $k$ is odd, there is no such character.
\end{proposition}
\begin{proof}
If $N$ is a power of $2$, then by \eqref{w'} and \eqref{winf}, $(-1)^k=\w'(-1)=\w_2(-1) =1$ 
since $\w_2$ is trivial on $\Z_2^*=1+2\Z_2$.  So $k$ must be even.
\end{proof}

Under the stated hypotheses, for each $p|S$, $\w_p$ is trivial on $1+p\Z_p$.
We may thus view $\w_p$ as a character of $(\Z_p/p\Z_p)^*=\F_p^*$.
For each such $p$, fix a primitive character $\nu_p$ of $\F_{p^2}^*$ such that $\nu_p|_{\F_p^*}=\w_p$.
Recall that the number $P_{\w_p}>0$ of such primitive characters is given in \eqref{Pw}.
We define $\nu_p(p)=\w_p(p)$ and extend multiplicatively so that $\nu_p$ can also be viewed
as a character of $\Q_p^*$, which allows us to view $\rho_{\nu_p}$ as a representation of $Z_pK_p$
with central character $\w_p$.   We let 
  \[\sigma_p=\sigma_{\nu_p}=\cInd_{Z_pK_p}^{G(\Q_p)}(\rho_{\nu_p})\]
be the associated
supercuspidal representation of $G(\Q_p)$.
The number of isomorphism classes of supercuspidal representations of conductor $p^2$ and central character
  $\w_p$ is $P_{\w_p}/2$.

For each prime $p|T$, fix a
  simple supercuspidal representation $\sigma_p=\sigma_{t_p}^{\zeta_p}$ of $G(\Q_p)$ with
  central character $\w_p$, where $t_p\in (\Z/p\Z)^*$ and $\zeta_p^2=\w_p(t_pp)$.
When the prime $p$ is understood, we sometimes write $t,\zeta$ instead of $t_p,\zeta_p$.
By \eqref{wpp},  
\begin{equation}\label{zeta22}
\zeta_p^2 = \w_p(t_pp)=\w_p(t_p)\prod_{\substack{\ell|N,\\\ell\neq p}}\w_\ell(p^{-1}).
\end{equation}
In particular, when $N=p^3$ for $p$ prime,
$\zeta_p^2=\w_p(t_p)$.

Having made the above choices, we let $\widehat{\sigma}=(\sigma_p)_{p|N}$
denote this tuple of local representations.
Then $S_k(\widehat{\sigma})\subset S_k^{\new} (S^2T^3,\w')$.

Now consider the test function 
\begin{equation}\label{fnt}
f=f^\n=f_\infty\prod_{p|N}f_p\prod_{\ell\nmid N}\fnp
\end{equation}
as in \eqref{fn} with $N=S^2T^3$,
where, for $p|S$ (resp. $p|T$), $f_p$ is the chosen test function given in \eqref{fpd0}
 (resp. \eqref{fp}).

The above setup is slightly different from that used in \eqref{fn} and Proposition
\ref{proj} since for $p|S$,
$f_p$ is not a single matrix coefficient, but a certain sum of matrix coefficients, and without the
formal degree coefficient.  Nevertheless, the conclusions of Proposition \ref{proj} do hold
for the above test function, as the next result shows.

\begin{proposition}\label{projtr}
With $f$ defined above, $\tr(T_\n|S_k(\widehat{\sigma}))=\n^{k/2-1}\tr R(f)$.
\end{proposition}
\noindent{\em Remark:} 
By \cite[Proposition 1.2]{KR}, the proof below applies
with any unramified (even power conductor) supercuspidals $\sigma_p$ at $p|S$, using $d_{\sigma_p}=\dim\rho$ in 
place of $p-1$, where $\sigma_p=\cInd_{ZK}^G(\rho)$.  
(Ramified supercuspidals may be induced from a {\em character} of an appropriately chosen open 
compact-mod-center subgroup, so for these, one can use a test function analogous to \eqref{fp}.)

\begin{proof}
In the proof of Proposition \ref{proj}, we used the fact (\cite[Corollary 10.26]{KL})
  that for $\sigma=\sigma_p$,
 the operator $\sigma\bigl(d_\sigma\ol{\sg{\sigma(g)w,w}}\bigr)$ is the orthogonal projection
of the space of $\sigma$ onto $\C w$. 
For $f_p$ in \eqref{fpd0}, by \cite[Proposition 1.1]{KR}, there is an orthonormal set $\{w_1,\ldots,w_{p-1}\}$
of vectors in the space of $\sigma$ such that 
\[f_p(g)=\sum_{j=1}^{p-1}\ol{\sg{\sigma(g)w_j,w_j}}.\]
Therefore $\sigma(d_{\sigma}f_p)=\sigma((p-1)f_p)$ is the orthogonal projection onto 
$\Span\{w_1,\ldots,w_{p-1}\}$.  So using this local test function in the proof of
  Proposition \ref{proj} would give us a block sum of $p-1$ copies of the matrix 
  for $\n^{1-k/2}T_\n$. To get the correct trace, we would need to divide by $p-1$, 
which is achieved by simply taking $f_p$ instead of $(p-1)f_p$.
\end{proof}

Noting that for $p|S$, $f_p(1)=\dim\rho_\nu=p-1=d_{\sigma_p}$, 
  the identity term in the formula for $\tr R(f)$ is
\begin{equation}\label{id}
     \ol{\w'(\n^{1/2})}\frac{k-1}{12}\prod_{p|S}(p-1)\prod_{p|T}\frac{p^2-1}2,
\end{equation}
as seen between the brackets in Theorem \ref{stf2}.
  We remark that this is not always an integer when $\n=1$. For example consider the
case where $S=1$.  For $p\ge 3$ prime,
     \[v_2(p^2-1)=v_2(p-1)+v_2(p+1)\ge 3,\]
     with equality holding precisely when $p\equiv 3,5\mod 8$.  (Here, $v_2$ is the $2$-adic
valuation.)
     It follows easily that when $\n=1$, the identity term $\frac{k-1}{12}\prod_{p|T}\frac{p^2-1}2$
  fails to be an integer in exactly 
the following situations:
\begin{itemize}
\item $T=2$ and $k\not\equiv 1\mod 8$;
\item $T=3$ and $k\not\equiv 1\mod 3$;
\item $T=2p$ for some $p\equiv 3,5\mod 8$, and $k$ is even.
\end{itemize}
  So in such instances, when $S=\n=1$ the elliptic contribution to $|H_k(\widehat{\sigma})|$ 
  in Theorem \ref{stf2} must be nonzero for this simple reason.

The list of relevant matrices in the trace formula of Theorem \ref{stf2} can be refined in certain 
situations.  

\begin{proposition}\label{relevant}
Let $N=S^2T^3$ as above,
let $f=f^\n$ be the test function defined in \eqref{fnt}, let $M|T$, and $0\le r<\sqrt{4\n/M}$. 
Then $\Phi(\mat0{-\n M}1{rM},f)= 0$ in each of the following situations:
\begin{itemize}
\item $r=0$ and $k$ is odd.
\item There exists $p|N$ such that $X^2-rMX+\n M$ has a root in $\Q_p$.
\item There exists $p|M$ such that $-pt_p/\n M$ is not a square modulo $p$, where $t_p$
is the parameter of the local representation $\sigma_{t_p}^{\zeta_p}$.
\item There exists $p|\frac{T}{M}$ such that $X^2-rMX+\n M\equiv (X-z)^2\mod p$ has no solution $z\in(\Z/p\Z)^*$.
\end{itemize}
\end{proposition}
\noindent{\em Remark:}
For the case $\n=1$, we can refine the list of relevant $\g$ even further 
(see Proposition \ref{glist} below).

\begin{proof}
The first bullet point follows from \eqref{Phiinf}.

Let $\g=\smat{}{-\n M}1{rM}$, and suppose that $\Phi(\g,f)\neq 0$.
Then by Proposition \ref{ellg}, $\g$ is elliptic in $G(\Q_p)$, which gives the second bullet point.

For the third bullet point, suppose $p|M$.  Write $\det\g=up$ for some $u\in \Z_p^*$.
Assuming the local orbital integral $\Phi(\g,f_p)$ is nonzero, $f_p(g^{-1}\g g)\neq 0$ 
for some $g\in G(\Q_p)$.  Then
$g^{-1}\g g$ belongs to the ramified component of $\Supp(f_p)$, i.e., writing $t=t_p$,
\[g^{-1}\g g=z\mat{}t{p}{}\mat ab{pc}d\in Zg_{\chi_p} K'\]
for some $b,c\in\Z_p$, $a,d\in 1+p\Z_p$, and $z\in\Z_p^*$.
Taking determinants, we have
\[up = -tpz^2(ad-pbc),\]
and hence 
\begin{equation}\label{qr}
    u\equiv-tz^2\mod p.
\end{equation}
This shows that $-t/u$ is a quadratic residue modulo $p$.

Finally, if $p|\frac NM$, then $\det\g\in\Z_p^*$ so if $\Phi(\g,f_p)\neq 0$, some conjugate $g^{-1}\g g$ lies in the
unramified component of $\Supp(f_p)$:
\[g^{-1}\g g=z\mat ab{pc}d\in ZK'\]
for $z,a,b,c,d$ as above.  Taking determinants, $\det\g \equiv z^2\mod p$.
Taking the trace, $\tr \g \equiv 2z\mod p$.  Hence $P_\g(X)\equiv X^2-2zX+z^2\equiv (X-z)^2\mod p$. 
\end{proof}

\section{Local orbital integrals at primes $p|N$ for $N=S^2T^3$}\label{pN}

Our goal here is to compute
\[\Phi(\g,f_p)=\int_{\ol{G_\g(\Q_p)}\bs \olG(\Q_p)}f_p(g^{-1}\g g)dg\]
    taking for $f_p$ the test functions given in \eqref{fpd0} and \eqref{fp},
 and for $\g$ the
    matrices given in Theorem \ref{stf2}, and using the quotient measure defined
    in \S\ref{gell}, so
\[\Phi(\g,f_p)=\int_{\olG(\Q_p)}f_p(g^{-1}\g g)dg.\]
With these calculations in hand, Theorem \ref{mainST} will follow immediately from Theorem \ref{stf2}.

We will use the strategy adopted by Palm in \cite[Prop. 9.11.3]{P}
which avoids the use of lattices or buildings.  There are 
errors in the statement and proof of his proposition, so we cannot simply quote 
the result.
However, the basic method is sound and can be adapted to 
give the result in the cases of interest to us here.

The following lemma will allow us to rewrite the integral in such a way as to exploit
the structure of the support of $f_p$.

\begin{lemma}[{\cite[Lemma 6.4.10]{P}}]\label{Ilem}
    Let $G$ be a unimodular locally compact group, and suppose $I_1,I_2$ are two open compact
    subgroups of $G$, each given total Haar measure $1$.  Then for any choice of Haar measure on $G$
    we have
    \begin{equation}\label{dc}
        \int_G \phi(g)dg=\sum_{x\in I_1\bs G/I_2}\meas_G(I_1xI_2)\int_{I_1}\int_{I_2}\phi(i_1xi_2)
    di_2di_1
    \end{equation}
    for all $\phi\in C_c(G)$.
\end{lemma}
\begin{proof}
    For $\phi\in C_c(G)$, we see that
    \[\int_G\phi(g)dg=\int_G\int_{I_1}\int_{I_2}\phi(i_1gi_2)di_2di_1dg\]
by changing the order of integration and using the bi-invariance of $dg$.
    The inner double integral defines a compactly supported function $F$
of $g\in G$ which is constant on double cosets $I_1gI_2$, and is therefore a finite linear 
combination of characteristic functions of such double cosets.
    The identity \eqref{dc} clearly holds for the characteristic function of
     a double coset.  By linearity it holds for $F$ as well, so 
    \[\int_G\phi(g)dg=\int_GF(g)dg
     = \sum_{x\in I_1\bs G/I_2}\meas_G(I_1xI_2)\int_{I_1}\int_{I_2}F(i_1xi_2)
    di_2di_1\]
    \[ = \sum_{x\in I_1\bs G/I_2}\meas_G(I_1xI_2)F(x) =
        \sum_{x\in I_1\bs G/I_2}\meas_G(I_1xI_2)\int_{I_1}\int_{I_2}\phi(i_1xi_2)
    di_2di_1.\qedhere\]
\end{proof}

\subsection{Preliminaries when $p|T$}

Throughout much of this section, we will work over a $p$-adic field $F$ with notation
as in \S\ref{ssr}, and write $G$ for $G(F)$, and $\olG$ for $G/Z$.
Having fixed a simple supercuspidal representation $\sigma_t^\zeta$ of $G$ with unitary central
character $\w_\p$, 
we take $\fp$ to be the test function given in \eqref{fp}.

Applying Lemma \ref{Ilem} to \eqref{ellint}, we have
\[\Phi(\g,\fp)=\int_{\olG}\fp(g^{-1}\g g)dg\]
\[
=\sum_{x\in \ol{K'}\bs \olG/\ol{K'}}\meas_{\olG}(\ol{K'}x\ol{K'})\int_{\ol{K'}}\int_{\ol{K'}}
\fp(h_2^{-1}x^{-1}h_1^{-1}\g h_1 xh_2)dh_1dh_2,\]
where each $dh_i$ is normalized to have total measure $1$.
Since $\fp|_{K'}$ is a character,
  $h_2$ has no effect, and we obtain
\begin{equation}\label{Phifp}
    \Phi(\g,\fp)=\sum_{x\in \ol{K'}\bs \olG/\ol{K'}}\meas_{\olG}(\ol{K'}x\ol{K'})\int_{\ol{K'}}
\fp(x^{-1}h^{-1}\g h x)dh.
\end{equation}

In order to compute the above, we need a few preparations.  First, recall the affine
Bruhat decomposition
\[G= K'MK' \cup K'MwK'=K'MK' \cup K'Mg_\chi K',\]
where $w=\smat{}{-1}1{}$ and $M$ is the diagonal subgroup (\cite[Prop. 17.1]{BH}).
Accordingly, we may take as a set of representatives $x\in \ol{K'}\bs\olG/\ol{K'}$
the elements $x=m$ and $x=mg_\chi$ for 
\begin{equation}\label{xreps}
    m\in \left\{\mat y{}{}1,\, \mat y{}{}{\varpi^j},\,
    \mat{\varpi^n}{}{}y |\,j>0,
n>0, y\in (\O/\p)^*\right\}.
\end{equation}
For each such $x$
we need to compute the integral in \eqref{Phifp}, which we denote by
\[ K_\g(x)=   \int_{\ol{K'}} \fp(x^{-1}h^{-1}\g h x)dh.\]
By \eqref{gchi}, 
\[\fp(g_\chi^{-1}g g_\chi)=\fp(g)\]
    for all $g$.  Therefore $K_\g(xg_\chi)=K_\g(x)$. Furthermore, since
    $g_\chi$ normalizes $K'$, the measure of $\ol{K'}x\ol{K'}$ is
    unchanged if $x$ is replaced by $xg_\chi$.
    It follows that
\begin{equation}\label{Phifp2}
 \Phi(\g,\fp)=2\sum_{x\text{ in }\eqref{xreps}}
 \meas_{\olG}(\ol{K'}x\ol{K'})K_\g(x).
\end{equation}

\begin{lemma}
Let $x=\smat{\varpi^n}{}{}y$ or $\smat y{}{}{\varpi^n}$ for $n\ge 0$ and $y\in \O^*$.
Then with measure on $\olG$ normalized so that $\meas(\ol{K})=1$, 
\begin{equation}\label{Knmeas}
\meas_{\olG}(\ol{K'}x\ol{K'})=\frac{q^n}{q^2-1}.
\end{equation}
\end{lemma}

\begin{proof}
We may assume that $y=1$ since, for example,
\[
\meas(\ol{K'}\smat{\varpi^n}{}{}{y} \ol{K'})=
    \meas( \ol{K'}\smat{\varpi^n}{}{}1\ol{K'} \smat1{}{}{y})
    =\meas(\ol{K'}\smat{\varpi^n}{}{}1\ol{K'}).\]
Likewise, since $g_\chi$ normalizes $K'$ and $g_\chi^{-1}\smat{\varpi^n}{}{}1g_\chi
=\smat1{}{}{\varpi^n}$, we may assume that $x=\smat{\varpi^n}{}{}1$. 

We claim that for $n\ge 0$,
\begin{equation}\label{cosets}
    {K'}\mat{\varpi^n}{}{}1{K'}=\bigcup_{b\in \O/\p^n}\mat{\varpi^n} b01K',
\end{equation}
a disjoint union. The union is disjoint since
\[{\mat{\varpi^n} {b_1}01}^{-1}\mat{\varpi^n} {b_2}01=\mat 1{\frac{b_2-b_1}{\varpi^n}}01,\]
which is in $K'$ if and only if $b_1\equiv b_2\mod \p^n$.
The inclusion $\supseteq$ in \eqref{cosets} follows from
\[\mat{\varpi^n} b01=\mat1b01\mat{\varpi^n}{}{}1\in K'\mat{\varpi^n}{}{}1.\]
The reverse inclusion follows from 
\[\mat abcd\mat{\varpi^n}{}{}1=\mat{\varpi^n}{bd^{-1}}01\mat{a-cbd^{-1}}0{c\varpi^n}d.\]

  By the decomposition \eqref{cosets}, 
  \[  \meas(\ol{K'}\mat{\varpi^n}{}{}1\ol{K'})=q^n\meas(\ol{K'})=\frac{q^n}{q^2-1},\]
since $\meas(\ol{K'})=\frac1{q^2-1}$ when $\meas(\ol{K})=1$, 
as shown in the proof of \cite[Cor. 6.5]{super}.
\end{proof}

If $x=\smat y{}{}1$, then $K_\g(x)=\fp(\g^y)$ where $\g^y=\smat{y^{-1}}{}{}1\g\smat y{}{}1$, 
since $\fp$ is a character of $K'$, $\smat y{}{}1$ normalizes $K'$, and we give
    $\ol{K'}$ measure 1. Thus, in view of the above lemma, \eqref{Phifp2} now becomes
\begin{align}\label{Phifp3}
    \Phi(\g,\fp)=\frac2{q^2-1}&\sum_{y\in (\O/\p)^*}f_\p(\g^y)\\
\notag&+2\sum_{n=1}^\infty\frac{q^n}{q^2-1}\sum_{y\in (\O/\p)^*}\left[K_\g(\smat{\varpi^n}{}{}y)
+K_\g(\smat y{}{}{\varpi^n})\right].
\end{align}

To compute $K_\g(x)$, we fix coordinates on $\ol{K'}$ with the following.

\begin{lemma}\label{GHK} Let $G,H,K$ be compact topological groups, with $G=HK$ and $H\cap K=\{1\}$. 
    Let $dh$ and $dk$ be the respective Haar measures on $H,K$ of total measure $1$.
    Then the Haar measure on $G$ of measure $1$ is given by
    \[\int_Gf(g)dg=\int_H\int_Kf(hk)dk\,dh.\]
\end{lemma}
\begin{proof} This is a special case of \cite[Lemma 7.13]{KL}.
\end{proof}

We will use the Iwahori decomposition \cite[(7.3.1)]{BH} of $K'$.
Letting $M(1+\p)=\smat{1+\p}{}{}{1+\p}$, $N(\O)=\smat1\O{0}1$, and $N'(\p)=\smat{1}{0}{\p}1$, 
the decomposition 
\[K'=N(\O)\cdot N'(\p)\cdot M(1+\p)\]
is a (topological) direct product, and the same is true for any ordering of the three factors.
We will take $\meas(\ol{K'})=\meas(K')=1$, so that
applying the above lemma, this Haar measure on $\ol{K'}$ is given by both of the following:
\begin{align}\label{b1}\int_{\ol{K'}}\phi(k)dk 
    &= \int_{\O}\int_{\O}\int_{M(1+\p)}
\phi(\mat 10{\varpi c}1\mat 1b01 m)dm\, db\, dc\\
    \label{c1}    &= \int_{\O}\int_{\O}\int_{M(1+\p)}\phi(\mat1b01\mat10{\varpi c}1m)dm\,dc\,db,
\end{align}
where $dm,db,dc$ each have total measure 1.

\subsection {The case where $p|T$ and $\g$ is ramified}

The aim here is to compute $\Phi(\g,f_p)$ when $p|N$ and $\g$ is ramified at $p$, i.e., 
$v_p(\det\g)$ is odd.
As above we work over a $p$-adic field $F$ with uniformizer $\varpi$, and a fixed 
supercuspidal representation $\sigma_t^\zeta$ of $G(F)$ as in \S\ref{ssr}. 
We can assume that $v_\p(\det\g)=1$, and further, by Lemma \ref{pd}, that $v_\p(\tr\g)\ge 1$.
So we will consider matrices in the $F$-rational canonical form:
\begin{equation}\label{gp}
    \g=\mat 0 {-u\varpi}1{v\varpi}=w\mat1{v\varpi}0{u\varpi}
\end{equation}
for $u\in \O^*$, $v\in \O$, and $w=\smat{}{-1}1{}$. 

\begin{proposition}  \label{Phiell}
    For $\g$ as in \eqref{gp} and $\fp$ as in \eqref{fp}, 
$\Phi(\g,\fp)=0$ unless $-t/u$ is a square modulo $\p$.
    If the latter condition holds and $y^2\equiv -t/u\mod \p$, then
 \[\Phi(\mat0{-u\varpi}1{v\varpi},\fp)=\ol{\zeta}\Bigl(\ol{\psi(yv)}\w_\p(y)+\delta(\p\nmid 2)\,
  \ol{\psi(-yv)}\w_\p(-y)\Bigr),\]
where $\psi$ is the nontrivial character of $\O/\p$ used in \eqref{chi}.
    Thus, in the case of trivial central character (so $\zeta^2=1$), we have
 \[\Phi(\g,\fp)=\begin{cases} 2\zeta\Re(\psi(yv))&\text{if }\p\nmid 2\\{\ol{\psi(yv)}\zeta}&\text{if }\p|2.
    \end{cases}\]
When the central character is trivial, $F=\Q_p$, and $v\in\Z$, this gives
 \begin{equation}\label{PhiQp}
\Phi(\g,f_p)=\begin{cases} 2\zeta\cos(\frac{2\pi yv}{p})&\text{if }p\neq 2
\\{(-1)^v\zeta}&\text{if }p=2.  \end{cases}
\end{equation}
\end{proposition}

\begin{proof}
We need to compute each term of \eqref{Phifp3}.
First, note that for $y\in (\O/\p)^*$,
    \[\g^y=\mat{0}{-u\varpi/y}y{v\varpi}
    =\mat{}t\varpi{}\mat{y/\varpi}{v}{}{-u\varpi/ty}\]
    does not belong to the support of $\fp$.  Hence $\fp(\g^y)=K_\g(\smat y{}{}1)=0$.

Next, suppose $x=\smat y{}{}{\varpi^j}$ with $j>0$ and $y\in\O^*$.  
 Then we use the measure in \eqref{b1}:
\[  K_\g(x)=
     \int_{\O}\int_{\O}\int_{M(1+\p)} \fp(x^{-1}m^{-1}\smat 1{-b}01\smat10{-\varpi c}1\g
     \smat10{\varpi c}1\smat 1b01 mx)dm\, db\, dc.\]
     Note that $m$ commutes with $x$, and lies in the kernel of $\fp$.
     Therefore the integration over $M(1+\p)$ has no effect, and
\[  K_\g(x)=
     \int_{\O}\int_{\O} \fp(x^{-1}\mat1{-b}01\mat10{-\varpi c}1\g
     \mat10{\varpi c}1 \mat1b01x)db\, dc.\]
     Likewise
\[\mat1b01x=
\mat1b01\mat{y}{}{}{\varpi^j}=
\mat{y}{}{}{\varpi^j}\mat1{b\varpi^j/y}{}1.
\]
Note that the right-hand matrix lies in $K'$ since $j>0$, and in fact it is in the kernel
of $\fp$.  Therefore the integral over $b$ also has no effect, and
    \begin{equation}\label{Kj}
     K_\g(x)= \int_{\p}\fp(x^{-1}\mat10{-r}1\g \mat10{r}1 x) dr,
    \end{equation}
     where $dr$ gives $\p$ the measure $1$.

     Taking $\g=w\mat1{v\varpi}0{u\varpi}$ as in \eqref{gp},
\[\mat10{-r}1w\mat1{v\varpi}0{u\varpi}\mat10r1=w\mat1r01\mat{1+rv\varpi}{v\varpi}{ru\varpi}{u\varpi}
\]
\[
=w\mat{1+r^2u\varpi+rv\varpi}{(v+ru)\varpi}{ru\varpi}{u\varpi}.\]
Writing the above as $w\mat abcd \in w\mat{1+\p^2}{\p}{\p^2}{\varpi\O^*}$,
\[x^{-1}\mat10{-r}1\g\mat10r1x=\mat{y^{-1}}{}{}{\varpi^{-j}}w
\mat abcd \mat y{}{}{\varpi^j}\]
\[=w\mat{\varpi^{-j}}{}{}{y^{-1}}\mat{ay}{b\varpi^j}{cy}{d\varpi^j}
=w\mat{ay/\varpi^{j}}bc{d\varpi^j/y}\]
\[     =g_\chi\mat{\varpi^{-1}}{}{}{-t^{-1}}\mat{ay/\varpi^{j}}bc{d\varpi^j/y}
    =g_\chi\mat{ay/\varpi^{j+1}}{b/\varpi}{-c/t}{-d\varpi^j/ty}.\]
 Since the determinant is $u\varpi$, this belongs to the support of $\fp$ 
    if and only if the matrix on the right belongs to $\O^* K'$.
     But this would require $j+1=0$, which is impossible since $j>0$.  Hence 
     \[K_\g(\mat y{}{}{\varpi^j})=0\]
     for all $j>0$ and all $y\in \O^*$.

     Lastly, consider $x=\smat{\varpi^n}{}{}y$ for $n>0$ and $y\in \O^*$.
     We proceed in just the same way, only this time using the 
     coordinates given in \eqref{c1}.  Taking $-b$ in place of $b$ for convenience, and eliminating
the integral over $M(1+\p)$ with the same justification as before,
\[  K_\g(x)=
     \int_{\O}\int_{\O} \fp(x^{-1}
     \mat10{-\varpi c}1 \mat 1{b}01 \g \mat 1{-b}01 \mat10{\varpi c}1 x) db\, dc.\]
     Now
\[\mat10{c\varpi}1x=
\mat10{c\varpi}1 \mat{\varpi^n}{}{}{y}=
\mat{\varpi^n}{}{}{y}\mat1{}{c\varpi^{1+n}/y}1.
\]
The matrix on the right lies in the kernel of $\fp$.
Therefore
\begin{equation}\label{Kk}
    K_\g(x)= \int_{\O}\fp(x^{-1} \mat 1{b}01 \g \mat 1{-b}01 x) db.
\end{equation}
      Taking $\g=w\smat1{v\varpi}{0}{u\varpi}$, we have
\[ \mat 1{b}01 w\mat1{v\varpi}{}{u\varpi} \mat 1{-b}01 = w\mat 10{-b}1\mat{1}{-b+v\varpi}0{u\varpi}
=w\mat 1{-b+v\varpi}{-b}{b^2-v\varpi b+u\varpi}.
\]
Thus letting $P_\g(X)$ denote the characteristic polynomial of $\g$, 
\[x^{-1}\mat1{b}01\g\mat1{-b}01x=
\mat{\varpi^{-n}}{}{}{y^{-1}}w \mat 1{-b+v\varpi}{-b}{P_\g(b)}\mat{\varpi^n}{}{}y\]
\begin{equation}\label{wPg}
 =w\mat{y^{-1}}{}{}{\varpi^{-n}}\mat{\varpi^n}{y(-b+v\varpi)}{-b\varpi^n}{yP_\g(b)}
=w\mat{\varpi^n/y}{-b+v\varpi}{-b}{yP_\g(b)/\varpi^n}
\end{equation}
\[=g_\chi \mat{\varpi^{-1}}{}{}{-t^{-1}}\mat{\varpi^n/y}{-b+v\varpi}{-b}{yP_\g(b)/\varpi^n}
=g_\chi \mat{\varpi^{n-1}/y}{v-b/\varpi}{b/t}{-yP_\g(b)/t\varpi^n}.\]
Since the determinant is $u\varpi$, the above belongs to the support of $\fp$
if and only if the matrix on the right belongs to $\O^* K'$. This means in particular
that $n=1$ and $b\in\p$.  Make the change of
variables $b=c\varpi$, $db=|\varpi|dc=q^{-1}dc$.
Then 
\[P_\g(b)=\varpi(u-vc\varpi+ c^2\varpi),\]
and
\[K_\g(\smat{\varpi}{}{}y)=q^{-1}\int_{\O} 
\fp(g_\chi y^{-1}\mat{1}{yv-cy}{cy\varpi/t}{-y^2(u-vc\varpi +c^2\varpi)/t})dc.\]
From the definition of $K'$, the integrand is nonzero 
if and only if $y^2\equiv -t/u\mod\p$. 
(We have already seen in Proposition \ref{relevant} that $-t/u$ must be a square mod $\p$.)
Assuming this to be the case, then
from \eqref{chi}, \eqref{xzdef}, and \eqref{fp}, we have
\[K_\g(x)=q^{-1} \ol{\w_\p(y^{-1})\zeta}d_\chi\int_{\O}\ol{\psi(yv-cy +cy)}dc
=q^{-1}{\w_\p(y)}\ol{\zeta}d_\chi\ol{\psi(yv)}\]
\begin{equation}\label{K1}
=\frac{q^2-1}{2q}{\w_\p(y)}\ol{\zeta \psi(yv)}.
\end{equation}

To recap, for $\g=\smat0{-u\varpi}1{v\varpi}$, $K_\g(x)=0$ unless $x=\smat{\varpi}{}{}y$ for $y^2\equiv -t/u\mod\p$, 
so \eqref{Phifp3} becomes
\[\Phi(\g,\fp)=\frac{2q}{q^2-1}\sum_{y\in (\O/\p)^*}K_\g(\smat{\varpi}{}{}y)
=\ol{\zeta}\sum_{\e\in\{\pm 1\mod \p\}}\ol{\psi(\e y_0 v)}\w_\p(\e y_0),\]
where $y_0$ is any fixed solution to
$y_0^2\equiv -t/u\mod\p$.  Note that when $\p|2$, we can take $\e=1$, and if $F=\Q_2$ we can
also take $y_0=1$.
\end{proof}

\subsection{The case where $p|T$ and $\g$ is unramified}

We adopt the same notation used in the previous subsection.
Suppose $\g$ is unramified, i.e., $\ord_\p(\det\g)$ is even.  Scaling if needed, we may assume that
$\det\g\in\O^*$.
For the nonvanishing of $\Phi(\g,f_\p)$, it is necessary that some conjugate of
$\g$ belong to the unramified component of the support of $\f_\p$, namely $ZK_\p'$.
Given that $u=\det\g\in\O^*$, this means that $\tr\g$ must also be integral.  So we may take $\g$ in
rational canonical form
\begin{equation}\label{grat}
\g=\mat0{-u}1v
\end{equation}
for some $u\in\O^*$ and $v\in \O$.

\begin{proposition}\label{gunram}
For $\g$ elliptic in $G(F)$ and of the form \eqref{grat},
$\Phi(\g,f_\p)=0$ unless the characteristic polynomial $P_\g$ has a nonzero double
root modulo $\p$: 
\begin{equation}\label{Pgz}P_\g(X)\equiv (X-z)^2\mod \p\end{equation}
for some $z\in(\O/\p)^*$.
  Under this condition (see also the refinement \eqref{**} in Remark 3 below), 
\begin{equation}\label{phiunram}
\Phi(\g,\fp) =\frac{\ol{\w_\p(z)}}q\sum_{n=1}^\infty \sum_{c\mod\p}
\mathcal{N}_\g(c,n)
\sum_{y\in(\O/\p)^*}\psi(\frac{yc}z)
\psi(-\frac{t}{yz})^{\delta(n=1)}
\end{equation}
where $\psi$ is the nontrivial character of $\O/\p$ used in \eqref{chi},
$t\in (\O/\p)^*$ is the parameter of $\sigma_t^\zeta$, and 
\[\mathcal{N}_\g(c,n)=\#\{b\mod \p^{n+1}|\, P_\g(b)\equiv c\varpi^n\mod \p^{n+1}\}.\]
\end{proposition}

\noindent{\em Remarks:} 1. Since $P_\g$ is irreducible over $F$, 
there exists $r$ such that $P_\g(X)\equiv 0
\mod \p^{r}$ has no solution, and hence $\mathcal{N}_\g(c,n)=0$ for all pairs $(c,n)$ with $n\ge r$.
  So the series is actually a finite sum.
When $F=\Q_p$, we can take $r=\ord_p(\Delta_\g)+1$, \cite[Lemma 9.6]{ftf}.
\vskip .1cm

2. When $n=1$ the sum over $y$ is a Kloosterman sum.  When $n>1$, 
\[\sum_y\psi(\frac{yc}z)=\begin{cases}q-1&\text{if }c\equiv 0\mod \p\\-1&\text{otherwise.}\end{cases}\]
\vskip .1cm

%3.  Original version with errors: When $F=\Q_p$, the integer $\mathcal{N}_\g(c,n)$ is given explicitly in \cite[Lemma 9.6]{ftf}, 
%  and presumably there is a version of that lemma for an arbitrary $p$-adic field.
%In particular, $N_\g(c,n)=0$ unless $n\le \ord_p(\Delta_\g)-1$, and for such $n$,
%\[ N_\g(c,n) \le p^{\lfloor\frac{n+1}2\rfloor}\]
%assuming $\g$ is elliptic in $G(\Q_p)$ and satisfies \eqref{Pgz}.
%This gives the following bound for the orbital integral: setting $\delta=\ord_p(\Delta_\g)$,
%\[
%|\Phi(\g,f_p)|\le \sum_{n=1}^{\delta-1}(p-1)(p^{1/2})^{n+1}
%=p(p-1)\sum_{n=0}^{\delta-2}(p^{1/2})^n =p(p-1)\frac{(p^{1/2})^{\delta-1}-1}{p^{1/2}-1}\]
%\[=p(p^{1/2}+1)\left(p^{-1/2}p^{\delta/2} -1\right) \le 2p|\Delta_\g|_p^{-1/2}.\]
%This illustrates the general bound given in \cite[(1.8) and Theorem 3.11]{KST}, according to which
%\[
%|\Phi(\g,f_\p)| \le C\cdot(d_{\sigma_\p})^\eta |\Delta_\g|_\p^{-1/2},
%\]
%where $C>0$ and $\eta<1$ depend only on $G(F)$.

%\scriptsize
3.  (Updated Sept. 2026.) 
We can greatly refine \eqref{phiunram} as follows. For $n\ge 1$,  define
\[M_\g(n)=\mathcal{N}_\g(0,n-1)=\#\{b\mod \p^n|\, P_\g(b)\equiv 0\mod \p^n\}.\]
Each solution to $P_\g(b)\equiv 0\mod \p^n$ has exactly $q$ lifts modulo $\p^{n+1}$,
providing one solution to $P_\g(b)\equiv c\varpi^n\mod \p^{n+1}$ for each $c\mod \p$.
Consequently
\begin{equation}\label{*}\tag{*}
\sum_{c\mod \p}\mathcal{N}_\g(c,n)=qM_\g(n).
\end{equation}
Now by Remark 2 above, for $n>1$ we have
\[\frac1q\sum_{c\mod \p}\mathcal{N}_\g(c,n)\sum_{y}\psi(\tfrac{yc}z)
=\frac1q\Bigl((q-1)\mathcal{N}_\g(0,n)-\sum_{c\neq 0}\mathcal{N}_\g(c,n)\Bigr)
=\mathcal{N}_\g(0,n)-M_\g(n)
=M_\g(n+1)-M_\g(n),\]
and so
\[\frac1q\sum_{n=2}^\infty\sum_{c\mod \p}\mathcal{N}_\g(c,n)\sum_y\psi(\tfrac{yc}z)
=\sum_{n=2}^\infty\Bigl(M_\g(n+1)-M_\g(n)\Bigr)=-M_\g(2)\]
by Remark 1 above.  Fix any lift $z_0$ of $z$ to $\O/\p^2$, and write
  $P_\g(z_0)=c_0\varpi$.  An elementary calculation using \eqref{Pgz} shows that 
$c_0$ is independent of the choice of $z_0$, i.e.,
$P_\g(z_0+c\varpi)\equiv c_0\varpi\mod\p^2$ for all $c\mod \p$.  It follows that 
  $M_\g(2)\in \{0,q\}$ is nonzero if and only if $c_0=0$, 
  and more generally $\mathcal{N}_\g(c,1)\in \{0,q\}$ is nonzero if and only if $c=c_0$.
Consequently
\begin{equation}\label{**}\tag{**}
\Phi(\g,f_\p) =\ol{\w_\p(z)}\cdot\begin{cases}-(q+1)&\text{if }M_\g(2)=q\\
\sum_{y\in(\O/\p)^*}\psi(\frac{yc_0}z)
\psi(-\frac{t}{yz}) &\text{if }M_\g(2)=0.\end{cases}
\end{equation}

Note that $|\Phi(\g,f_\p)|\le q+1$.
This may be compared with the general bound given in \cite[(1.8) and Theorem 3.11]{KST}, according to which
\[
|\Phi(\g,f_\p)| \le C\cdot(d_{\sigma_\p})^\eta |\Delta_\g|_\p^{-1/2},
\]
where $C>0$ and $\eta<1$ depend only on $G(F)$.

\begin{proof} [Proof of Proposition \ref{gunram}]
The first statement was proven in Proposition \ref{relevant}. 
Suppose $\Phi(\g,\fp)\neq 0$ for $\g$ as in \eqref{grat}.
We will compute each term of \eqref{Phifp3}.
It is not hard to check that $\fp(\g^y) =0$ and $K_\g(\smat y{}{}{\varpi^n})=0$,
 since the matrices involved do not 
intersect the support of $\fp$.  Therefore
\begin{equation}\label{Phifp4}\Phi(\g,\fp)=2\sum_{n=1}^\infty\frac{q^n}{q^2-1}
\sum_{y\in (\O/\p)^*}K_\g(\mat{\varpi^n}{}{}y).\end{equation}
Now fix $n\ge 1$ and $y\in\O^*$ and let $x=\smat{\varpi^n}{}{}y$.  As in \eqref{Kk},
we have
 \[   K_\g(x)= \int_{\O}\fp(x^{-1} \mat 1{b}01 \g \mat 1{-b}01 x) db.\]
By a quick calculation (see \eqref{wPg}
 with $u,v$ in place of $u\varpi,v\varpi$),
\begin{equation}\label{xconj}
x^{-1}\mat1b01\g\mat1{-b}01x = \mat b{-yP_\g(b)/\varpi^n}{\varpi^n/y}{v-b}.
\end{equation}
Since the determinant is $u\in\O^*$, this belongs to the support of $f_\p$ only if
it belongs to $\O^*K'$. In particular, $b\in \O^*$ and 
$P_\g(b)\equiv 0\mod \p^n$. 
Therefore $b\in z+\p$ for $z$ as in \eqref{Pgz}.
From \eqref{Pgz} we see that $v\equiv 2z\mod\p$ so $v-b\in z+\p$ as well.
Therefore, pulling out a factor of $z$ from the above matrix,
\[K_\g(\smat{\varpi^n}{}{}y) = \ol{\w_\p(z)}\int_{z+\p} f_\p(\mat1{-yP_\g(b)/z\varpi^n}{\varpi^n/yz}1)db.\]
Writing
$P_\g(b)\equiv c\varpi^n\mod\p^{n+1}$ for some $c\in \O/\p$, by \eqref{fp} the integrand becomes
\[f_\p(\mat1{-yc/z}{\varpi^n/yz}1)=\frac{q^2-1}2\ol{\psi(-\frac{yc}z+\frac{t\varpi^{n-1}}{yz})}.\]
This depends (via $c$) only on the coset $b+\p^{n+1}$ (in fact it depends only on $b+\p^n$ but 
we will not use this).  Each such coset has measure $q^{-(n+1)}$.   Therefore if we let
\[\mathcal{N}_\g(c,n)=\#\{b\mod \p^{n+1}|\, P_\g(b)\equiv c\varpi^n\mod \p^{n+1}\}\]
for $c\in\O/\p$, we find that
\[K_\g(x)=\ol{\w_\p(z)}\frac{q^2-1}{2q^{n+1}}\sum_{c\mod \p}\psi(\frac{yc}z)
\psi(-\frac{t}{yz})^{\delta(n=1)}\mathcal{N}_\g(c,n).\]
Inserting this into \eqref{Phifp4} gives the result.
\end{proof}

\begin{example}\label{Phi2a}  
 For $M\in \Z_2^*$ 
and $f_2$ as in \eqref{fp},
 \[\Phi(\mat{}{-M}1{},f_2)=\begin{cases}1&\text{if }M\equiv 1\mod 4\\
  -3&\text{if } M\equiv 3\mod 8\\
0&\text{if }M\equiv 7\mod 8.\end{cases}\]
\end{example}

\begin{proof}
First, $\g=\smat{}{-M}1{}$ is hyperbolic in $G(\Q_2)$ if and only if $-M$ is a square in $\Q_2$, 
which holds if and only if $M\equiv 7\mod 8$.  In this case, $\Phi(\g, f_2)=0$ by \eqref{hyp}.

 Assuming $M\not\equiv 7\mod 8$, we may apply Proposition \ref{gunram}. We need to determine 
\[\mathcal{N}_\g(0,n) =\text{ number of solutions to }x^2\equiv -M\mod 2^{n+1},\]
 and 
\[\mathcal{N}_\g(1,n) =\text{ number of solutions to }x^2\equiv 2^n-M\mod 2^{n+1}.\]
Given any odd integer $D$, the number of solutions to $x^2\equiv D\mod 2^j$ is
\[\begin{cases}1, & j=1\\
2,&j=2, D\equiv 1\mod 4\\
0,& j=2, D\equiv 3\mod 4\\
4,& j>2, D\equiv 1\mod 8\\
0,&j>2, D\not\equiv 1\mod 8\end{cases}\]
(\cite[Theorem 87]{Land}).
Therefore
\[\hskip -.5cm \mathcal{N}_\g(0,n) =\begin{cases} 
2 &\text{if }n=1\text{ and }M\equiv 3\mod 4\\
0 &\text{if }n=1\text{ and }M\equiv 1\mod 4\\
4 &\text{if }n\ge 2\text{ and }M\equiv 7\mod 8\\
0&\text{if }n\ge 2\text{ and }M\not \equiv 7\mod 8,
\end{cases}\qquad
\mathcal{N}_\g(1,n)=\begin{cases}
0 &\text{if }n=1\text{ and }M\equiv 3\mod 4\\
2 &\text{if }n=1\text{ and }M\equiv 1\mod 4\\
4 &\text{if }n= 2\text{ and }M\equiv 3\mod 8\\
0 &\text{if }n=2\text{ and }M\not\equiv 3\mod 8\\
4&\text{if }n\ge 3\text{ and }M\equiv 7\mod 8\\
0&\text{if }n\ge 3\text{ and }M\not \equiv 7\mod 8.
\end{cases}
\]
By definition, $\psi_2(x)=(-1)^x$ for $x\in \Z$, and $\w_2$ is trivial on $1+2\Z_2=\Z_2^*$.  
So by \eqref{phiunram} and the above,
\[\Phi(\g,f_2)=\frac12[\mathcal{N}_\g(0,1)\psi_2(0)\psi_2(1)+\mathcal{N}_\g(1,1)\psi_2(1)^2]
+\frac12[\mathcal{N}_\g(0,2)\psi_2(0)+\mathcal{N}_\g(1,2)\psi_2(1)]\]
\[=\begin{cases}\frac12[0+2]+\frac12[0+0]=1&\text{if }M\equiv 1\mod 4\\
\frac12[-2+0]+\frac12[0-4]=-3&\text{if }M\equiv 3\mod 8.\end{cases}\qedhere\]
\end{proof}

\begin{example}\label{g3b}
For $f_3$ as in \eqref{fp} and any $m\in \Z_3^*$,
 \[\Phi( \mat0{-m^2}1m, f_3)=\w_3(-m)t\cdot 2_{t=1},\]
 where $2_{t=1}$ is a factor of $2$ which is present only when $t=1$.
Here, $t\in \{\pm 1\}=(\Z/3\Z)^*$ is the parameter of the fixed 
simple supercuspidal representation $\sigma_t^\zeta$ of $G(\Q_3)$.
\end{example}
\noindent{\em Remark:} When $N=3$, $\w_3(-1)=\w'(-1)=(-1)^k$, so
\[\w_3(-m)=\begin{cases}(-1)^k&\text{if }m\in 1+3\Z_3\\1&\text{if }m\in -1+3\Z_3.\end{cases}\]

\begin{proof}
We will apply Proposition \ref{gunram}.  First note that
\[P_\g(X)=X^2-mX+m^2\equiv (X+m)^2\mod 3,\]
so we can take $z=-m$ in \eqref{Pgz}.
We need to find $\mathcal{N}_\g(c,n)=\#\{b|\, b^2-mb+m^2\equiv c3^n\mod 3^{n+1}\}$. 
If $b\in\Z_3^*$ is a double root of $P_\g$ modulo $3$, then by the above we may write $b=-m+3d$, so
\[P_\g(b)=(-m+3d)^2-m(-m+3d)+m^2 = 3m^2 + 9(d^2-md)\in 3\Z_3^*.\]
Thus, $\ord_3(P_\g(b))=1$, which means that $\mathcal{N}_\g(c,n)=0$ for all $n\ge 2$, and also
$\mathcal{N}_\g(0,1)=0$.  Some elementary calculations show that independently of $m$,
$\mathcal{N}_\g(-1,1)=0$ and $\mathcal{N}_\g(1,1)=3$. 
In view of \eqref{phiunram}, this means
\[
\Phi(\g,f_3)= \frac{\ol{\w_3(-m)}}3\mathcal{N}_\g(1,1)\Bigl(\psi_3(\tfrac1{-m})\psi_3(\tfrac{-t}{-m})
+\psi_3(\tfrac1m)\psi_3(\tfrac{-t}m)\Bigr)
=\ol{\w_3(-m)}\Bigl(e(\tfrac{1-t}{-3m})+e(\tfrac{1-t}{3m})\Bigl)\]
\[=\w_3(-m)\Bigl[e(\frac{t-1}3)+e(\frac{1-t}3)\Bigr].\]
When $t=1$ (resp. $t=-1$), the expression in the brackets equals $2$ (resp. $-1$).
\end{proof}

\subsection{The case where $p|S$}

When $p|S$, the support of $f_p$ is contained in $Z_pK_p$, so the orbital integral vanishes unless $\g$
is unramified.
We again work over a $p$-adic field $F$, with the usual notation, and fix a depth zero
supercuspidal representation $\sigma_{\nu}$ of $G=G(F)$ for $\nu$ a primitive character of
$\F_{q^2}^*$.

\begin{proposition}\label{elld0}
Let $f_\p$ be the test function defined in \eqref{fpd0}, and let $\g=\mat{}{-u}1v$ be an elliptic
element of $G(F)$, where $u\in\O^*$ and $v\in \O$.  If there exists $z\in (\O/\p)^*$ such that
\[  P_\g(X)\equiv (X-z)^2\mod \p,\]
then (updated Sept. 2026)
\begin{equation}\label{ZU}
\Phi(\g,\fp)= -2\ol{\w_\p(z)}.
\end{equation}
% Original version - correct but not simplified:
%\begin{equation}\label{ZU}
%\Phi(\g,\fp)= -\ol{\w_\p(z)}+\frac{\ol{\w_\p(z)}}q\sum_{n=1}^\infty\left[
%(q-1)\mathcal{N}_\g(0,n)-\sum_{c\in (\O/\p)^*}\mathcal{N}_\g(c,n)\right],
%\end{equation}
%where $\mathcal{N}_\g(c,n)=\#\{b\mod \p^{n+1}|\, P_\g(b)\equiv c\varpi^n\mod \p^{n+1}\}$. 

If $P_\g(X)$ is irreducible modulo $\p$, then
\begin{equation}\label{Tcase}
\Phi(\g,\fp)=-\ol{\nu(\g)}-\ol{\nu^q(\g)},
\end{equation}
where we interpret the above to mean $-\ol{\nu(x)}-\ol{\nu^q(x)}$ if $x\in \F_{q^2}^*$ has the
 same minimum polynomial over $\F_q$ as the reduction of $\g$ mod $\p$, i.e., $\g$ is conjugate
to $x\in \T$.
\end{proposition}
\noindent{\em Remark:}  The remaining 
possibility where $P_\g(X)$ has two distinct roots mod $\p$ cannot
occur due to Hensel's Lemma, since we are assuming that $\g$ is elliptic in $G(F)$.

% No longer needed
%2. See the remarks after Proposition \ref{gunram} regarding $\mathcal{N}_\g(c,n)$. 
%In particular, the sum in \eqref{ZU} is finite, and when $F=\Q_p$ we find
%$|\Phi(\g,f_p)|\le 1+4|\Delta_\g|_p^{-1/2}$. 

\begin{proof}
In this proof we write $\olG$ for $\olG(F)$, $Z$ for $Z(F)$, and $K$ for $G(\O)$. 
  By Lemma \ref{Ilem},
\[\Phi(\g,\fp)=\int_{\olG}\fp(g^{-1}\g g)dg
=\sum_{x\in \ol{K}\bs \olG/\ol{K}}\meas_{\olG}(\ol{K}x\ol{K})\int_{\ol{K}}\int_{\ol{K}}
\fp(h_2^{-1}x^{-1}h_1^{-1}\g h_1 xh_2)dh_1dh_2,\]
with $dh_1$ and $dh_2$ each having total measure $1$.
The integrand is nonzero only if $x^{-1}h_1^{-1}\g h_1x\in ZK$.  Therefore,
 since $\fp$ is a trace, $h_2$ has no effect, so
\[\Phi(\g,\fp)=\sum_{x\in \ol{K}\bs \olG/\ol{K}}\meas_{\olG}(\ol{K}x\ol{K})\int_{\ol{K}}
\fp(x^{-1}h\g h^{-1} x)dh.\]
(For convenience in what follows, we have set $h=h_1^{-1}$.)

By the Cartan decomposition of $G$, a set of representatives for
$\ol{K}\bs \olG/\ol{K}$ is given by
\[\left\{\mat{\varpi^n}{}{}1|\, n\ge 0\right\}.\]
Arguing as in \cite[Lemma 4.5.6(2)]{M}, for $x=\smat{\varpi^n}{}{}1$,
\[|K\bs KxK| =\begin{cases}q^{n-1}(q+1)&\text{if } n>0\\ 1&\text{if }n=0.\end{cases}\]
Therefore
  $\meas_{\olG}(\ol{K}x\ol{K})=q^{n-1}(q+1)$ when $n>0$, so
\begin{equation}\label{Phign}\Phi(\g,\fp)=\fp(\g) + \sum_{n=1}^\infty q^{n-1}(q+1)K_\g(n),
\end{equation}
where
\[K_\g(n)=\int_{K}\fp(\mat{\varpi^{-n}}{}{}1h\g h^{-1} \mat{\varpi^n}{}{}1)dh\qquad(n>0).\]
(We may integrate over $K$ since $K$ and $\ol{K}$ both have measure 1.)
Write $h\g h^{-1}=\smat wxyz \in K$.  Then
\begin{equation}\label{hcong}
\mat{\varpi^{-n}}{}{}1 \mat wxyz \mat{\varpi^n}{}{}1=\mat w{x/\varpi^n}{y\varpi^n}z.
\end{equation}
This belongs to the support of $\fp$ only if $x\in \p^n$. 

In the integrand above, we can freely multiply $h$ by a diagonal element of $K$ since
such an element commutes with $x$ and can be eliminated since $\fp$ is a trace.
In particular, we can assume $\det h=1$.  Write $h=\smat abcd$, $\det h=1$. Then
\[h\g h^{-1} = \mat *{-b^2+abv-a^2u}**.\]
If $a\in \p$ then we must have $b\in \O^*$ since $\smat abcd\in K$.  But then the upper
right entry above cannot belong to $\p^n$, so the integrand vanishes by \eqref{hcong}.
Therefore we may assume $a\in\O^*$, i.e., $h\in A$, where
\[A=\mat{\O^*}***\cap K.\]
Let's find the measure of $A$.  Let $K(\p)=1+M_2(\p)\subset A$.  This is the kernel of the reduction
 mod $\p$ map $K\rightarrow G(\O/\p)$.  
Since $|G(\O/\p)|=(q^2-1)(q^2-q)$, we see that $\meas(K(\p))=\frac1{(q^2-1)(q^2-q)}$. 
Let $\ol A = A \mod K(\p)$. Thinking of $\ol A$ as a set of matrices in $G(\O/\p)$, we see that
\[|\ol{A}| = (q-1)q(q^2-q).\]
(There are $(q-1)q$ possible top rows, and then $q^2-q$ remaining choices for the bottom row.)
Hence
\[\meas(A)=\frac{(q-1)q(q^2-q)}{(q^2-1)(q^2-q)}=\frac q{q+1}.\]

It is not hard to show that 
\[A=\mat{\O^*}{}{}{\O^*}\mat 1{}\O1\mat1\O{}1,\]
and that the corresponding decomposition of any element of $A$ is unique.  
Therefore by Lemma \ref{GHK} we can use the above as a coordinate system for integration over $A$.
Since, as noted above, the diagonal component has no effect on the value of the integral, we have
\[K_\g(n)=\frac q{q+1}\int_\O\int_\O \fp(x^{-1}\mat1{}{c}1\mat1b{}1\g\mat1{-b}{}1\mat1{}{-c}1 x)
db\,dc,\]
where $x=\smat{\varpi^n}{}{}1$, and $db$ and $dc$ each have total measure $1$.
The integral over $c$ can be eliminated, since $\smat1{}{-c}1\smat{\varpi^n}{}{}1=\smat{\varpi^n}{}{}1
\smat 1{}{-\varpi^n c}1$, and the rightmost matrix belongs to $K$.
Therefore
\[K_\g(n)=\frac q{q+1}\int_\O \fp(\mat{\varpi^{-n}}{}{}1\mat1b{}
  1\mat{}{-u}1v\mat1{-b}{}1 \mat{\varpi^n}{}{}1) db\]
\[=\frac q{q+1}\int_{\O^*} \fp(\mat{b}{-P_\g(b)/\varpi^n}0{v-b})db\]
as in \eqref{xconj}.  (As a reminder, $db$ is additive measure.)
We have replaced the lower left entry by $0$, using the fact that by definition (see \eqref{fpd0}),
  $\fp$ is sensitive only to the reduction of its argument mod $\p$.
Further, the integrand is nonzero only if $P_\g(b)\in\p^n$.  Under this condition, 
given that the characteristic polynomial of the matrix in the integrand is $P_\g(X)$, and this 
cannot have distinct roots mod $\p$ as $\g$ is elliptic in $G(\Q_p)$,
%given that $P_\g$ is elliptic in $G(\Q_p)$, 
%the matrix (viewed modulo $\p$) does not lie in $\T$ since its characteristic polynomial is reducible. 
%  Therefore the integrand is nonzero if and only if 
% This holds if and only if $P_\g(b)\in \p^n$.  In this case 
   there exists $z\in (\O/\p)^*$ such that 
  $b\equiv v-b\equiv z \mod \p$.
In particular, the matrix (viewed modulo $\p$) belongs to $ZU$, with notation as in \eqref{trho}.
  Write $P_\g(b)\equiv c\varpi^n\mod\p^{n+1}$, for $c\in \O/\p$. 
The integrand becomes
\[\ol{\w_\p(z)}\fp(\mat{1}{-c/z}0{1})=\begin{cases}\ol{\w_\p(z)}(q-1)&\text{if }c=0\\
  -\ol{\w_\p(z)}&\text{if }c\in (\O/\p)^*.\end{cases}\]
This depends (via $c$) only on the coset $b+\p^{n+1}$, which has measure $q^{-(n+1)}$. 
Therefore
\[K_\g(n)=\ol{\w_\p(z)}\cdot \frac q{q+1}\cdot\frac1{q^{n+1}}\Bigl[(q-1)\mathcal{N}_\g(0,n)
-\sum_{c\in(\O/\p)^*}\mathcal{N}_\g(c,n)\Bigr].\]
 Plugging the above into \eqref{Phign}, we obtain
\[\Phi(\g,\fp)= -\ol{\w_\p(z)}+\frac{\ol{\w_\p(z)}}q\sum_{n=1}^\infty\Bigl[
(q-1)\mathcal{N}_\g(0,n)-\sum_{c\in (\O/\p)^*}\mathcal{N}_\g(c,n)\Bigr],
\]
where $\mathcal{N}_\g(c,n)=\#\{b\mod \p^{n+1}|\, P_\g(b)\equiv c\varpi^n\mod \p^{n+1}\}$. 
(Note that $f_\p(\g)=-\ol{\w_\p(z)}$ in this case, since
$\g - z =\smat{-z} {-u}1{v-z}\not\equiv 0\mod \p$, so $\g$ is conjugate mod $\p$ to $zu$ for 
some $1\neq u\in U$.)
The above simplifies to $-2\ol{\w(z)}$ by the argument given in 
  Remark 3 after Proposition \ref{gunram}, noting that $M_\g(1)=1$ due to the double root.

By the above discussion $K_\g(n)=0$ for all $n>0$ if $P_\g(X)$ is irreducible mod $\p$.  So
in this case \eqref{Phign} gives $\Phi(\g,\fp)=\fp(\g) = -\ol{\nu(\g)}-\ol{\nu^q(\g)}$ by \eqref{trho}
and \eqref{fpd0}.
\end{proof}

\section{General dimension formula, and examples with $N=S^2T^3$}\label{dimex}

When $\n=1$, the list of relevant $\g$ in Theorem \ref{stf2}
can be simplified.  The result is the following general dimension formula. 

\begin{theorem}\label{dimST}
Let $N=\prod_{p|N}p^{N_p}>1$ with $N_p\ge 2$ for all $p|N$.  Fix $k>2$ and 
  a tuple $\widehat{\sigma}=(\sigma_p)_{p|N}$
of supercuspidal representations so that $S_k(\widehat{\sigma})\subset S_k^{\new}(N,\w')$ for a Dirichlet
character $\w'$, as detailed at the beginning of \S\ref{count}. 
Let $T$ be the product of all primes $p|N$ with $N_p$ odd.  
 Let $f=f^1$ be the test function defined in \eqref{fn} with $\n=1$ but with $f_p$ chosen 
  as in \eqref{fpT} below for all $p|T$.
Then 
\[\dim S_k(\widehat{\sigma}) =
 \frac{k-1}{12}\prod_{p|N}d_{\sigma_p}
+\frac12\Phi(\mat{}{-T}1{},f)+\frac{\delta_{T\in 2\Z^+}}2\Phi(\mat{}{-T/2}1{},f)\]
\[+\delta_{T=2}\Phi(\mat{0}{-2}12,f)+\delta_{T=3}\Phi(\mat{0}{-3}13,f)+\delta_{T\in\{1,3\}}
\Phi(\mat{0}{-1}11,f).\]
Here, $d_{\sigma_p}$ is the formal degree of $\sigma_p$ relative to the Haar measure fixed in \S\ref{notation},
and the orbital integrals $\Phi(\g,f)$ are given as in Theorem \ref{stf2}.
\end{theorem}
\begin{proof}
The case where $T=1$ is already contained in Theorem \ref{stf2} by taking $\n=1$.  The simplifications
when $T>1$ are proven in Proposition \ref{glist} below.
\end{proof}

As with Theorem \ref{mainST}, using the results of \S\ref{pN}
we can compute the above explicitly in any case of interest when $N=S^2T^3$ with $S$ and $T$
square-free relatively prime positive integers.
Although there is not a particularly nice formula for all such levels, as an illustration
we will work everything out in the two special cases where $N=S^2$ and $N=T^3$. These results are
stated in \S\ref{dimS} and \S\ref{dim} respectively.  In \S\ref{ng1} we give some examples
to illustrate Theorem \ref{mainST} with $\n>1$.

First, we highlight the following consequence of Theorem \ref{dimST}.

\begin{corollary}\label{kodd} In the setting of Proposition \ref{dimST} above, 
suppose that the weight $k$ is odd, so $\w'(-1)=-1$.
For $T$ as in Theorem \ref{dimST}, if $T> 3$
    the elliptic terms vanish, so 
        \[\dim S_k(\widehat{\sigma})
=\frac{k-1}{12} \prod_{p|N}d_{\sigma_p}
 \qquad(k>2\text{ odd, } T>3).\]
\end{corollary}
\noindent{\em Remark:} If $N=2^2$ or $N=2^3$, then $S_k(\widehat{\sigma})$ is undefined when $k$ is odd
since by Proposition \ref{k2N2} there is no appropriate nebentypus.
\begin{proof}
    If $\g=\smat{}{-T}1{}$ or $\smat{}{-T/2}1{}$, then $\Phi(\g,f_\infty)=0$ when $k$ is odd, by
\eqref{Phiinf}.
\end{proof}

\subsection{Dimension formula and root number bias when $N=S^2$}\label{dimS}

When we set $T=1$ and take $N=S^2$, the formula in Theorem \ref{dimST} gives the following.

\begin{theorem}\label{d0dim}
Let $N=S^2$ for $S$ squarefree, $k>2$, $\w'$ a Dirichlet character of modulus $N$ and conductor
dividing $S$, and let $\widehat{\sigma}=(\sigma_p)_{p|S}$ be a tuple of depth zero
supercuspidal representations chosen compatibly with $\w'$ as in \S \ref{test}, with $T=1$.  
Then the subspace $S_k(\widehat{\sigma})\subset S_k^{\new}(S^2,\w')$ has dimension
\[\dim S_k(\widehat{\sigma}) = \frac{k-1}{12}\prod_{p|N}(p-1) + A_1 + A_2,\]
where
\begin{equation}\label{A1}
A_1 = \frac14(-1)^{S+1+k/2}\delta_{k\in2\Z}\prod_{\text{odd }p|N}(-\ol{\nu_p(\alpha)}
-\ol{\nu_p^p(\alpha)})\delta_{p\equiv 3\mod 4}
\end{equation}
where $\nu_p$ is the primitive character of $\F_{p^2}^*$ defining $\sigma_p$
and $\alpha\in \F_{p^2}^*$ satisfies $\alpha^2=-1$, and
\begin{equation}\label{A2}
A_2 =
\frac{\delta_{k\equiv 0,2\mod 3}}{3}(-1)^{\delta_{k\equiv 2,3\mod 6}}
\bigl(-\w_3(-1)\bigr)^{\delta(3|N)}\prod_{p|N,\atop{p\neq 3}}(-\ol{\nu_p(\beta)}-\ol{\nu_p^p(\beta)})
  \delta_{p\equiv 2\mod 3},
\end{equation}
where $\beta\in\F_{p^2}^*$ satisfies $\beta^2-\beta+1=0$. 
\end{theorem}

\noindent{\em Remarks:}  1. Note that $A_1=A_2=0$ in each of the following situations:
(i) $k\equiv 1\mod 6$, (ii) there exist primes $p,q|N$ (which could be equal) such that 
 $p\equiv 1\mod 4$ and $q\equiv 1\mod 3$, 
(iii) $k$ is odd and $p\equiv 1\mod 3$ for some $p|N$, 
(iv) $k\equiv 1\mod 3$ and $p\equiv 1\mod 4$ for some $p|N$. 
\vskip .1cm

 2. By summing the above formula over all tuples $\widehat{\sigma}$, one obtains a formula
for the dimension of the space $S_k^{\min}(S^2,\w')$ of twist-minimal newforms. See Proposition
\ref{Smin}.
\vskip .1cm

 3. Theorem \ref{d0thm} from the introduction follows from the above by taking $\w'$ trivial.
  We will prove this after first proving the above result.

\begin{proof}
Taking $T=1$ in Theorem \ref{dimST}, we have
\[\dim S_k(\widehat{\sigma}) = \frac{k-1}{12}\prod_{p|N}(p-1)+\frac12\Phi(\mat{}{-1}1{},f)
+\Phi(\mat{}{-1}11,f).\]
Consider $\g=\smat{}{-1}1{}$. Its discriminant is $\Delta_\g=-4$, and we adopt the shorthand
\[\Phi(\g)=m\Phi_\infty\Phi_2\prod_{\text{odd }p|N}\Phi_p\]
 for \eqref{phifact}, where $m=\frac{2h(E)}{w(E)2^{\w(d_E)}}$ for
  $E=\Q[\g]$.  We find that $m=\tfrac14$ and $\Phi_\infty=(-1)^{k/2}\delta_{k\in 2\Z}$.
If $S$ is odd, then $\Phi_2 = 2$ by Example \ref{Nell}.
If $S$ is even, $\Phi_2=-2$ by \eqref{ZU}. 
%Here, $\mathcal{N}_\g(c,n)=0$ for all $n\ge 2$, $\mathcal{N}_\g(0,1)=0$ and $\mathcal{N}_\g(1,1)=2$. 
%So $\Phi_2 = -1 +\frac12(-2) = -2$.  
Thus in both cases, $\Phi_2=2(-1)^{S+1}$. 
Finally, for odd $p|S$, $\g$ is elliptic in $G(\Q_p)$ if and only if $-1$ is not a square in $\Q_p$,
 i.e., $p\equiv 3\mod 4$.  In such cases, $P_\g(X)$ is irreducible modulo $p$, so by \eqref{Tcase}
$\Phi_p = -\ol{\nu_p(\g)}-\ol{\nu_p^p(\g)}$. 
Multiplying everything together, we see that
$\frac12\Phi(\smat{}{-1}1{},f)$ gives \eqref{A1}.

Now consider $\g=\smat{}{-1}11$. Then $\Delta_\g=-3$, so 
\[\Phi(\g)=m\Phi_\infty\Phi_3\prod_{p|N,\atop{p\neq 3}}\Phi_p.\]
We find that $m=\frac16$, and 
\[\Phi_\infty = -\frac{\sin((k-1)\pi/3)}{\sin(\pi/3)}=\begin{cases}0&\text{if }k\equiv 1\mod 3\\
1&\text{if }k\equiv 0,5\mod 6\\-1&\text{if }k\equiv 2,3\mod 6.\end{cases}\]
If $3\nmid N$, then by Proposition \ref{elle}, $\Phi_3=2$ since $3$ is ramified
 in $\Q_3[\g]=\Q_3[\sqrt{-3}]$ 
 and $\O_\g=\Z_3[\frac{1+\sqrt{-3}}2] =\Z_3[\sqrt{-3}]$ is the full ring of integers.
If $3|N$, then $\Phi_3=-2\w_3(-1)$ by \eqref{ZU} with $z=-1$.  
%We find that $\mathcal{N}_\g(c,n)=0$ for all $n\ge 2$, $\mathcal{N}_\g(1,1)=3$, and 
%$\mathcal{N}_\g(0,1)=\mathcal{N}_\g(2,1)=0$.  So
%\[\Phi_3=-\ol{\w_3(-1)}+\frac{\ol{\w_3(-1)}}3(-3)=-2\w_3(-1).\]
For $p|N$ with $p\neq 3$, $P_\g(X)=X^2-X+1$ is irreducible in $\Q_p$ if and only if 
  $-3$ is not a square in $\Q_p$, or equivalently, $p\equiv 2\mod 3$ (cf. \cite[Lemma 27.4]{KL}).
For such $p$, $\Phi_p$ is given by \eqref{Tcase}.
Multiplying these factors together gives \eqref{A2}, and the theorem follows.
\end{proof}

Now suppose $\w'$ is trivial, so $k>2$ is even.
In this case we can simplify the expressions for $A_1$ and $A_2$
to obtain Theorem \ref{d0thm}, as follows.

\begin{proof}[Proof of Theorem \ref{d0thm}]
Recall that by \eqref{ed0b}, when $p\equiv 3\mod 4$ and $\w_p$ is trivial,
 $-\nu_p(\alpha)=-\nu_p^p(\alpha) =\epsilon_p$
is the root number of $\sigma_p$.  Likewise, by \eqref{ed0a}, $(-1)^{S+1}=\epsilon_2$ 
when $S$ is even (and $1$ otherwise).
So in this case, we simply have
\begin{equation}\label{A1b}
A_1=\frac{\epsilon(k,\widehat{\sigma})}4 D_4(S)\prod_{\text{odd }p|S}2,
\end{equation}
where $\epsilon(k,\widehat{\sigma})=(-1)^{k/2}\prod_{p|S}\epsilon_p$ is the common global root number
of the newforms in $S_k(\widehat{\sigma})$, and $D_4(S)\in\{0,1\}$ vanishes exactly when $S$ is
divisible by a prime $p\equiv 1\mod 4$. 

Turning to \eqref{A2}, if $p\equiv 2\mod 3$,
the polynomial $X^2-X+1$ is irreducible over $\F_p$.  So $L=\F_{p^2}$ has a root $\beta\in L^*-\F_p^*$. 
Let $t$ be a generator of the cyclic group $L^*$.
The dual group of $L^*$ is the set $\{\nu_m|\, 1\le m\le p^2-1\}$, where 
$\nu_m=\nu_{p,m}$
 is defined by
\[\nu_m(t)=\nu_{p,m}(t)=e(\frac {m}{p^2-1}).\]
Suppose $p$ is odd.
As shown in the proof of Corollary \ref{ecount}, the list of depth zero supercuspidal representations
of $G(\Q_p)$ with trivial central character is $\{\sigma_{\nu_{p-1}},\sigma_{\nu_{2(p-1)}},\ldots
\sigma_{\nu_{\frac{p-1}2(p-1)}}\}$.
So
there exists $m=k(p-1)$ such that 
the primitive character $\nu_p$ of $\F_{p^2}^*$ fixed in Theorem \ref{d0thm} is given by
\[\nu_p = \nu_{p,m}=\nu_m.\]
Hopefully this conflict of notation ($\nu_p=\nu_m$) 
will cause no confusion, since $m$ cannot equal $p$.

Noting that $\beta^3=-1$, we can take $\beta = t^{\frac{p^2-1}6}$. 
Then for $m=k(p-1)$, 
\[\nu_m(\beta)=e(\frac{k(p-1)(p^2-1)}{6(p^2-1)})=e(\frac{k(p-1)}6)=\begin{cases} 1&\text{if }3|k\\
-\frac12\pm i\frac{\sqrt3}2&\text{otherwise.}\end{cases}\]
Therefore
\begin{equation}\label{A2b}
B(\nu_p):=-\ol{\nu_p(\beta)}-\ol{\nu_p(\beta^p)} = -2\Re(\nu_m(\beta)) =\begin{cases}-2&\text{if }3|k\\
1&\text{if }3\nmid k.\end{cases}
\end{equation}
 When $p=2$, there is only one supercuspidal, corresponding to $m=k=1$, we can take $t=\beta$,
and \eqref{A2b} holds as well.
By \eqref{numorder}, $3|k$ if and only if the order of
$\nu_m$ divides $\frac{p+1}3$.  So the above coincides with $B(\nu_p)$ defined in Theorem \ref{d0thm},
and the theorem follows from \eqref{A1b} and \eqref{A2b}.
\end{proof}

Next we will use Theorem \ref{d0thm} to count the locally supercuspidal newforms 
  of level $S^2$ with a given global root number.  (What we will actually compute is the bias in global
root number, but the count for each sign could be determined easily by following the proof 
of Proposition \ref{Sbias}.)

To understand the impact of the local root numbers on the product of $B(\nu_p)$ in \eqref{d0dima},
 the primes of interest are equivalent to $2\mod 3$, 
so aside from $p=2$, we have $p\equiv 5\mod 6$.
It is helpful to look at two typical examples:
\begin{equation}\label{B17}        \begin{array}{l|c||c|c|c|c|c|c|c|c}
            \rule[-3mm]{0mm}{8mm}
p=11&\nu&\nu_{10}&\nu_{20}&\nu_{30}&\nu_{40}& \nu_{50}\\
\cline{2-7}
            \rule[-3mm]{0mm}{8mm}
&AL&+&-&+&-&+\\
\cline{2-7}
            \rule[-3mm]{0mm}{8mm}
&B(\nu)&1&1&-2&1&1\\
     \hline
     \hline
            \rule[-3mm]{0mm}{8mm}
p=17&\nu&\nu_{16}&\nu_{2\cdot 16}&\nu_{3\cdot16}&\nu_{4\cdot16}&\nu_{5\cdot 16}&\nu_{6\cdot 16}
&\nu_{7\cdot 16}&\nu_{8\cdot16}\\
\cline{2-10}
            \rule[-3mm]{0mm}{8mm}
&AL&+&-&+&-&+&-&+&-\\
\cline{2-10}
            \rule[-3mm]{0mm}{8mm}
&B(\nu)&1&1&-2&1&1&-2&1&1
 \end{array}
\end{equation}
The Atkin-Lehner sign in the second row comes from \eqref{ek}, and the third row is from \eqref{A2b}.

\begin{lemma}\label{B}
Given $S>1$ square-free, let $H_S^+$ (resp. $H_S^-$) 
 denote the set of tuples $\widehat{\sigma}=(\sigma_p)_{p|S}$ satisfying:
\begin{itemize}
\item For each $p|S$, $\sigma_p$ has trivial central character and conductor $p^2$
\item $\prod_{p|S}\epsilon_p = 1$ (resp. $-1$), where $\epsilon_p$ is the root number of $\sigma_p$.
\end{itemize}
For $\nu_p$ the primitive character of $\F_{p^2}^*$ attached to $\sigma_p$, and $B(\nu_p)$
defined in \eqref{A2b}, define
\[\mathcal{B}(S)^\pm = \sum_{\widehat{\sigma}\in H_S^\pm}\prod_{p|S,\atop{p\neq3}}B(\nu_p).\]
Suppose $D_3(S)=1$ (in the notation of Theorem \ref{d0thm}), and let
$\w(S)$ denote the number of prime factors of $S$.  Then if $\gcd(S,6)=1$,
\begin{equation}\label{BS}
\mathcal{B}(S)^+ =\begin{cases}
2^{\w(S)-1}&\text{if there exists $p|S$ with }p\equiv 5\mod 12,\\
2^{\w(S)}&\text{if $\w(S)$ is even and }p\equiv 11\mod 12\text{ for all $p|S$,}\\
0&\text{if $\w(S)$ is odd and }p\equiv 11\mod 12\text{ for all $p|S$,}
\end{cases}
\end{equation}
and $\mathcal{B}(S)^-$ is the same but with ``even" and ``odd" interchanged,
 i.e., $\mathcal{B}(S)^-=2^{\w(S)}-\mathcal{B}(S)^+$.

If $S$ is odd and $3|S$, then $\mathcal{B}(S)^\pm =\mathcal{B}(\frac S3)^\pm$ if $S>3$,
 and $\mathcal{B}(3)^+=1$, $\mathcal{B}(3)^-=0$.

If $S$ is even, then $\mathcal{B}(S)^\pm=\mathcal{B}(\frac S2)^\mp$ if $S>2$, and $\mathcal{B}(2)^+ =0$,
$\mathcal{B}(2)^-=1$.  
\end{lemma}

\begin{proof}
Suppose $\gcd(S,6)=1$.  We prove \eqref{BS} by induction on $\w(S)$.  For the base case, we take $S=p$ for 
a prime $p\equiv 5\mod 6$.  As in \eqref{B17}, there are $\frac{p+1}3$ representations with
 $B(\nu_p)=1$, of which
$\frac{p+1}6$ have $\epsilon_p=1$ and $\frac{p+1}6$ have $\epsilon_p=-1$.  There are 
$\frac{p-5}6$ representations with $B(\nu_p)=-2$, of which $\lceil\frac{p-5}{12}\rceil$ have $\epsilon_p=1$, and
$\lfloor \frac{p-5}{12}\rfloor$ have $\epsilon_p=-1$.
Therefore
\[\mathcal{B}(p)^+=\sum_{\sigma_p\in H_p^+}B(\nu_p) = \frac{p+1}6-2\Bigl\lceil\frac{p-5}{12}\Bigr\rceil
=\begin{cases}1&\text{if }p\equiv 5\mod 12\\0&\text{if }p\equiv 11 \mod 12.\end{cases}\]
Likewise
\[\mathcal{B}(p)^-=\sum_{\sigma_p\in H_p^-}B(\nu_p) = \frac{p+1}6-2\Bigl\lfloor\frac{p-5}{12}\Bigr\rfloor
=\begin{cases}1&\text{if }p\equiv 5\mod 12\\2&\text{if }p\equiv 11 \mod 12.\end{cases}\]
This proves the base case.  Now suppose \eqref{BS} holds for some $S>1$ with $\gcd(S,6)=1$, and
$\ell\equiv 5\mod 6$ is a prime not dividing $S$.  Then the result follows, by considering cases,
 from the fact that
\[\mathcal{B}(S\ell)^+=\mathcal{B}(S)^+\cdot\mathcal{B}(\ell)^++\mathcal{B}(S)^-\cdot\mathcal{B}(\ell)^-\]
and 
\[\mathcal{B}(S\ell)^-=\mathcal{B}(S)^+\cdot\mathcal{B}(\ell)^-+\mathcal{B}(S)^-\cdot\mathcal{B}(\ell)^+.\]

When $3|S$, the claim follows from the fact that there is a unique depth zero supercuspidal 
representation of $\PGL_2(\Q_3)$, and it has root number $+1$ (see Corollary \ref{ecount}).  When $2|S$,
 the claim follows from the fact that there is a unique depth zero supercuspidal representation 
$\sigma_\nu$ of $\PGL(\Q_2)$, and it has $B(\nu)=1$ and root number $-1$.
\end{proof}

\begin{lemma}\label{globalecount}
Let $H_S^+$ and $H_S^-$ be defined as in Lemma \ref{B} above.  As in Theorem \ref{d0dima},
define $D_4(S)\in\{0,1\}$ to be $0$ if and only if $S$ is
divisible by a prime $p\equiv 1\mod 4$.  Then
\[|H_S^\pm|=\begin{cases} \frac12\prod_{\text{odd }p|S}\frac{p-1}2&\text{if }D_4(S)=0\\
 \frac12\prod_{\text{odd }p|S}\frac{p-1}2\pm \frac{(-1)^{\delta(2|S)}}2&\text{if }D_4(S)=1.
\end{cases}
\]
\end{lemma}
\begin{proof}
For each odd prime $p$, there are $\frac{p-1}2$ depth zero supercuspidals with trivial central 
character (see \S\ref{d0}).  For $p=2$, there is only one.  Therefore for all square-free
 $S>1$, the total number of tuples $\widehat{\sigma}=(\sigma_p)_{p|S}$ with each $\sigma_p$
having depth zero and trivial central character is
\begin{equation}\label{sigtot}
|H_S^+|+|H_S^-|=\prod_{\text{odd }p|S}\frac{p-1}2.
\end{equation}

Now suppose $S$ is divisible by a prime $p_0\equiv 1\mod 4$.  Fix $\epsilon_{\fin}=\pm 1$.
By the above, the number of tuples $(\sigma_p)_{p|\frac{S}{p_0}}$
is $\prod_{\text{odd }p|\frac{S}{p_0}}\frac{p-1}2$.
 Having fixed one such tuple, by Corollary \ref{ecount}
 there are then $\frac{p_0-1}4$
choices for $\sigma_{p_0}$ subject to $\prod_{p|S}\epsilon_{\sigma_p}=\epsilon_{\fin}$.
This proves the result when $D_4(S)=0$.

Now suppose $p\equiv 3\mod 4$ for all odd $p|S$.  For this case, in view of \eqref{sigtot},
 the given formula is equivalent to
$|H_S^+|-|H_S^-|=(-1)^{\delta(2|S)}$.   We will prove the latter by induction
on the number $\w(S)$ of primes dividing $S$.  If $S=2$, the given formula holds since there is just 
one representation $\sigma_2$, and it has $\epsilon_{\sigma_2}=-1$.  If $S=p\equiv 3\mod 4$, 
the given formula holds by Corollary \ref{ecount}.
Having established the base case, suppose now that the given formula holds for some 
$S$ satisfying $D_4(S)=1$, and that $p_0\nmid S$ is a prime satisfying $p_0\equiv 3\mod 4$.  
We construct a tuple $\widehat{\sigma}=(\sigma_p)_{p|Sp_0}$
by first choosing the components at $p|S$, and then at $p_0$.  
Let $P=|H_S^+|$ and $Q=|H_{p_0}^+|$, so $|H_{S}^-|=P-(-1)^{\delta(2|S)}$ and $|H_{p_0}^-|=Q-1$
by the inductive hypothesis.
Then
\[|H_{Sp_0}^+|=PQ+(P-(-1)^{\delta(2|S)})(Q-1)=2PQ-P-(-1)^{\delta(2|S)}Q+(-1)^{\delta(2|S)},\]
and
\[|H_{Sp_0}^-|=P(Q-1)+(P-(-1)^{\delta(2|S)})Q=2PQ-P-(-1)^{\delta(2|S)}Q.\]
Subtracting, 
\[|H_{Sp_0}^+|-|H_{Sp_0}^-| = (-1)^{\delta(2|Sp_0)},\]
as needed.
\end{proof}

\begin{proposition}\label{Sbias}
For $S>1$ square-free, let $\Delta(S^2,k)^{\min}=\dim S_k^{\min}(S^2)^+-\dim S_k^{\min}(S^2)^-$. Then
for $k>2$ even,
\begin{equation}\label{Deltasum}
\Delta(S^2,k)^{\min}=\Delta_M+\Delta_{A_1}+\Delta_{A_2},
\end{equation}
where
\[
\Delta_M = D_4(S)(-1)^{k/2+\delta(2|S)}\frac{k-1}{12}\prod_{p|S}(p-1),\qquad
\Delta_{A_1} = \frac{D_4(S)}4\prod_{p|S}(p-1)
\]
for $D_4(S)$ as in Lemma \ref{globalecount} above, and
\[\Delta_{A_2}= \delta(k\equiv 0,2\mod 6)\frac{D_3(S)}3(-1)^{\delta(k\equiv 6,8\mod 12)}\mu(S)\Omega_0(S')\]
where $D_3(S)\in\{0,1\}$ is $0$ if and only if $p\equiv 1\mod 3$ for some $p|S$, 
$\mu(S)=\prod_{p|S}(-1)$ is the M\"obius function, 
and for $S'=\frac S{\gcd(S,6)}$,
\[\Omega_0(S')=\begin{cases} 0&\text{if there exists $p|S'$ such that }p\equiv 5\mod 12\\
2^{\w(S')}&\text{if }p\equiv 11\mod 12\text{ for all $p|S'$},
\end{cases}
\]
where 
$\w(S')=\sum_{p|S'}1$.  (Note that $\Omega_0(1)=1$.)
\end{proposition}
\noindent{\em Remark:} Proposition \ref{Sbiasa}, which  summarizes the conditions under which 
$\Delta(S^2,k)^{\min}$ vanishes, is positive, or is negative, follows easily.
 The claim in the third paragraph of Proposition \ref{Sbiasa}
is due to the fact that when $D_4(S)=1$,
\[|\Delta_{A_1}+\Delta_{A_2}|\le \frac14\prod_{p|S}(p-1)+\frac13\prod_{p|S'}2
=\left[\tfrac14+\tfrac13\prod_{p|S'}\tfrac{2}{p-1}\prod_{p|\gcd(S,6)}\tfrac1{p-1}\right]\prod_{p|S}(p-1)
\le \frac{7}{12}\prod_{p|S}(p-1)\]
where the last inequality is strict if $S>2$. 
So if $k\ge 10$, or $k=8$ and $S>2$, it follows that $|\Delta_{A_1}+\Delta_{A_2}|<|\Delta_M|$, 
  and hence the sign of $\Delta_M$ is the sign of the bias.  One checks by hand (or LMFDB)
that $S_8^{\min}(2^2)=0$.  
The case $k=6$ follows similarly, replacing the rightmost inequality by $<\frac5{12}\prod_{p|S}(p-1)$ 
when $S>6$ and $D_4(S)=1$, and checking the $S|6$ cases by hand.

\begin{proof}[Proof of Proposition \ref{Sbias}]
We have
\begin{equation}\label{Dminsum}
\Delta(S^2,k)^{\min}=
\sum_{\widehat{\sigma}:\epsilon(k,\widehat{\sigma})=1}
\dim S_k(\widehat{\sigma})
-\sum_{\widehat{\sigma}:\epsilon(k,\widehat{\sigma})=-1}
\dim S_k(\widehat{\sigma}).
\end{equation}
Applying Theorem \ref{d0thm} to each summand, we get a sum of three terms as in \eqref{Deltasum}.
Since the archimedean factor of the global root number is $(-1)^{k/2}$ 
(cf. \cite[Theorem 14.17]{IK} and \cite{C}),  the set of tuples
$\widehat{\sigma}$ with global root number $\epsilon$ is $H_S^{(-1)^{k/2}\epsilon}$, with
notation as in Lemma \ref{B}.  Therefore the contribution of the main term is
\[\Delta_M = \frac{k-1}{12}\prod_{p|S}(p-1)\left(\Bigl|H_S^{(-1)^{k/2}}\Bigr|-\Bigl|H_S^{-(-1)^{k/2}}\Bigr|\right),\]
and using Lemma \ref{globalecount} we obtain the formula given for $\Delta_M$.

Likewise, the contribution of the $A_1$ term of Theorem \ref{d0thm} to \eqref{Dminsum} is
\[\Delta_{A_1} = \Bigl|H_S^{(-1)^{k/2}}\Bigr|\frac{D_4(S)\cdot 1}4\prod_{\text{odd }p|S}2
 - \Bigl|H_S^{-(-1)^{k/2}}\Bigr|\frac{D_4(S)\cdot (-1)}4\prod_{\text{odd }p|S}2
\]
\[= \frac{D_4(S)}4\left(|H_S^+|+|H_S^-|\right)\prod_{\text{odd }p|S}2,\]
and the given formula follows from \eqref{sigtot}.

In the notation of Theorem \ref{d0thm} and Lemma \ref{B}, the contribution of $A_2$ to \eqref{Dminsum} is
\[\Delta_{A_2}
=\frac{D_3(S)b(k)(-1)^{\delta(3|S)}}3\left(\mathcal{B}(S)^{(-1)^{k/2}}-
\mathcal{B}(S)^{-(-1)^{k/2}}\right)\]
\[
=\frac{D_3(S)b(k)(-1)^{\delta(3|S)+k/2}}3\left(\mathcal{B}(S)^+-
\mathcal{B}(S)^{-}\right).\]
By considering possibilities for $\gcd(6,S)$, it is easy to check using Lemma \ref{B} that
\[\mathcal{B}(S)^+-\mathcal{B}(S)^- = (-1)^{\delta(2|S)}\mu(S')\Omega_0(S').\]
The result then follows from $(-1)^{\delta(2|S)+\delta(3|S)}\mu(S')=\mu(S)$ and the fact that
\[(-1)^{k/2}b(k) = \begin{cases} 1&\text{if }k\equiv 0,2\mod 12,\\
-1&\text{if }k\equiv 6,8\mod 12\\
0&\text{if }k\equiv 4\mod 6.\end{cases}\qedhere\]
\end{proof}

By similar arguments, we obtain the dimension of the space of twist-minimal forms of level $S^2$.

\begin{proposition}\label{Smin}
For $S>1$ square-free and $k>2$ even, 
\begin{align*}
\dim S_k^{\min}(S^2) =\frac{k-1}{12}\prod_{\text{odd }p|S}\frac{(p-1)^2}2
&+\frac{D_4(S)}4 (-1)^{\delta(2|S)+k/2}\prod_{\text{odd }p|S}2\\
&+\frac{D_3(S)b(k)}3(-1)^{\delta(3|S)}
\prod_{p|\frac{S}{\gcd(6,S)}}2
\end{align*}
for 
$b(k)=\begin{cases}1&\text{if }6|k\\-1&\text{if }k\equiv 2\mod 6\\0&\text{otherwise.}\end{cases}$
\end{proposition}
\noindent{\em Remarks:} Although we have assumed $k>2$, the
above formula is valid when $k=2$ as well.
 More generally, the dimension of $S_k^{\min}(N,\chi)$ has been computed by Child,
\cite[\S5.1]{Ch}.

\begin{proof}
We have 
\[\dim S_k(S^2)^{\min}= d_M + d_{A_1}+d_{A_2},\]
where
\[d_M= 
 \frac{k-1}{12}\prod_{p|S}(p-1)\left(\Bigl|H_S^{(-1)^{k/2}}\Bigr|+\Bigl|H_S^{-(-1)^{k/2}}\Bigr|\right),\]
\[d_{A_1}= \frac{D_4(S)}4(-1)^{k/2}\left(|H_S^+|-|H_S^-|\right)\prod_{\text{odd }p|S}2,\]
and
\[d_{A_2}= \frac{D_3(S)b(k)(-1)^{\delta(3|S)}}3\left(\mathcal{B}(S)^++ \mathcal{B}(S)^{-}\right).\]
The result follows upon applying \eqref{sigtot} to $d_M$, Lemma \ref{globalecount} to $d_{A_1}$,
and the fact that $\mathcal{B}(S)^++\mathcal{B}(S)^- = 2^{\w(S')}=\prod_{p|S'}2$, for $S'=S/\gcd(6,S)$.
\end{proof}

\subsection{Simplification when $\n=1$ and $T>1$}

We return to the general setting of Theorem \ref{stf2} with no constraint on the 
conductor exponents of the $\sigma_p$.  Our aim here is to cull the list of matrices that appear in
Theorem \ref{stf2} when $\n=1$ and $T>1$.
The result is Proposition \ref{glist}, from which Theorem \ref{dimST} follows.

Recall that for $p|T$, $\sigma_p$ is a supercuspidal representation whose conductor is of 
the form $p^{n}$ with $n\ge 3$ odd.  It is well known (see, e.g., \cite[\S A.3.8]{He}) that there is a
ramified quadratic extension $E/\Q_p$ with $E^*$ embedded in $G(\Q_p)$ such that
$\sigma_p$ is compactly induced from a character $\chi$ of $J_n=E^* U^{(n-1)/2}$, 
where $U^r=1+\smat{p\Z_p}{\Z_p}{p\Z_p}{p\Z_p}^r$ is an open compact subgroup of $G(\Q_p)$ and
$\chi|_{F^*}=\w_p$. 
In the notation of \S\ref{ssr}, $U^1$ coincides with $K'$, $J_3$ with $H'$, and in general
\[J_n\subset H'.\]
We use the local test function defined for $g\in G(\Q_p)$ by
\begin{equation}\label{fpT}
f_p(g)=\begin{cases}d_{\sigma_p}\ol{\chi(g)}&\text{if }g\in J_{n}\\0&\text{otherwise,}\end{cases}
\end{equation}
 where $d_{\sigma_p}$ is the formal degree (depending only on the conductor).
  This coincides with \eqref{fp} when $n=3$.

If $p|T$, the support of $f_p$ is the disjoint union of its unramified and ramified elements:
\begin{equation}\label{suppfp}
\Supp(f_p)=J_n= (J_n\cap ZK') \bigcup (J_n\cap \pi_E ZK'),
\end{equation}
where $\pi_E$ is a prime element of $E$ whose square is a prime element of $\Q_p$.
We may decompose $f_p$ as $f_p=f_u + f_r$, a sum of two functions supported
on the unramified and ramified elements of $J_n$ respectively.
In the paper of Gross discussed in \S\ref{bias}, $n=3$ and the local test function used 
is a multiple of $f_u$, \cite[p. 1240]{G}.
The following is largely contained in \cite[Prop. 5.1]{G}.

\begin{proposition}\label{unram}
Let $f^1=f^\n$ for $\n=1$.
    Suppose $\g$ is elliptic in $G(\Q)$ and unramified at some prime $p|T$.
    Then either $\g$ has $p$-torsion in $\olG(\Q)$ and $p\in\{2,3\}$, or $\Phi(\g,f^1)=0$.    
    As a result, $\Phi(\g,f^1)=0$ in each of the following situations:
    \begin{enumerate}
        \item $\g$ is unramified at some prime $p|T$ with $p>3$;
        \item $\g$ is unramified at $3|T$ and $T\neq3$.
    \end{enumerate}
\end{proposition}

\begin{proof}
Write $f=f^1$. Suppose $\Phi(\g,f)\neq 0$.
      By Proposition \ref{ellg}, $\g$ is elliptic in $G(\R)$ and $\det\g>0$. Hence it belongs to a 
compact-mod-center subgroup $U_\infty$ of $G(\R)$
    ($U_\infty$ being some conjugate of $\R^*\cdot\SO(2)$).
Likewise, at every finite place $v$, the support of $f_v$ is a 
compact-mod-center subgroup $J_v$ of $G_v$ (here is where we use $\n=1$), and $\g$ belongs to some
    conjugate $U_v$ of $J_v$.  (In fact since $\g\in K_v$ a.e., we can take $U_v=K_v$ a.e.)
 Hence $\g$ belongs to a compact-mod-center subgroup $\prod_v U_v$ of $G(\A)$.
 Identifying $\g$ with its image modulo the center, we have
\[\g\in \olG(\Q)\cap \prod_v \ol{U_v}.\]
    This is a {\em finite} group since $\olG(\Q)$ is discrete in $\olG(\A)$ (\cite[\S7.11]{KL}).
In particular, $\g$ is a torsion element of $\olG(\Q)$, i.e.,
some power of $\g$ lies in the center $Z(\Q)$.

 Since $\g$ is unramified at $p|T$, some conjugate of $\g$ belongs to the unramified part of
the support of $f_p$, which is a subset of the pro-$p$ group $\ol{K'}$.
    (Recall that $K'$ is the pro-$p$-Sylow subgroup of the Iwahori subgroup of $G(\Q_p)$).
    It follows that the order of $\g$ in $\olG(\Q)$ is a power of $p$.
    However, it is known that any torsion element of $\olG(\Q)$ has order $1, 2, 3, 4,$ or $6$,
 \cite[Lemma 1]{D}.
    Since $\g\neq 1$, we conclude that $p\le 3$.  This proves (1).
    
    The $3$-torsion elements of $\olG(\Q)$ comprise a single conjugacy
    class containing $\smat 0{-1}11$ (\cite[Lemma 1]{D}). 
    Therefore if $p=3$, $\g$ is conjugate in $G(\Q)$
    to a matrix of the form $\smat 0{-z}zz$ and is hence everywhere unramified.
    By the above, this means $T$ is not divisible by any prime $p>3$.  
  It is also odd, because otherwise $\g$ would somehow simultaneously have $3$-torsion and
  $2$-power torsion.  Hence $T=3$, which proves (2).

    By the same reference, the $4$-torsion elements of $\olG(\Q)$ are all conjugate to 
    $\smat 1{-1}11$. But such an element is ramified at $2$.  Hence if $p=2$,
    $\g$ has $2$-torsion.
\end{proof}

\begin{proposition}\label{glist}
With notation as in \S\ref{spectral}, let $T$ be the product
of the primes $p$ for which $\ord_p(N)$ is odd, and for $p|T$ take $f_p$ as in \eqref{fpT}.
  Then for $\g\in \olG(\Q)$,
 $\Phi(\g,f^1)=0$ unless either $\g=1$ or the conjugacy class of $\g$ has a representative in $G(\Q)$
 of one of the forms given in the table below:
        \[        \begin{array}{l|c}
            \rule[-3mm]{0mm}{8mm}
     \text{Form of $T$}&\text{List of relevant elliptic $\g$ for $\n=1$}\\
     \hline
     \hline
            \rule[-3mm]{0mm}{8mm}
            \text{even }  T\neq 2 &\smat{}{-T}1{},\smat{}{-T/2}1{}\\
     \hline
            \rule[-3mm]{0mm}{8mm}
            T=2&\smat{}{-2}1{},\smat{}{-1}1{},\smat0{-2}12\\
     \hline
            \rule[-3mm]{0mm}{8mm}
            \text{odd }  T> 3 &\smat{}{-T}1{}\\
     \hline
            \rule[-3mm]{0mm}{8mm}
            T=3&\smat{}{-3}1{},\smat0{-3}13,\smat{0}{-1}1{1}\\
     \hline
            \rule[-3mm]{0mm}{8mm}
            T=1&\smat{}{-1}1{},\smat{0}{-1}1{1}.
 \end{array}
    \]
\end{proposition}
\noindent{\em Remark:} When $T/2\equiv 7\mod 8$, the matrix $\smat{}{-T/2}1{}$ is hyperbolic
(rather than elliptic) in $G(\Q_2)$, so its orbital integral vanishes.
All other entries in the above table are elliptic in $G(\Q_p)$ for each $p|T$, but for $p|S$
 this needs to be checked on a case-by-case basis.

\begin{proof}
The case where $T=1$ is already contained in Theorem \ref{stf2}, taking $\n=1$.
So suppose $T>1$ and $\Phi(\g,f)\neq 0$.
By Proposition \ref{rel1}, we may take $\g=\mat{0}{- M}1{rM}$ for some $M|T$ and
$0\le r< \sqrt{4/M}$.  Notice that if $M>3$ then $r=0$.  Suppose first that $T\neq 3$.  
By Proposition \ref{unram}, $\g$ must be ramified at all odd
primes dividing $T$, so $M=T$ or $M=T/2$.  If $T$ is odd, this means $M=T$ and we obtain the third
row of the above table.
Suppose $T$ is even and $M=T/2$. By Proposition \ref{unram}, $\g$ has $2$-torsion in $\olG(\Q)$.
Note that
\[\g^2=\mat{- M}{-r M^2}{rM}{r^2M^2- M}\]
is a scalar matrix if and only if $r=0$.  Therefore
$\g=\smat0{- M}10$.  This establishes the top two rows of the table.
(When $M=T=2$, $r=1$ is admissible, and for $\g = \smat0{-2}12$, 
$P_\g(X)= X^2-2X+2$ is
an Eisenstein polynomial for the prime $2$, which is indeed irreducible in $\Q_2[X]$,
\cite[p. 19]{S}.)

Now suppose $T=3$. Then $M=1$ or $M=3$.  In the latter case,
$\g=\mat0{-3}1{3r}$ for $r=0,1$.
If $M=1$, then $\g=\smat0{-1}1r$ for $r=0,1$, and $\g$ is unramified at $3$. 
  If $r=0$, this matrix has 2-torsion, in violation of Proposition \ref{unram}.
Hence $\g=\smat0{-1}1{1}$. 
(In this case, $P_\g(X)=X^2-X+1$ has discriminant $-3$, which is
is not a square in $\Q_3$, and hence $\g$ is indeed elliptic in $G(\Q_3)$.)
\end{proof}

\subsection{Global orbital integrals for $\n=1$, $N=T^3$.}\label{global}

Here we will evaluate the global elliptic orbital integrals of Theorem \ref{dimST} 
explicitly when $N=T^3>1$ for $T$ square-free.
 We must consider
  \[\g=\smat{}{-T}1{},\, \smat{}{-T/2}1{}_{(T\text{ even})},\, 
  \smat0{-2}1{2}_{(T=2)}
  ,\, 
  \smat0{-1}1{1}_{(T=3)},\,\smat0{-3}1{3}_{(T=3)}
  \]
  as appearing in Proposition \ref{glist}.

We introduce some notation before stating the global results.
    Given our tuple $\widehat{\sigma}=(\sigma_{t_p}^{\zeta_p})_{p|T}$ of simple
    supercuspidal representations, for $k>2$ define
    \begin{equation}\label{epsilon}
    \epsilon(k,\widehat{\sigma})=i^k\prod_{p|N}\zeta_p.
\end{equation}
This is the common global root number of the cusp forms comprising $H_k(\widehat{\sigma})$
(see Proposition \ref{ssc}, \cite[Theorem 14.17]{IK}, and \cite{C}).
Throughout this section $f=f^1$ as in \eqref{fnt}.

\begin{proposition}\label{e1}
    For $N=T^3$, with notation as above, suppose that for each odd prime factor $p$ of the square-free 
    integer $T>1$,
    $-pt_p/T$ is a square modulo $p$.  Then for $k\ge 4$ even,
 \[\Phi(\mat{}{-T}1{},f)=
    \frac {\ol{\epsilon(k,\widehat{\sigma})}\,2_7\,4_3\, h(-T)} {3_{T=3}\,2^{\w(T)}}\sum_{y}\w'(y),\]
    where numbers with subscripts are present only when $T$ falls into the subscript's equivalence
class modulo 8,
$3_{T=3}$ is a factor of 3 which is
    present only when $T=3$, 
and $y$ ranges over all integers modulo $T$ that satisfy $y^2\equiv -pt_p/T\mod p$ for all $p|T$.
If the central character is trivial, the above simplifies to
\begin{equation}\label{e1w1}
\Phi(\mat{}{-T}1{},f)=    \frac {{\epsilon(k,\widehat{\sigma})} h(-T) w_T} {3_{T=3}},
\end{equation}
where
    \[w_T=\begin{cases}1/2&\text{if $T$ is even}\\
        1&\text{if }T\equiv 1\mod 4\\
        2&\text{if }T\equiv 7\mod 8\\
    4&\text{if }T\equiv 3\mod 8.\end{cases}\]
\end{proposition}

\noindent{\em Remark:} If the first hypothesis is not satisfied or $k$ is odd,
   then $\Phi(\g,f)=0$; see Proposition \ref{relevant}.

\begin{proof}
 Take $\g=\smat{}{-T}1{}$, $\Delta_\g=-4T$, and let $M$ be the odd part of $T$, so that $T=2^aM$ for some $a\in \{0,1\}$.
 Corresponding to \eqref{phifact}, write
\[\Phi(\g,f)=m\Phi_\infty\Phi_2\prod_{p|M}\Phi_p
    =m (-1)^{k/2}\Phi_2
  \prod_{p|M}\ol{\zeta_p}\sum_{y_p}\w_p(y_p),\]
where we have applied \eqref{Phiinf} and Proposition \ref{Phiell},
%(\w_p(y_p)+\w_p(-y_p)),\]
    with $y_p$ running over the two (since $p$ is odd) solutions to 
 $y_p^2\equiv -pt_p/T\mod p$.
We can exchange the sum and product.
To each of the $2^{\w(M)}$ tuples $(y_p)_{p|M}$, 
the Chinese remainder theorem assigns a unique integer $y$ modulo $T$
  satisfying $y\equiv y_p\mod p$  for all $p|T$, where we take $y_2=1$ if $T$ is even.  Further,
\[\w'(y)=\prod_{p|T}\w_p(y)=\prod_{p|T}\w_p(y_p)=\prod_{p|M}\w_p(y_p).\]
The first equality holds because $\gcd(y,T)=1$ (see \cite[(12.4)]{KL}); the second holds
since each $\w_p$ is trivial on $1+p\Z_p$.
 By Example \ref{Nell} (for $T$ odd) or Proposition \ref{Phiell} (for $T$ even), 
    \[\Phi_2=\begin{cases}\ol{\zeta_2}&\text{if $T$ is even}\\
    2&\text{if }T\equiv 1,5,7\mod 8\\4&\text{if }T\equiv 3\mod 8.\end{cases}\]
    It follows that
    \[\Phi(\g,f)= \frac{2h(E)}{w_E 2^{\w(d_E)}}
 \ol{\epsilon(k,\widehat{\sigma})} a_T \sum_y \w'(y),\]
for $y$ as in the statement of the proposition, 
    \[a_T=\begin{cases}1&\text{if $T$ is even,}\\2&\text{if }T\equiv 1,5,7\mod 8,\\
    4&\text{if }T\equiv 3\mod 8,\end{cases}\]
and
$E=\Q(\sqrt{-T})$.
 Since $T>1$, we know that 
 \[w_E=|\O_E^*|=\begin{cases} 6&\text{if }T=3\\ 2&\text{otherwise.}\end{cases}\]
So $w_E/2=3_{T=3}$, and $\frac{2h(E)}{w_E}=\frac{h(-T)}{3_{T=3}}$. 
Recall that
\[d_E=\begin{cases}-4T,&-T\equiv 2,3\mod 4\\-T,&-T\equiv 1\mod 4.\end{cases}\]
    Therefore, placing the congruence condition on $T$ rather than $-T$,
    \[2^{\w(d_E)}=\begin{cases}2\cdot 2^{\w(T)}&\text{if }T\equiv 1\mod 4\\
     2^{\w(T)}&\text{if }T\equiv 2,3\mod 4.
    \end{cases}\]
Hence using the definition of $a_T$ in the following numerator,
    \[\Phi(\g,f)=\ol{\epsilon(k,\widehat{\sigma})}h(-T)
\frac{2_{1,5,7}\cdot 4_3}{3_{T=3}\cdot 2_{1,5}\cdot 2^{\w(T)}}\sum_y\w'(y),\]
    where numbers with subscripts are only present when $T$ falls into one of the
    subscript equivalence classes modulo $8$.  The general result now follows.

If $\w'$ is trivial, the sum over $y$ equals the number of terms, namely $2^{\w(M)}$.   
Equation \eqref{e1w1} then follows from
    \[\frac{2^{\w(M)}}{2^{\w(T)}}=\begin{cases}1&\text{if $T$ is odd}\\1/2&\text{if $T$ is even}
    \end{cases}\]
and the fact that $\epsilon(k,\widehat{\sigma})\in \{\pm 1\}$ is real in this case.
\end{proof}

\begin{proposition}\label{e2}
    For $N=T^3$, suppose that the square-free integer $T=2M$ is even, and that for each prime factor $p$ of $T$,
    $-pt_p/M$ is a square modulo $p$.  Then for even $k\ge 4$,
 \[\Phi(\mat{}{-M}1{},f)=
h(-M)\frac{\ol{\epsilon(k,\widehat{\sigma})}}{\zeta_2}
 \cdot\frac{z_M}{2_{M=1}\, 3_{M=3}\, 2^{\w(M)}}\sum_y\w'(y),\]
 where $2_{M=1}$ is a factor of 2 which is present only when $M=1$, $3_{M=3}$ is defined
 similarly, 
 \[z_M=\begin{cases}\frac12&\text{if }M\equiv 1\mod 4\\
  -3&\text{if }M\equiv 3\mod 8\\
 0&\text{if }M\equiv 7\mod 8,\end{cases}\]
and $y$ ranges over all elements modulo  $M$ that satisfy $y^2\equiv -pt_p/M\mod p$ for each $p|M$.
If $\w'$ is trivial, the sum over $y$ simply cancels with the factor of $2^{\w(M)}$. 
 (Again, if the condition on the $t_p$ fails to hold or $k$ is odd,
 the orbital integral vanishes.)
\end{proposition}

\begin{proof}
We use the same proof as for the previous proposition, with minor modifications.  First,
 by Example \ref{Phi2a},
 \[\Phi(\smat{}{-M}1{},f_2)=\begin{cases}
1&\text{if }M\equiv 1\mod 4
  \\-3 &\text{if }M\equiv 3\mod 8\\
0&\text{if }M\equiv 7\mod 8.
\end{cases}\]
  Taking $E=\Q[\sqrt{-M}]$ we have
  \[2^{\w(d_E)}=\begin{cases} 2\cdot 2^{\w(M)}&\text{if }M\equiv 1\mod 4\\
  2^{\w(M)}&\text{if }M\equiv 3\mod 4\end{cases}\]
  as in the previous proof, and $\frac{2h(E)}{w_E}=\frac{h(-M)}{3_{M=3}\,2_{M=1}}$ since $\Q[\sqrt{-1}]$ has
 unit group of order 4 when $M=1$.
  Hence (assuming $M\not\equiv 7\mod 8$)
 \[\Phi(\smat{}{-M}1{},f)= \frac{h(-M)\,(-3)_3}{3_{M=3}\,2_{M=1}\,2_{1,5}\,2^{\w(M)}}
\frac{\ol{\epsilon(k,\widehat{\sigma})}}{\zeta_2}\sum_y\w'(y)
    \]
 where numerical subscripts refer to the congruence class of $M$ modulo 8.
\end{proof}

\begin{proposition}\label{2T}  Suppose $N=2^3$, $\zeta\in \{\pm 1\}$ and $\sigma=\sigma^\zeta$ is our 
  fixed simple supercuspidal
representation of $G(\Q_2)$ (the parameter $t$ must equal $1$ when $p=2$).  Then
 \[\Phi(\mat0{-2}12,f)=\frac{\epsilon(k,{\sigma})}4g_8(k),\]
 where 
 $g_8(k)=-1$ if $k\equiv 0,2\mod 8$, and 
 $g_8(k)=1$ if $k\equiv 4,6\mod 8$.
\end{proposition}
\noindent{\em Remark:} In view of Proposition \ref{k2N2}, we assume that $k$ is even.

\begin{proof}
 Given that $\g$ has characteristic polynomial $X^2-2X+2$ with discriminant $\Delta_\g=-4$,
   we find $E=\Q[\g]=\Q[i]$.
Hence $h(E)=1$, $w_E=|\O_E^*|=4$, and $d_E=-4$.
 By \eqref{phifact},
 \[\Phi(\g,f)=m\Phi_\infty\Phi_2=
 \frac14\Phi_\infty\Phi_2.\]
Applying Proposition \ref{Phiell} with $p=2$ and $v=1$, $\Phi_2=-\zeta$.  So
 \begin{equation}\label{Pi2}
 \Phi(\g,f)=-\frac{\Phi_\infty\,\zeta}4.
 \end{equation}
 The complex eigenvalues of $\g$ are $1\pm i$, so we apply \eqref{Rell} with $\theta=\pi/4$ to get
 \[ \Phi_\infty=
 -\sqrt2\,\sin\Bigl(\frac{(k-1)\pi}4\Bigr)
 =\begin{cases}
 1&\text{if }k\equiv 0,6\mod 8\\
  -1&\text{if }k\equiv 2,4\mod 8.\end{cases}\]
 Multiplying this by $-1$ as in \eqref{Pi2} yields $(-1)^{k/2}g_8(k)$ with $g_8$ as given.
\end{proof}

\begin{proposition}\label{G3a} Suppose $T=3$ so $N=3^3$, and let $\sigma=\sigma_t^\zeta$ be our
 fixed simple supercuspidal representation of $G(\Q_3)$, for $t=\pm 1$.  Then
 \[\Phi(\mat0{-1}11,f)= \frac{t}{2_{t=-1}}c_3(k),\]
 where
 $c_3(k) =\frac13+\lfloor\frac{k}3\rfloor-\frac k3$.
\end{proposition}

\begin{proof}
 Let $E=\Q[\g]=\Q[\sqrt{-3}]$.  Then $h(E)=1$, $d_E=-3$, and $w_E=|\O_E^*|=6$.
 By \eqref{phifact} and taking $m=1$ in Example \ref{g3b} and its remark,
 \[\Phi(\g,f)=\Phi_\infty \cdot \frac{(-1)^k t\cdot 2_{t=1}}{6}
  =\frac{(-1)^k\Phi_\infty}3\frac t{2_{t=-1}}.\]
By \eqref{Rell}, we find that
 \begin{equation}\label{g3k}
(-1)^k\Phi(\g,f_\infty)=(-1)^{k+1}\frac{\sin(\frac{(k-1)\pi}3)}{\sin(\pi/3)}
 =\begin{cases} 1&\text{if }k\equiv 0\mod 3\\
  0&\text{if }k\equiv 1\mod 3\\
 -1&\text{if }k\equiv 2\mod 3.
\end{cases}
 \end{equation}
Using the above, we see that $\frac{(-1)^k\Phi(\g,f_\infty)}3 = 
 \frac13+\lfloor\frac{k}3\rfloor-\frac k3$.
\end{proof}

\begin{proposition}\label{G3b} Suppose $N=3^3$, and let $\sigma=\sigma_t^\zeta$ be a
 fixed simple supercuspidal representation of $G(\Q_3)$.  Then 
 \[\Phi(\mat0{-3}13,f)= \begin{cases}0&\text{if }t=1\\
{\epsilon(k,\sigma)}g_{6}(k)&\text{if }t=-1,\end{cases}\]
 where
 \[g_{6}(k)=\begin{cases}0&\text{if }k\equiv 1\mod 6\\
  -1/6&\text{if }k\equiv 0,2\mod 6\\
1/2&\text{if }k\equiv 3\mod 6\\
 1/3 &\text{if }k\equiv 4\mod 6\\
 -1/2 &\text{if }k\equiv 5\mod 6.
\end{cases}\]
\end{proposition}

\begin{proof}
Let $\g=\smat0{-3}13$, so $\Delta_\g=-3$.  
  We have $E=\Q[\g]=\Q[\sqrt{-3}]$, so the measure factor is $1/6$ as in the previous prooof. 
Therefore as in \eqref{phifact}, we may write
\begin{equation}\label{Pg3}
\Phi(\g,f)=\frac 16\Phi_\infty\Phi_3.
\end{equation}
By Proposition \ref{Phiell},  $\Phi_3=0$ unless $-t$ is a square modulo $3$, i.e.,
unless $t=-1$. Assuming this holds, we have
\begin{align*}
\Phi_3&=\ol{\zeta}\cdot\Bigl(\ol{\psi(1)}\w_3(1)+\ol{\psi(-1)}\w_3(-1)\Bigr)\\
&=\ol{\zeta}\cdot(e^{-2\pi i/3}+(-1)^ke^{2\pi i/3})=
-\ol{\zeta}\,[i\sqrt 3]_{k\text{ odd}},\end{align*}
where the factor of $i\sqrt 3$ is present only when $k$ is odd.

By \eqref{wp} with $N=3$, $\w_3(3)=1$.  So $\zeta^2=\w_3(t)=\w_3(-1)=(-1)^k$, so $\zeta = \pm (i^k)$. 
In particular, the global root number
$\e(\sigma,\zeta)=i^k \zeta$
is real and $\ol{\zeta}=(-1)^k\zeta$.

The complex roots of $P_\g(X)=X^2-3X+3$ are $\frac{3\pm i\sqrt{3}}2=\sqrt3(\frac{\sqrt 3\pm i}2)$, so in \eqref{Rell} we can take
$\theta=\pi/6$ and $\Phi(\g,f_\infty)=-2\sin\Bigl(\frac{(k-1)\pi}{6}\Bigr)$.
Hence \eqref{Pg3} becomes
\[\Phi(\g,f)=\frac{{(-1)^k\zeta}}3\sin\Bigl(\frac{(k-1)\pi}6\Bigr)[i\sqrt 3]_{k\text{ odd}}
=\begin{cases}
\zeta/3&\text{if }k\equiv 4\mod 12\\
-i\zeta/2&\text{if }k\equiv 3,5\mod 12\\
\zeta/6&\text{if }k\equiv 2,6\mod 12\\
0&\text{if }k\equiv 1,7\mod 12\\
-\zeta/6&\text{if }k\equiv 0,8\mod 12\\
i\zeta/2&\text{if }k\equiv 9,11\mod 12\\
-\zeta/3&\text{if }k\equiv 10\mod 12.\end{cases}\]
Upon factoring out $\epsilon(k,\sigma)=i^k\zeta$, we obtain $g_6(k)$ as given.
\end{proof}

\subsection{Dimension formulas when $N=T^3$}\label{dim}

Here we put everything together to compute $|H_k(\widehat{\sigma})|=\dim S_k(\widehat{\sigma})$ for
   $\widehat{\sigma}=(\sigma_p)_{p|N}$ a tuple of 
simple supercuspidal representations of $G(\Q_p)$ as in Theorem \ref{dimST} with $S=1$.

We begin with the case $N=2^3$, where the central character is necessarily trivial
 due to \eqref{wpp} and Proposition \ref{k2N2}.

\begin{theorem}\label{N2dim} Let $N=2^3$, fix $\zeta\in \{\pm1\}$, and let $\sigma=\sigma_\zeta$ be the
 associated simple supercuspidal representation of $G(\Q_2)$ with trivial central character.  Then
 \[|H_k(\sigma)|=\begin{cases}0&\text{if $k$ is odd}\\
\lfloor\frac k8\rfloor &\text{if }k\equiv 0,2\mod 8\\
 \lfloor\frac k8\rfloor +\frac{1+\epsilon(k,\sigma)}2 
  &\text{if }k\equiv 4,6\mod 8,\end{cases}\]
where $\epsilon(k,\sigma)=(-1)^{k/2}\zeta$ is the global root number.
\end{theorem}
\begin{proof}
 When $k$ is odd, the assertion follows from Proposition \ref{k2N2}.
Suppose $k$ is even.  By Theorem \ref{dimST},
 \[|H_k(\sigma)|=\frac{k-1}{12}\cdot \frac32 +\frac12\Phi(\smat{}{-2}1{},f)
  +\frac12\Phi(\smat{}{-1}1{},f)
 +\Phi(\smat0{-2}12,f).\]
 Applying the results of \S\ref{global} using $h(-2)=h(-1)=1$, we find
 \[|H_k(\sigma)|=\frac{k-1}{8} +\frac{(-1)^{k/2}\zeta}4+\frac{(-1)^{k/2}}{8}
 +\frac{(-1)^{k/2}\zeta}4g_8(k)\]
for
$g_8(k)=\begin{cases} -1&\text{if }k\equiv 0,2\mod 8
\\ 1&\text{if }k\equiv 4,6\mod 8.\end{cases}$
The result follows upon simplifying each of the cases.
\end{proof}

\begin{theorem}\label{N3dim} Let $N=3^3$, fix $t\in \{\pm 1\}$, a character $\w_3$ of $\Q_3^*$
trivial on $1+3\Z_3$, $\zeta\in \C$ with $\zeta^2=\w_3(t)$ (cf. \eqref{wpp}),
  and let $\sigma=\sigma_t^\zeta$ be the associated simple 
  supercuspidal representation of $G(\Q_3)$ with central character $\w_3$.
Then for $k>2$, setting $\epsilon = i^k\zeta$, we have
 \[|H_k(\sigma)|=\begin{cases}
\lfloor\tfrac k3\rfloor +\frac{\epsilon-1}2&\text{if }k\equiv 0 \mod 3\text{ and }t=-1\\
\lfloor\tfrac k3\rfloor &\text{if }k\equiv 1 \mod 6\text{ or }t=1\\
\lfloor\tfrac k3\rfloor +\frac{\epsilon+1}2&\text{if }k\equiv 2 \mod 6\text{ and }t=-1\\
\lfloor\tfrac k3\rfloor+\epsilon &\text{if }k\equiv 4 \mod 6\text{ and }t=-1\\
\lfloor\tfrac k3\rfloor +\frac{1-\epsilon}2&\text{if }k\equiv 5 \mod 6\text{ and }t=-1.
\end{cases}\]
\end{theorem}
\noindent{\em Remarks:} (1)  If $t=-1$, then $\zeta^2=\w_3(-1)=(-1)^k$, so $\zeta=\pm i^k$, as
noted earlier.  Therefore $\epsilon \in\{\pm 1\}$ when $t=-1$.  When $t=1$ and $k$ is odd,
$\epsilon =\pm i$. 

(2) There is one more newform with $\epsilon=-1$ than with $\epsilon=1,i,$ or $-i$ when
$k\equiv 5\mod 6$, i.e., the root number has a slight bias toward $-1$ in this case.
For example, when $k=5$ and $\w'$ is the Dirichlet character
of conductor 3, there are five newforms of 
  level 27, with respective root numbers $1, -1, -1, i, -i$.
These newforms are discussed further in \S\ref{ng1}.

\begin{proof}
By Theorem \ref{dimST},
 \begin{align*}|H_k(\sigma)|&=\frac{k-1}{12}\cdot \frac82 +\frac12\Phi(\smat{}{-3}1{},f)
 +\Phi(\smat0{-1}11,f)+\Phi(\smat0{-3}13,f)\\
 &=\frac{k-1}{3} +
\frac{2\epsilon}3\delta_{t=-1}\cdot\delta_{k\in2\Z}
+\frac{t}{2_{t=-1}}c_3(k)
+\epsilon {g_{6}(k)}\delta_{t=-1},
\end{align*}
where we have applied Propositions \ref{e1}, \ref{G3a}, and \ref{G3b}, and $c_3(k), g_6(k)$ are
recalled below.  (For nonvanishing of $\Phi(\smat{}{-3}1{},f)$, the hypothesis in 
  Proposition \ref{e1} requires that $-t$ be a square modulo $3$, i.e., $t=-1$,
 and $k$ even.  Then $\ol{\epsilon}=\epsilon$ and the sum over $y$ in that result is $1+(-1)^k=2$.)

If $t=1$, then because $c_3(k)=\frac{1-k}3+\lfloor\frac k3\rfloor$, the above simplifies to 
  $\lfloor\frac k3\rfloor$, as needed.

Now suppose $t=-1$, and write $k=a+6c$ for some $0\le a\le 5$.  If $k$ is odd, then
\[|H_k(\sigma)|=\tfrac{k-1}3 -\tfrac12(\tfrac{1-k}3+\lfloor\tfrac k3\rfloor)+\epsilon g_6(k)
=\tfrac{k-1}2-\tfrac12\left\lfloor\tfrac k3\right\rfloor+\epsilon g_6(k).\] 
Using the fact that $g_6(k)=0,\tfrac12, -\tfrac12$ when $a=1,3,5$ respectively,
we get
\[|H_k(\sigma)|=\begin{cases}2c=\lfloor\tfrac k3\rfloor&\text{if }a=1\\
2c+1+\frac{\epsilon -1}2=\lfloor\tfrac k3\rfloor+\frac{\epsilon-1}2
&\text{if }a=3\\
2c+1+\frac{1-\epsilon}2=\lfloor \frac k3\rfloor+\frac{1-\epsilon}2&\text{if }a=5.
\end{cases}\]
If $k$ is even, then there is one extra term, namely $\frac{2\epsilon}3$, so
\[|H_k(\sigma)|=\tfrac{k-1}2 -\tfrac12\lfloor\tfrac k3\rfloor+\epsilon(\tfrac23+ g_6(k)).
\]
Here, $g_6(k)=-\frac16,-\frac16,\frac13$ when $a=0,2,4$ respectively.  Upon simplifying,
\[|H_k(\sigma)|=\begin{cases}
2c+\frac{\epsilon -1}2=\lfloor\tfrac k3\rfloor+\frac{\epsilon-1}2&\text{if }a=0\\
2c+\frac{1+\epsilon }2=\lfloor\tfrac k3\rfloor+\frac{1+\epsilon}2&\text{if }a=2\\
2c+1+\epsilon=\lfloor\tfrac k3\rfloor+\e&\text{if }a=4.
\end{cases}\qedhere\]
\end{proof}

\begin{theorem}\label{main}
    Suppose $N=T^3$ with $T>3$ square-free, $M=T/2$, $k\ge 4$ is even, and 
    $\widehat{\sigma}=(\sigma_{t_p}^{\zeta_p})_{p|N}$ is a tuple of simple supercuspidal
    representations with trivial central characters.
Then 
    \begin{equation}\label{mainf}
        |H_k(\widehat{\sigma})|=\frac{k-1}{12}\prod_{p|T}\frac{p^2-1}{2}
    +\Delta_1(\widehat{t})\epsilon(k,\widehat{\sigma}) b_Th(-T)
  +\Delta_2(\widehat{t})
\frac{\epsilon(k,\widehat{\sigma})j_Mh(-M)}{\zeta_2\,3_{M=3}},
    \end{equation}
    where $\epsilon(k,\widehat{\sigma})\in\{\pm 1\}$ is the 
    common global root number of the newforms
    in $H_k(\widehat{\sigma})$ given in \eqref{epsilon}, 
    \[b_T=\begin{cases}1/4&\text{if $T$ is even}\\
        1/2&\text{if }T\equiv 1\mod 4\\
        1&\text{if }T\equiv 7\mod 8\\
    2&\text{if }T\equiv 3\mod 8,\end{cases}\]
 \[j_M=\begin{cases}1/4&\text{if }M\equiv 1\mod 4\\
  -3/2&\text{if }M\equiv 3\mod 8\\
 0&\text{if }M\equiv 7\mod 8,\end{cases}\]
$h(d)$ is the class number of $\Q[\sqrt{-d}]$, 
and $\Delta_i(\widehat{t})\in\{0,1\}$ is nonzero if and only 
  if (1) $T$ is even in the case $i=2$, and
   (2)
 $-2^{i-1}pt_p/T$ is a square modulo $p$ for each odd $p|T$.
\end{theorem}
\noindent{\em Remarks:} 
To keep the formula simple, we have restricted 
ourselves to the case of trivial central character; the general case is obtained similarly.
Even in the general case, one may restrict to $k$ even because by Corollary \ref{kodd}, 
\begin{equation}\label{Hkodd}
|H_k(\widehat{\sigma})|=\frac{k-1}{12}\prod_{p|N}\frac{p^2-1}2\qquad(T>3,\, k\text{ odd}).
\end{equation}

\begin{proof}
 This follows from Theorem \ref{dimST} and Propositions \ref{e1} and \ref{e2}.
\end{proof}

As a corollary, we recover the following dimension formulas of \cite{GMar}.
\begin{corollary}
For $T=2,3$ and $k\ge 4$ even,
\[\dim S_k^{\new}(8)=\lfloor\tfrac k4\rfloor,
\qquad \dim S_k^{\new}(27)=k-1+\lfloor\tfrac k3\rfloor.\]
    For $T>3$ square-free, and $k\ge 4$ even,
    \begin{equation}\label{dimnew}
        \dim S_k^{\new}(T^3)=\frac{k-1}{12}\prod_{p|T}(p-1)^2(p+1).
    \end{equation}
\end{corollary}
\noindent {\em Remarks:} As shown in \cite{GMar}, the formula is also valid for 
$k= 2$.  When $k$ is odd and $\w'$ has conductor dividing $T$, $\dim S_k^{\new}(T^3,\w')$
  is also equal to \eqref{dimnew}. This follows from \eqref{Hkodd}.
\begin{proof}
For $T=2$, by Theorem \ref{N2dim},
\[|H_k(2^3)|=|H_k(\sigma^+)|+|H_k(\sigma^-)|=\begin{cases}2\lfloor\frac k8\rfloor&
\text{if }k\equiv 0,2\mod 8\\
2\lfloor\frac k8\rfloor+1&\text{if }k\equiv 4,6\mod 8.
\end{cases}\]
This is easily seen to be the same as $\lfloor k/4\rfloor$. 

For $T=3$, for fixed $k$ we add the formula in Theorem \ref{N3dim} over all 
  $t,\zeta\in\{\pm 1\}$, or equivalently, $t,\epsilon\in \{\pm 1\}$.  Writing the 
$t=1$ contribution first, we obtain
\[|H_k(3^3)|=2\lfloor\tfrac k3\rfloor +\begin{cases}
2\lfloor\tfrac k3\rfloor -1&\text{if }k\equiv 0\mod 3\\
2\lfloor\tfrac k3\rfloor &\text{if }k\equiv 1\mod 3\\
2\lfloor\tfrac k3\rfloor +1&\text{if }k\equiv 2\mod 3.
\end{cases}\]
The above is easily seen to equal $k-1+\lfloor \tfrac k3\rfloor$, as required.

For $T>3$ we have
    \[\dim S_k^{\new}(T^3)=|H_k(T^3)|=\sum_{\widehat{\sigma}}|H_k(\widehat{\sigma})|,\]
    where $\widehat{\sigma}$ ranges over the $\prod_{p|T}2(p-1)$ tuples $(t_p,\zeta_p)$, with
trivial central character.
    By \eqref{mainf}, this is
    \[=\frac{k-1}{12}\prod_{p|T}\frac{p^2-1}22(p-1) + \sum_{\widehat{\sigma}}
    \Delta_1(\widehat{t})\epsilon(k,\widehat{\sigma})b_T h(-T)
+\sum_{\widehat{\sigma}}\Delta_2(\widehat{t})\frac{\epsilon(k,\widehat{\sigma})j_Mh(-M)}
{\zeta_23_{M=3}}.\]
    It is clear from \eqref{epsilon}
    that exactly half of the $\widehat{\sigma}$ satisfying $\Delta_1(\widehat{t})=1$
  have $\epsilon(k,\widehat{\sigma})=+1$,
    and half have $\epsilon(k,\widehat{\sigma})=-1$.
    So the first sum over $\widehat{\sigma}$ vanishes.
  Likewise if $T$ is even, $\frac{\epsilon(k,\widehat{\sigma})}{\zeta_2}=+1$ (resp. $-1$)
  exactly half of the time since $T$ is divisible by at least one prime different from $2$,
 so the second sum also vanishes.
\end{proof}

Next, we compute the dimension of the subspace of forms with a given root number, which recovers the main
result \eqref{pq} of \cite{PQ}.

\begin{corollary}[\cite{PQ}]
    For $T>3$ square-free and $k\ge 4$ even, the subspace of $S^{\new}_k(T^3)$ with root number $\pm 1$ 
  has dimension
    \[|H^\pm_k(T^3)|=\frac{k-1}{24}\prod_{p|T}(p-1)^2(p+1)\pm \frac{c_Th(-T)}2\prod_{p|T}(p-1),\]
    where $c_T=b_T$ if $T$ is odd, and $c_T=2b_T$ if $T$ is even, i.e.,
\begin{equation}\label{cN}
c_T=\begin{cases}1/2&\text{if }T\equiv 1,2\mod 4\\1&\text{if }T\equiv 7\mod 8\\
    2&\text{if }T\equiv 3\mod 8.\end{cases}
\end{equation}
\end{corollary}

\begin{proof}
    Given $\widehat{\sigma}=(\sigma_{t_p}^{\zeta_p})_{p|T},$ let $\widehat{t}=(t_p)_{p|T}$
    and $\widehat{\zeta}=(\zeta_p)_{p|T}$.  The root number is determined by $\widehat{\zeta}$
    and $k$.
    Let $A_k^\pm$ be the set of all tuples $\widehat{\zeta}$ for which 
$(-1)^{k/2}\prod_{p|T}\zeta_p =\pm 1$.  Then
    \begin{equation}\label{A+}
        |A_k^+|=|A_k^-|=\frac12\prod_{p|T}2.
    \end{equation}
    By \eqref{mainf}, we see that
    \begin{align*}
|H^\pm_k(T^3)|=\sum_{\widehat{\zeta}\in A_k^\pm}\sum_{\widehat{t}}|H_k(\widehat{\sigma})|
    =\sum_{\widehat{\zeta}\in A_k^\pm}\sum_{\widehat{t}}&\left(\frac{k-1}{12}\prod_{p|T}\frac{p^2-1}2
    \pm b_T h(-T)\Delta_1(\widehat{t})\right.\\
&\left.\pm \zeta_2 \frac{j_Mh(-M)}{3_{M=3}}\Delta_2(\widehat{t})\right),
\end{align*}
where $M$ is the odd part of $T$.
Recal that $\Delta_2(\widehat{t})=0$ if $T$ odd. If $T$ is even, upon summing over $\zeta_2=\pm1$ the last
term will be eliminated, so we can ignore it henceforth.
    For any given odd prime $p$, exactly half of
    the elements $t_p\in(\Z/p\Z)^*$ have the property that $-pt_p/T$ is a square.
    Therefore, the number of tuples $\widehat{t}$ for which $\Delta_1(\widehat{t})\neq 0$
    is $\prod_{p|M}\frac{p-1}2$.
The total number of tuples $\widehat{t}$ is
    $\prod_{p|T}(p-1)=\prod_{p|M}(p-1)$.  It follows that
    \[|H^\pm_k(T^3)|
    =\sum_{\widehat{\zeta}\in A_k^\pm}\left(\frac{k-1}{12}\prod_{p|T}\frac{p^2-1}2(p-1)
    \pm b_T h(-T)\prod_{p|M}\frac{p-1}2\right).\]
    By \eqref{A+}, we obtain
    \[|H^\pm_k(T^3)|
    =\frac{k-1}{24}\prod_{p|T}(p-1)^2(p+1)
    \pm \frac{2_Tb_T h(-T)}2\prod_{p|T}(p-1),\]
    where $2_T$ is a factor of $2$ which is only present when $T$ is even.
We see immediately that $2_Tb_T=c_T$ as given.
\end{proof}

\subsection{Some examples with $\n>1$}\label{ng1}

In this section we illustrate Theorem \ref{mainST} with some examples.  (A different set of
examples is given in the earliest version of this paper posted on the arxiv.)
We will compare with the Galois orbits of newforms tabulated in the \cite{LMFDB}.  Though  
$S_k(\widehat{\sigma})$ occasionally forms a Galois orbit, typically the orbit is a
direct sum of more than one such space.  It also happens that a
space $S_k(\widehat{\sigma})$ decomposes as a direct sum of more than one Galois orbit.
Examples of these phenomena can be found in $S_4^{\min}(23^2)$, where Theorem \ref{d0thm}
gives $\dim S_4(\widehat{\sigma})=\frac{11+\epsilon}2\in\{5,6\}$, but the twist-minimal
Galois orbits can have dimensions $1,2,5,6,12$ or $24$. 
% KEEP:  Using LMFDB and Hecke operator traces, we find that
% b+c+f = nu_6, d+e+j = nu_2+nu_10, k = nu_3+nu_9, l=nu_1+nu_5+nu_7+nu_11, and
% h,i are nu_4, nu_8 (can be distinguished using T_5 but we didn't do it)

\subsubsection{}
We first consider an example with odd weight.
Take $N=3^3$, $k=5$, and $\w'$ the Dirichlet character of modulus 27 and conductor $3$.
We consider simple supercuspidal representations $\sigma_t^\zeta$, where $t\in \{\pm 1\}$ 
and $\zeta^2 = \w'(t)$.
In the \cite{LMFDB} we find the following data for the space $S_5(27,\w')$:
\[        \begin{array}{c|c|c|c|c||c}
            \rule[-3mm]{0mm}{8mm}
     \text{LMFDB label}&\epsilon &\text{dim}&\tr T_4&\tr T_7&(\zeta,t)\\
     \hline
     \hline
            \rule[-3mm]{0mm}{8mm}
27.5.b.a&1&1&16&71&(-i,-1)\\
     \hline
            \rule[-3mm]{0mm}{8mm}
27.5.b.b&-1&2&-76&34&(i,-1)\\
     \hline
            \rule[-3mm]{0mm}{8mm}
27.5.b.c&\pm i&2&14&-38&(1,1)\oplus (-1,1)\\
 \end{array}
\]
The final column, using the shorthand $(\zeta,t)=S_5(\sigma_t^\zeta)$, 
is immediate upon comparing Theorem \ref{N3dim} with the $\epsilon$ and dim columns.
Using Theorem \ref{main} we find the following, which refines the above.

\begin{example}\label{N3T4}
With notation as above,
\[\tr(T_4|S_5(\sigma_t^\zeta))=\frac{37t-23}2+46i\zeta  \cdot \delta_{t=-1},\]
\[\tr(T_7|S_5(\sigma_t^\zeta))=\frac{67-143t}4+\frac{37 i\zeta}2 \delta_{t=-1}.\]
\end{example}

We will give an indication of the proof of the above formulas.  
  The calculations for $\n=7$ are a little bit more interesting, so we start with this case.
By Theorem \ref{mainST}, 
\[\tr(T_7|S_5(\sigma_t^\zeta))=7^{3/2}\Bigl[\Phi(\mat{}{-21}13)+\Phi(\mat{}{-21}16)+\Phi(\mat{}{-21}19)
+\sum_{r=1}^5 \Phi(\mat{}{-7}1r)\Bigr].\]
We have used \eqref{Phiinf} to eliminate the trace zero matrices, since $k$ is odd.
The matrix $\smat{}{-7}13$ is unramified at $p=3$ but has no double characteristic root mod 3.  So its orbital integral
vanishes by Proposition \ref{relevant}.  The first three integrals vanish unless 
\[y^2\equiv -t/7 \equiv -t\mod 3\]
has a solution, i.e., $t=-1$.  In this case, applying Proposition \ref{Phiell} to 
$\g=\smat{}{-21}19$ and $p=3$, we see that $v=3$ so the local integral has the value
$\ol{\zeta_5}(\w_3(1)+\w_3(-1))=0$.   Hence this $\g$ can be discarded.
We compute the remaining orbital integrals locally as summarized in the following table, where
$m=\frac{2h(E)}{w(E)2^{\w(d_E)}}$ is the global measure factor for $E=\Q[\g]$, and
$\ell$ denotes a prime factor of the discriminant $\Delta_\g$ other than $3$ (if such exists).  
The global orbital integral is then $\Phi=m\Phi_\infty\Phi_3\Phi_\ell$.  The factor 
$\Phi_\infty=-\frac{\sin(4\arctan(\frac{\sqrt{|\Delta_\g|}}{\tr\g}))}{\sin(\arctan(\frac{\sqrt{|\Delta_\g|}}{\tr\g}))}$
 was computed using software.
  \[        \begin{array}{c|c|c||c|c|c|c}
            \rule[-3mm]{0mm}{8mm}
     \g&\Delta_\g&\ell&m&\Phi_\infty&\Phi_3&\Phi_\ell\\
     \hline
     \hline
            \rule[-3mm]{0mm}{8mm}
\smat{}{-21}1{3}&-3\cdot 5^2&5&\frac16&11\sqrt3\cdot 7^{-3/2}&-i\ol{\zeta}\sqrt 3\cdot\delta_{t=-1}&7\\
     \hline
            \rule[-3mm]{0mm}{8mm}
\smat{}{-21}1{6}&-2^4\cdot 3&2&\frac16&4\sqrt3\cdot 7^{-3/2}&i\ol{\zeta}\sqrt 3\cdot\delta_{t=-1}&10\\
     \hline
            \rule[-3mm]{0mm}{8mm}
\smat{}{-7}1{1}&-3^3& &\frac16&13\cdot 7^{-3/2}&4& \\
     \hline
            \rule[-3mm]{0mm}{8mm}
\smat{}{-7}1{2}&-2^3\cdot 3&2&\frac12&20\cdot 7^{-3/2}&\frac{1-3t}2&2\\
     \hline
            \rule[-3mm]{0mm}{8mm}
\smat{}{-7}1{4}&-2^2\cdot 3&2 &\frac16&-8\cdot 7^{-3/2}&-\frac{3t+1}2&4\\
     \hline
            \rule[-3mm]{0mm}{8mm}
\smat{}{-7}1{5}&-3& &\frac16&-55\cdot 7^{-3/2}&\frac{3t+1}2& \\
 \end{array}
    \]
The formula for $\tr T_7$ in Example \ref{N3T4} follows upon simplifying. Most of the entries
in the above table are straightforward, but we highlight a few.
 For example,
$\g=\smat{}{-21}16$ is elliptic in $G(\Q_2)$, and by the quadratic formula,
   \[\Z_2[\g]=\Z_2[\tfrac{6+2^2\sqrt{-3}}2]
=\Z_2+\Z_22^2\varepsilon,\]
 where $\varepsilon=\frac{1+\sqrt{-3}}2$.  So $n_\g=2$ and $\Phi_2(\g)=1+(2+1)+(4+2)=10$ by
  Proposition \ref{elle} and \eqref{index}.

The matrix $\g=\smat{}{-7}11$ is unramified at $p=3$, so $\Phi_3(\g)$ is computed using Proposition \ref{gunram}.
We find (using software) that $\mathcal{N}_\g(0,1)=\mathcal{N}_\g(0,2)=3$, $\mathcal{N}_\g(1,2)=6$, 
  $\mathcal{N}_\g(1,3)=9$, and $\mathcal{N}_\g(c,n)=0$ for all other pairs $(c,n)$.  Since $P_\g(X)\equiv (X+1)^2\mod 3$, 
we take $z=-1$, so, using the third remark after Proposition \ref{gunram}, for $t=\pm 1$ we have
\[\Phi_3(\smat{}{-7}11)=\frac{-1}3\Bigl[3\bigl(e(\tfrac t3)+e(\tfrac{-t}3)\bigr)+3(2)+6(-1)+9(-1)\Bigr]=4.\]

Finally, $\g=\smat{}{-7}12$ is unramified at $p=3$ and $\mathcal{N}_\g(-1,1)=3$ is the only nonzero value of
  $\mathcal{N}_\g(c,n)$.  We take $z=1$ in Proposition \ref{gunram} to get
\[\Phi_3(\smat{}{-7}12)=\frac13\cdot 3\Bigl[e(\frac{-1-t}3)+e(\frac{1+t}3)\Bigr]=2\cos(\frac{2\pi(1+t)}3)
=\begin{cases}-1&\text{if }t=1\\2&\text{if }t=-1.\end{cases}\]
This equals $\frac{1-3t}2$ for $t=\pm1$. 
The remaining entries in the above $T_7$ table are found in a similar fashion.

For $\tr T_4$, in the identity term we have $\w'(\sqrt{4})=-1$.  So 
\[\tr(T_4|S_5(\sigma_t^\zeta))=8\Bigl[-\tfrac43+\Phi(\smat{}{-12}13)+\Phi(\smat{}{-12}16)
+\Phi(\smat{}{-4}11)+\Phi(\smat{}{-4}12)+\Phi(\smat{}{-4}13)\Bigr].\]
The last term can be eliminated since it is unramified at $p=3$ and it has no characteristic root modulo $3$.
The remaining orbital integrals are computed locally as follows, and the formula for $\tr T_4$ in
Example \ref{N3T4} follows upon simplification.
  \[        \begin{array}{c|c|c||c|c|c|c}
            \rule[-3mm]{0mm}{8mm}
     \g&\Delta_\g&\ell&m&\Phi_\infty&\Phi_3&\Phi_\ell\\
     \hline
     \hline
            \rule[-3mm]{0mm}{8mm}
\smat{}{-12}1{3}&-3\cdot 13&13&1&5\sqrt3\cdot 8^{-1}&-i\ol{\zeta}\sqrt 3\cdot\delta_{t=-1}&2\\
     \hline
            \rule[-3mm]{0mm}{8mm}
\smat{}{-12}1{6}&-2^2\cdot 3&2&\frac16&-\sqrt3&i\ol{\zeta}\sqrt 3\cdot\delta_{t=-1}&4\\
     \hline
            \rule[-3mm]{0mm}{8mm}
\smat{}{-4}1{1}&-3\cdot 5&5&\frac12&7\cdot 8^{-1}&\frac{3t-1}2&2\\
     \hline
            \rule[-3mm]{0mm}{8mm}
\smat{}{-4}1{2}&-2^2\cdot 3&2&\frac16&1&\frac{3t+1}2&4\\
 \end{array}
\]

\subsubsection{} Let $N=2^311^2$ and $k=6$, and let $\sigma^{\zeta}$ be a simple supercuspidal
representation of $\PGL_2(\Q_2)$ and $\sigma_{\nu}$ a depth zero supercuspidal representation
of $\PGL_2(\Q_{11})$.  Here, $\zeta\in\{\pm 1\}$, and $\nu$ is one of the five primitive
characters of $L^*$ listed in \eqref{B17}, where $L=\F_{11^2}$ and we take the generator
$t$ of $L^*$ to be a root of the polynomial $X^2+7X+2\in \F_{11}[X]$. 
Let $\widehat{\sigma}$ be the associated tuple. Then by Theorem \ref{dimST}, 
\[\dim S_6(\widehat{\sigma})= \frac{25}4 +\frac12\Phi(\mat{}{-2}1{},f^1)+\frac12\Phi(\mat{}{-1}1{},f^1)
+\Phi(\mat{}{-2}12,f^1).\]
Over $\F_{11}$, $X^2+2=(x+3)(x-3)$, so $\smat{}{-2}1{}$ is hyperbolic in $G(\Q_{11})$ by Hensel's
Lemma, and therefore its orbital integral vanishes.  
Using Example \ref{Phi2a} and the argument at \eqref{A1b}, 
\[\frac12\Phi(\mat{}{-1}1{})= \frac12m\Phi_\infty\Phi_2\Phi_{11} = \frac12\cdot\frac14\cdot(-1)^{6/2}\cdot 1\cdot 2\epsilon_{11}
=-\frac{\epsilon_{11}}4.\]

Taking $\g=\smat{}{-2}12$, $P_\g(X)=X^2-2X+2$ is irreducible over $\F_{11}$, so 
by \eqref{Tcase}, 
\[\Phi_{11}=-\ol{\nu(\g)}-\ol{\nu^{11}(\g)}.\]
  For $L^*=\sg{t}$ as 
above, we find (using software) that $t^{51}$ has minimum polynomial $P_\g(X)$.
Therefore, if $\nu=\nu_m$ for $m=10w\in \{10,20,30,40,50\}$ as in \eqref{B17}
where $\mu_m(t)=e(m/120)$, we have
\[\nu_m(\g)=e(\tfrac{51m}{120})=e(\tfrac{17w}4)=e(\tfrac{w}4)=i^w.\]
Using this, $\Phi_{11}(\g)$ is given by
\begin{equation}\label{nusign}
\begin{array}{l||c|c|c|c|c}
            \rule[-2mm]{0mm}{6mm}
\nu&\nu_{10}&\nu_{20}&\nu_{30}&\nu_{40}& \nu_{50}\\
\hline
            \rule[-2mm]{0mm}{6mm}
\epsilon_{11}&+&-&+&-&+\\
\hline
            \rule[-2mm]{0mm}{6mm}
\Phi_{11}&0&2&0&-2&0
\end{array}
\end{equation}
As in the proof of Proposition \ref{2T}, $m=\frac14$, $\Phi_\infty= 1$ (since $k=6$), 
and $\Phi_2=-\zeta$.  Hence 
$\Phi(\g)=-\frac{\zeta\Phi_{11}}4$ for $\Phi_{11}$ as above.
Thus
\begin{equation}\label{dim211}
\dim S_6(\widehat{\sigma})=\frac{25}4-\frac{\epsilon_{11}}4-\frac{\zeta\Phi_{11}}4 
=\begin{cases} 6&\text{if }\epsilon_{11}=1,\text{ or $\zeta=1$ and $\nu=\nu_{20}$},\\
&\text{or $\zeta=-1$ and $\nu=\nu_{40}$};\\
7&\text{if $\zeta=1$ and $\nu=\nu_{40}$},\\
&\text{or $\zeta=-1$ and $\nu=\nu_{20}.$}\end{cases}
\end{equation}

We would like to match the above spaces to Galois orbits of twist-minimal newforms in $S_6^{\new}(2^311^2)$.
In the table below, the first five columns show \cite{LMFDB} data, with AL entries corresponding to the 
Atkin-Lehner signs at $p=2,11$.  These are equal to $\zeta$ and $\epsilon_{11}$ respectively.
The dim column gives the size of the orbit.
  \[        \begin{array}{c|c|c|c|c||c}
            \rule[-3mm]{0mm}{8mm}
 \text{LMFDB label}&\text{dim}&\tr T_7&\text{AL }2&\text{AL }11&(\zeta,\nu)\\
     \hline
     \hline
            \rule[-2mm]{0mm}{6mm}
968.6.a.f & 6 & -124 & - & - & (-1,\nu_{40})\\
     \hline
            \rule[-2mm]{0mm}{6mm}
968.6.a.g & 6 & 124 & + & - & (1,\nu_{20})\\
     \hline
            \rule[-2mm]{0mm}{6mm}
968.6.a.h & 6 & -88 & + & + & (1,\nu_{30})\\
     \hline
            \rule[-2mm]{0mm}{6mm}
968.6.a.i & 6 & 88 & - & + & (-1,\nu_{30})\\
     \hline
            \rule[-2mm]{0mm}{6mm}
968.6.a.j & 7 & -62 & - & - & (-1,\nu_{20})\\
     \hline
            \rule[-2mm]{0mm}{6mm}
968.6.a.k & 7 & 62 & + & - & (1,\nu_{40})\\
     \hline
            \rule[-2mm]{0mm}{6mm}
968.6.a.l & 6 & -206 & + & + & (1,\nu_{10})\oplus (1,\nu_{50})\\
     \hline
            \rule[-2mm]{0mm}{6mm}
968.6.a.m & 6 & 206 & - & + & (-1,\nu_{10})\oplus(-1,\nu_{50})
 \end{array}
\]
  In the final column we have adopted the notation 
  $S_6(\widehat{\sigma})=(\zeta,\nu)$.  This column was
 obtained as follows.  Comparing \eqref{nusign} and \eqref{dim211} with the AL and
dim columns, we immediately infer the entries with $\epsilon_{11}=-1$, i.e. with $\nu_{20}$ and $\nu_{40}$.
   We can distinguish the remaining entries by looking at Hecke eigenvalues.  
For this we apply Theorem \ref{mainST}
to compute $\tr(T_7|S_6(\widehat{\sigma}))$.
The result is the following.

\begin{example}\label{211ex}
Let $N=2^311^2$ and $\widehat{\sigma}=(\sigma^\zeta,\sigma_{\nu})$ be a tuple of supercuspidal
representations of conductors $2^3$ and $11^2$ respectively, as above.  Then
\[\tr(T_7|S_6(\widehat{\sigma}))=-98\zeta\epsilon_{11} -5\zeta X_{11}- 31Y_{11},
\]
where $\epsilon_{11}$, $X_{11}$ and $Y_{11}$ are given as follows:
\[\begin{array}{l||c|c|c|c|c}
            \rule[-2mm]{0mm}{6mm}
\nu&\nu_{10}&\nu_{20}&\nu_{30}&\nu_{40}& \nu_{50}\\
\hline
            \rule[-2mm]{0mm}{6mm}
\epsilon_{11}&+&-&+&-&+\\
\hline
            \rule[-2mm]{0mm}{6mm}
X_{11}&1&1&-2&1&1\\
\hline
            \rule[-2mm]{0mm}{6mm}
Y_{11}&\sqrt3&-1&0&1&-\sqrt3
\end{array}
\]
For example, in the notation used above, 
\[\tr(T_7|(1,\nu_{10}))= -103-31\sqrt3,\qquad
\tr(T_7|(1,\nu_{50}))=-103+31\sqrt3.\]
\end{example}

We sketch the proof as follows.  By Theorem \ref{mainST},
\[\tr(T_7|S_6(\widehat{\sigma}))=7^2\left[\frac12\Phi(\mat{}{-7}1{})+\frac12\Phi(\mat{}{-14}1{})
+\sum_{r=1}^5\Phi(\mat{}{-7}1r)+\sum_{r=1}^3\Phi(\mat{}{-14}1{2r})\right].\]
All but three of the orbital integrals vanish for simple reasons.  The matrices
$\smat{}{-7}1{},\smat{}{-7}12,\smat{}{-7}13$, $\smat{}{-14}11$, and $\smat{}{-14}12$ are 
hyperbolic in $G(\Q_{11})$, since their characteristic polynomials have two distinct roots modulo $11$.
The matrices $\smat{}{-7}11,\smat{}{-7}15$ are unramified at $p=2$ but do not have characteristic
roots modulo $2$.  So the associated orbital integrals vanish by Proposition \ref{relevant}, and
\[\tr(T_7|S_6(\widehat{\sigma}))=7^2\left[\frac12\Phi(\mat{}{-14}1{})
+\Phi(\mat{}{-7}14)+\Phi(\mat{}{-14}1{6})\right].\]
The formula in Example \ref{211ex} follows upon computing each of these terms locally.  The
local results are shown in the following table, with notation as in the previous $N=27$ 
example.  The global orbital integral for a given row is $\Phi=m\Phi_\infty\Phi_2\Phi_{11}\Phi_\ell$.
 \[        \begin{array}{c|c|c||c|c|c|c|c}
            \rule[-3mm]{0mm}{8mm}
     \g&\Delta_\g&\ell&m&\Phi_\infty&\Phi_2&\Phi_{11}&\Phi_\ell\\
     \hline
     \hline
            \rule[-3mm]{0mm}{8mm}
\smat{}{-14}1{}&-2^3\cdot7&7&1&-1&\zeta&2\epsilon_{11}&2\\
     \hline
            \rule[-3mm]{0mm}{8mm}
\smat{}{-14}1{6}&-2^2\cdot 5&5&\frac12&\frac5{7^2}&-\zeta& X_{11} & 2 \\
     \hline
            \rule[-3mm]{0mm}{8mm}
\smat{}{-7}1{4}&-2^2\cdot 3&3&\frac16&\frac{31}{7^2}&-3&Y_{11}&2
 \end{array}
\]
The $\Phi_{11}$ column was determined as follows.
As described earlier, $\F_{11^2}^*=\sg{t}$ where $t^2+7t+2=0$.  For each $\g$ as above, there is a power
$t^j$ whose minimum polynomial over $\F_{11}$ is $P_\g(X)$. 
The power $j$ was found with software, and is given as follows:
\[\begin{array}{l||c|c|c}
            \rule[-2mm]{0mm}{6mm}
\g&\smat{}{-14}1{}&\smat{}{-14}16&\smat{}{-7}14\\
\hline
            \rule[-2mm]{0mm}{6mm}
j&18&8&17
\end{array}
\]
In each case, \eqref{Tcase} implies that 
\[\Phi_{11}=-\ol{\nu(\g)}-\ol{\nu^{11}(\g)}= -\ol{\nu(t^j)}-\ol{\nu(t^{11j})}.\]
By definition, $\nu_m(t)=e(m/120)$, so if $\nu=\nu_m$ for $m=10w$,
\[\Phi_{11}(\g)=-e(-\frac{jw}{12})-e(-\frac{11jw}{12}),\]
which can be evaluated by hand or using software to obtain the $\Phi_{11}$ column in the above table.

\vskip 1cm

\small

%%% The following two lines remove references from the table of contents.
%\let\oldaddcontentsline\addcontentsline% Store \addcontentsline
%\renewcommand{\addcontentsline}[3]{}% Make \addcontentsline a no-op

\let\addcontentsline\oldaddcontentsline% Restore \addcontentsline
\end{document}